\documentclass[12pt]{amsart}   

\usepackage{latexsym, amssymb, amscd, amsfonts, enumerate, tikz-cd}
\usepackage[all,cmtip]{xy}
\usepackage{comment}

\usepackage{bm}
\usepackage[colorlinks=true,urlcolor=black,citecolor=black,linkcolor=black]{hyperref}
\usepackage{ifthen}

\usepackage{hyperref}

\DeclareMathAlphabet{\mathpzc}{OT1}{pzc}{m}{it}

\newboolean{bourbaki}
\setboolean{bourbaki}{false}

\newcommand{\np}{\medskip\noindent}

\newcounter{eqcounter}[section]
\renewcommand{\theeqcounter}{\arabic{section}.\arabic{eqcounter}}
\renewenvironment{equation}{\medskip\noindent\refstepcounter{eqcounter}\makebox[0pt][l]{({\bf\theeqcounter})}
\begin{minipage}[b]{\textwidth}$$}{$$\end{minipage}\medskip\noindent}

\providecommand{\binom}[2]{{#1\choose#2}}

\renewcommand{\geq}{\geqslant}
\renewcommand{\leq}{\leqslant}

\newcommand{\Comp}{{\operatorname{Comp}(\Phi)}}  

\newcommand{\Compf}{{\operatorname{Comp}}} 

\renewcommand{\inf}{\mathrm{inf}}  

\newcommand{\supp}{\operatorname{supp}}

\def\vep{\varepsilon}

\DeclareMathOperator{\Inv}{{Inv}}
\let\Pr\relax  
\DeclareMathOperator{\Pr}{{Pr}}
\DeclareMathOperator{\D}{{D}}

\ifthenelse{\boolean{bourbaki}}
{
\renewcommand{\AA}{\mathbf{A}} 
\newcommand{\CC}{\mathbf{C}} 
\newcommand{\RR}{\mathbf{R}} 
\newcommand{\ZZ}{\mathbf{Z}} 
\newcommand{\FF}{\mathbf{F}} 
}
{
\renewcommand{\AA}{\mathbb{A}} 
\newcommand{\BB}{\mathbb{B}} 
\newcommand{\CC}{\mathbb{C}} 
\newcommand{\DD}{\mathbb{D}} 
 
\newcommand{\EE}{\mathbb{E}} 
\newcommand{\FF}{\mathbb{F}} 
\newcommand{\GG}{\mathbb{G}} 
\newcommand{\RR}{\mathbb{R}} 
\newcommand{\ZZ}{\mathbb{Z}} 
}

\newcommand{\DTail}[2]{^{\displaystyle #1}_{\displaystyle #2}}

\newcommand{\Esix}[6]
{{\begin{array}{@{}c@{}c@{}c@{}c@{}c@{\hskip 1pt}}
 & &#2&&\\[-4pt]
 #1&#3&#4&#5&#6
 \end{array}}}

\newcommand{\Eseven}[7]
{{\begin{array}{@{}c@{}c@{}c@{}c@{}c@{}c@{\hskip 1pt}}
 & &#2&&&\\[-4pt]
 #1&#3&#4&#5&#6&#7
 \end{array}}}

\newcommand{\Eeight}[8]
{{\begin{array}{@{}c@{}c@{}c@{}c@{}c@{}c@{}c@{\hskip 1pt}}
 & &#2&&&&\\[-4pt]
 #1&#3&#4&#5&#6&#7&#8
 \end{array}}}

\newcommand{\DFour}[4]{#1#2\DTail{#3}{#4}}
\newcommand{\DFive}[5]{#1#2#3\DTail{#4}{#5}}

\newtheorem{thm}{Theorem}[section]
\newtheorem{theorem}[thm]{Theorem}
\newtheorem{prop}[thm]{Proposition}

\newtheorem{cor}[thm]{Corollary}
\newtheorem{corollary}[thm]{Corollary}
\newtheorem{lemma}[thm]{Lemma}

\newtheorem{proposition}[thm]{Proposition}

\theoremstyle{definition}
\newtheorem{definition}[thm]{Definition}
\newtheorem{example}[thm]{Example}

\newtheorem{remark}[thm]{Remark}

\newcommand{\rk}{{\mathrm{rk\,}}}

\newcommand{\Span}{\mathrm{span\,}}
\newcommand{\ZSpan}{\mathrm{span}_{\ZZ}}
\newcommand{\Gen}{\mathrm{Gen}}
\newcommand{\GPhi}{G(\Phi)}
\newcommand{\GPsi}{G(\Psi)}
\newcommand{\GX}{G(X)}

\newcommand{\innerProd}[2]{\left\langle #1, #2 \right\rangle}

\usepackage{color}

\newcommand{\R}{\mathbb R}
\newcommand{\beq}{\begin{equation}}
\newcommand{\eeq}{\end{equation}}

\newcommand{\fg}{\mathfrak{g}}
 
\newcommand{\fh}{{\mathfrak{h}}}

\newcommand{\ivan}[1]{\textcolor{red}{Ivan: #1}}
\definecolor{ColePurple}{RGB}{128,51,209}   
\newcommand{\cole}[1]{\textcolor{ColePurple}{Cole: #1}}

\definecolor{DavidGreen}{RGB}{29,162,55}   

\begin{document}

\author[I. Dimitrov]{Ivan Dimitrov}
\address{Department of Mathematics and Statistics, Queen's University}
\email{dimitrov@queensu.ca}

\author[C. Gigliotti]{Cole Gigliotti}
\address{Department of Mathematics, University of British Columbia}
\email{coleg@math.ubc.ca}

\author[E. Ossip] {Etan Ossip}
\address{Mathematical Institute, University of Oxford}
\email{etanossip@gmail.com}

\author[C. Paquette]{Charles Paquette}
\address{Department of Mathematics and Computer Science, Royal Military
College of Canada}
\email{charles.paquette.math@gmail.com}

\author[D. Wehlau]{David Wehlau}
\address{Department of Mathematics and Computer Science, Royal Military
College of Canada}
\email{wehlau@rmc.ca}

\pagestyle{plain} 
\title{\large{Inversion Sets and Quotient Root Systems}}

\subjclass[2020]{Primary 17B22; Secondary 17B20, 17B25, 22F30.}

\begin{abstract}
The main result of this paper is a recursive description of all decompositions 
\[ 
\Delta^+ = \Phi_1 \sqcup \Phi_2 \sqcup \dots \sqcup \Phi_k
\]
of the positive roots $\Delta^+$ of an arbitrary root system $\Delta$ into a disjoint union of inversion sets. 
Such decompositions play a central role in geometric invariant theory (GIT) in connection with studying 
the Littlewood-Richardson cone and related problems. This work can be considered as a continuation of
\cite{DDMRWW} in which similar questions were studied for root systems of type $\mathbb{A}$. The methods
of \cite{DDMRWW} rely on properties of permutations and are not transferable to an arbitrary root system.

\np
In order to develop a type-independent approach, we go beyond root systems and consider quotient 
root systems (QRSs for short). We study subsets of positive roots in an arbitrary QRS $R$.
We prove that every $\Phi \subseteq R^+$ can be represented in a canonical way as an inflation and develop 
methods to study recursively properties of such subsets. We extend the notion of an inversion 
to subsets of any QRS, i.e., beyond the case when a Weyl group is associated with $R$. If $\Phi \subseteq R^+$
is an inversion set, we introduce a graph $\GPhi$ and endow the set $\Comp$ of connected components of $\GPhi$
with a partial addition. The resulting monoid-like structure $(\Comp, +)$ is a further generalization of root 
systems beyond QRSs. We study in detail the properties of $(\Comp, +)$ and their applications to studying 
the properties of $\Phi$. In particular we investigate the relationship between 
$\Phi$ being primitive and $\Phi$ being irreducible. Apart from describing recursively all
decompositions of $\Delta^+$ into the disjoint union of inversion sets, we provide
applications to geometric invariant theory and 
derive enumerative results which may be of independent interest.

\np
Keywords: Root system, Quotient root system, Inversion set, Biconvex set.
\end{abstract}
\maketitle
\setcounter{tocdepth}{1}   
\tableofcontents

\section{Introduction} \label{sec: intro}

\np 
Let $\Delta$ be a root system with Weyl group $W$ and set of positive roots $\Delta^+$. The inversion
set of $w \in W$ is the set 
\[\Phi(w):= \{ \alpha \in \Delta^+ \, | \, w(\alpha) \in \Delta^- \} = \Delta^+ \cap w^{-1} (\Delta^-) \ ,\]
where $\Delta^- = -\Delta^+$. If $\Delta$ is of type $\AA_n$, then $\Delta^+$ can be identified with the set
\[ \{(i,j) \in \ZZ \times \ZZ \, | \, 1 \leq i < j \leq n+1 \} \ ,\]
$W = S_{n+1}$ is the symmetric group on $n+1$ elements, and the inversion set of $\sigma \in S_{n+1}$
is the set
\[\Phi(\sigma) = \{(i,j) \in \Delta^+ \, |\, \sigma(i) > \sigma(j) \} \ . \]
Inversion sets we first introduced by Kostant, \cite{K1}, in the context of Lie algebra cohomology.
Because of their relationship to cohomologies of homogeneous varieties, inversion sets 
have recently arisen in various problems in geometric invariant theory (GIT), e.g. when 
describing faces of the Littlewood-Richardson cone or studying the Belkale-Kumar product on full flag 
varieties, see for example \cite{BK} and \cite{DR}.

\np
Describing the decompositions of $\Delta^+$ as a disjoint union of inversion sets arises
in several of the problems mentioned above, see \cite{DDMRWW} and the references therein for details.
In \cite{DDMRWW} a complete description of such decompositions was obtained in terms of inflations of permutations.
The results for type $\AA$ are also carried over to types $\BB$ and $\CC$ by exploiting a realization of each
of the latter root systems in terms of root systems of type $\AA$ with additional symmetries. Unfortunately, the methods
developed there do not apply to root systems of type $\DD$ and to exceptional root systems. 

\np
The motivation and starting 
point for this paper was to provide a uniform approach to studying decompositions of $\Delta^+$  as
into a disjoint union of inversion sets, i.e., as

\begin{equation} \label{eq:1.11}
\Delta^+ = \Phi(w_1) \sqcup \Phi(w_2) \sqcup \dots \sqcup \Phi(w_k) \ .
\end{equation}
We develop machinery which allows us to provide in Theorem~\ref{theorem: main theorem} a recursive 
description of all decompositions \eqref{eq:1.11}. This description is then used to prove some statements
which have appeared recently in works related to GIT. 
In \cite{FR} a property of inversion sets
$\Phi(w_1), \Phi(w_2), \Phi(w_3)$ that satisfy \eqref{eq:1.11} for $k = 3$
is central for the main result. In the first
version of the paper the authors provided a very long case-by-case proof of that property. In the latest
version of \cite{FR} a much simpler proof based on the first version of this paper is provided. As an illustration of
the methods and results we develop, we provide two short proofs of the same result, see 
Proposition~\ref{prop:Ressayre}. In \cite{HP} the authors discuss the conjecture that the sum of the numbers of right
descents of elements $w_1, w_2, \dots, w_3$ of the Weyl group $W$ of $\Delta$ 
which satisfy \eqref{eq:1.11} equals the rank of $\Delta$. They provide a long
proof in the case when $\Delta$ is a root system of type $\mathbb{A}$. Propostion~\ref{prop: basis from decomposition}
establishes the same result for all finite Weyl groups. (In the case when $\Delta$ is of type
$\mathbb{A}$, Propostion~\ref{prop: basis from decomposition} was proved in \cite{DDMRWW}.)

\np
In order to extend the results of \cite{DDMRWW}, we introduce the notion of inflation of
inversion sets which generalizes inflations of permutations to inflations of arbitrary subsets of
positive roots in a (quotient) root system. 
The operation of taking an inflation of permutations, 
which plays a central role in \cite{DDMRWW}, is nothing but the composition
in the operad whose $n$-ary operations are the permutations of $n$ elements. This operad structure yields
a canonical expression of every permutation as a composition (in the sense of operads) of permutations. In
particular, one can introduce the notion of a primitive permutation --- a permutation which cannot be represented as a
non-trivial composition (in \cite{DDMRWW}  the term ``simple permutation'' is used instead of primitive permutation).
It turns out that the property of a permutation being primitive is closely related (though not identical) to
the corresponding inversion set being irreducible, i.e., one that does not decompose into the disjoint union
of two non-empty inversion sets. Following this point of view, it is natural to look for an operad structure 
on root systems beyond these of type $\mathbb{A}$ and to explore the relationship between compositions in the operad and 
decompositions of inversion sets. The first major obstacle to realizing the idea outlined above is that the 
class of root systems is not large enough to define the corresponding operad. It turns out that the natural extension
to the class of quotient root systems (QRSs for short) allows us to extend the notion of inflation 
to subsets of arbitrary QRSs. In this paper we do not use the language of operads but the interested reader will
be able to fill in the details.

\np 
Roughly speaking, if $\Delta$ is a root system with 
a base $\Sigma$ and $I$ is a subset of $\Sigma$, the QRS $\Delta/I$ is the image of $\Delta$
under the natural projection $\pi_I: \Span \Delta \to \Span \Delta/ \Span I$. QRSs have been studied extensively
in connection with Lie theory, \cite{K}, and simplicial hyperplane arrangements, \cite{Cu} and the references therein, 
to mention a few. The notions of positive roots, bases, etc.
extend to QRSs and the counterparts of Weyl groups are certain groupoids called Weyl groupoids.
Quotients of QRSs are also well-defined and taking quotients of QRSs is functorial. In particular,
for any QRS $R$ and any subset $I$ of a given base $S$ of $R$ we have the quotient $R/I$, the
subsystem $R_I$, and the surjection $\pi_I: R \to R/I$. 

\np
Let $\Psi \subseteq (R/I)^+$ and $X \subseteq R_I^+$, where $(R/I)^+$ and $R_I^+$ denote the positive roots
of $R/I$ and $R_I$, respectively, determined by a fixed basis $S$ of $R$. Define (see Definition~\ref{def: inflation})
the {\it inflation of $\Psi$ by $X$} as the subset 

\begin{equation} \label{eq:inflation}
\inf_I^S(\Psi, X) = \{\alpha \in R^+\backslash R_I \, | \, \pi_I(\alpha) \in \Psi\} \cup \{\alpha \in R \, |\, \alpha \in X\} 
\subseteq R^+ \ .
\end{equation}
Every $\Phi \subseteq R^+$ can be expressed as an inflation with $I = \emptyset$ and $I = S$ as follows:
$\Phi = \inf_\emptyset^S(\Phi, \emptyset) = \inf_S^S(\emptyset, \Phi)$.
We call these two expressions {\it trivial inflations}. A subset $\Phi \subseteq R^+$ is {\it primitive}
if it cannot be represented as a non-trivial inflation. In Theorem~\ref{theorem: canonical inflation} 
we prove that every $\Phi \subseteq R^+$ can be expressed uniquely
as $\Phi = \inf_I^S(\Psi, X)$, where $\emptyset \subseteq I \subsetneq S$ and 
\begin{enumerate}
\item[(i)] $\Psi$ is primitive
\item[] or
\item[(ii)] $\Psi = \emptyset$ or $\Psi = (R/I)^+$ and $I$ is minimal with this property. 
\end{enumerate}

\np 
Let $R$ be a QRS with positive roots $R^+$. By definition, a subset $\Phi \subseteq  R^+$ is an inversion set 
in $R$ if $\Phi$ is both closed and co-closed. If $R=\Delta$ is a root system, $\Phi \subseteq R^+$ is an inversion set
if and only if it is the inversion set of an element $w$ of the Weyl group $W$ of $\Delta$. 
Inflation preserves inversion sets: $\Phi = \inf_I^S(\Psi, X)$ is an inversion set in $R$ if and only if 
$\Psi$ and $X$ are inversion sets in $R/I$ and $R_I$ respectively. 

\np
In order to study various properties and attributes of inversion sets like their supports, canonical form
\eqref{eq:inflation}, decompositions into a disjoint union of inversion sets, etc., we introduce the graph 
$\GPhi$ associated with an inversion set $\Phi \subseteq R^+$. The vertices of $\GPhi$ are the elements of $\Phi$.
Two vertices $\alpha, \beta \in \Phi$ are connected with an edge if and only if $\alpha - \beta \in \pm \Phi^c$.
Let $\Comp$ denote the set of connected components of $\GPhi$. It turns out that the set $\Comp$ inherits
an order and a partial addition from $R^+$. Remarkably, the order and the partial addition enjoy many of the 
properties of the order and the addition on $R^+$. Since $\Comp = R^+$ when $\Phi = R^+$, we can consider
$\Comp$ as a further extension of the notion of a root system beyond QRSs. Despite the similarities, there 
are some important distinctions between $R^+$ and $\Comp$. The first one is that the addition is only partially
defined. This is actually not too surprising --- the addition in $R^+$ is also only partially defined. However,
in the case of $\Comp$, we do not know of a natural lattice that contains $\Comp$ and such that the partial addition
on $\Comp$ is inherited from the addition on the lattice. 
The second (and related)
problem is that cancellation rules do not always hold in $\Comp$: we may have $A+B = A+C$ without $B = C$
or even $A+A = A$ for $A \neq 0$. The core of the paper is devoted to studying the properties of addition in $\Comp$
and their applications to representing $\Phi$ as an inflation and to decomposing $\Phi$ into a disjoint union of
inversion sets. In particular, we establish the relationship between two important properties of an inversion set:
being primitive and being irreducible. 

\np
We complete the paper with two sections on applications of the machinery developed earlier. The first 
set of applications is to three problems arising in geometric representation theory --- we are
able to prove statements that have arisen in other works and up until now have been beyond reach. 
The second set of applications is to derive some enumerative results which may be of interest to combinatorialists.
For example, we arrive at an analog of Catalan numbers arising from roots systems of types $\mathbb{D}$. 

\np
We expect that the methods developed in this paper may find further applications. For example, 
the notion of inflation is not limited to inversion sets and it is reasonable to ask what other 
classes of subsets of positive roots will yield interesting combinatorial and/or algebraic structures similar
to $\Comp$. Another interesting direction is to understand whether the partial addition on $\Comp$ can 
be understood as the restriction of an additive structure on an appropriate lattice.

\np 
Here is a brief overview of the contents of the paper.  

\np 
Sections~\ref{sec: prelim} and \ref{sec: paths}
provide background and preliminary material. In Sections~\ref{sec: prelim} we introduce QRSs and inversion sets,
discuss some of their properties and list references for further information on these topics.
In Section~\ref{sec: paths} we study sums of roots in a QRS focusing on 
how roots are connected by paths $[\alpha; \kappa_1, \dots, \kappa_k; \beta]$, i.e., 
sequences of roots $\alpha, \beta, \kappa_1, \dots, \kappa_k$ in a QRS $R$ such that 
$\beta = \alpha + \kappa_1 +\dots + \kappa_n$ and $\alpha + \kappa_1 +\dots + \kappa_j \in R$
for every $1 \leq j \leq k$. We believe that most of the results in this section are known for 
root systems but we do not know of good references, so we provide short proofs in the generality of QRSs.

\np
In Section~\ref{sec: inflations} we introduce the notion of inflation, see \eqref{eq:inflation}.
Proposition~\ref{prop: inflation of an inflation is an inflation 1} establishes that inflation is 
functorial. 
Theorem~\ref{theorem: canonical inflation} proves that every $\Phi \subseteq R^+$ can be represented
canonically as an inflation. In Proposition~\ref{prop: inflations and inversions} we show that 
taking inflations preserves inversion sets, i.e., if $\Phi = \inf_I^S(\Psi, X)$ then $\Phi$
is an inversion set if and only if both $\Psi$ and $X$ are inversion sets. We also discuss 
the set $\Gen(\Phi)$ of generating sets of $\Phi$ defined as
\[
\Gen(\Phi) = \{ I \subseteq S \, | \, \Phi = \inf_I^S(\Psi, X) {\text { for some }} \Psi {\text { and }} X\} \ ,
\]
see Propositions~\ref{prop: Gen(Phi) for primitive Psi} and \ref{prop: Gen(Phi) from hyperplane}.

\np
Sections~\ref{sec: graphs} -- \ref{sec:additive structure through} are the core of the paper: for
any inversion set $\Phi \subseteq R^+$ we introduce the graph $\GPhi$, endow the set of its components $\Comp$
with a partial order and a partial addition and study the resulting structures. 
In Section~\ref{sec: graphs} we introduce the graph $\GPhi$ and study the partial order on
$\Comp$ inherited from $R^+$, see Propositions~\ref{prop: partial order equivalences} and 
\ref{prop:partial order on components is an order}.
In Section~\ref{sec: addition} we prove that $\Comp$ inherits a partial addition from $R^+$. For brevity we will refer to this
operation simply as ``addition'', remembering that for $A, B \in \Comp$, $A+B$ is not necessarily defined. 
Since in the case when $\Phi = R^+$,
$\Comp = R^+$ and the additions on the two sets are the same, one may consider the set $\Comp$
as a further generalization of QRSs. The addition is commutative by definition and in 
Proposition~\ref{prop: component 2 of 3 rule} we prove that it is associative whenever the respective 
expressions are all defined. We also introduce the notion of a simple component and prove that
every element of $\Comp$ is a sum of simple components; see Propositions~\ref{Simple Comp. equiv.} and
\ref{prop: all components are standard sums of simple components}. 
Proposition~\ref{prop: first cancellation rule for component addition} provides a sufficient condition for
the cancellation rule to be true and Proposition~\ref{prop: Repeated addition satisfies standard properties}
discusses the ``multiples'' $kA$ of a component $A$. 
In Section~\ref{sec: components as inflations} we discuss the relationship between the properties of an
inversion set being primitive and being irreducible. In Proposition~\ref{prop: sums of inflations is an inflation}
we show that if $A, B, A+B \in \Comp$, then $\Gen(A) \cap \Gen(B) \subseteq \Gen(A+B)$ and in 
Proposition~\ref{prop: components are inflated by the canonical} we show that if $\Phi$
is canonically inflated from $I$, then every element of $\Comp$ is inflated from $I$. The main results
of this part of the paper, Theorem~\ref{theorem: irreducibility equivalent conditions},
establishes conditions for $\Phi$ to be irreducible. As a consequence, we prove in
Corollary~\ref{cor: primitive implies irreducible}
that every primitive $\Phi$ is irreducible and, conversely, if both $\Phi$ and $\Phi^c$ are
irreducible, then $\Phi$ is primitive. In the final 
Section~\ref{sec:additive structure through}
of this part of the paper we show that the addition on $\Comp$ is completely determined by the additions on
$\Compf(\Psi)$ and $\Compf(X)$ where $\Phi=\inf_I^S(\Psi,X)$ is the canonical form of $\Phi$; see 
Propositions~\ref{prop: pi_I is surjective on components and bijective when I is canonical} and 
\ref{addition within or without X}. Using this we also prove that addition in $\Comp$, whenever defined,
is independent of the order of summands and their parenthesization; see
Proposition~\ref{prop: addition is independent of bracketing}.

\np 
In Sections~\ref{sec: GIT} and \ref{sec: enumerate} we present some applications of our results. 
In Section~\ref{sec: GIT} these are applications to problems arising in GIT: Theorem~\ref{theorem: main theorem}
provides a recursive description of all decomposition like \eqref{eq:1.11}, Proposition~\ref{prop: basis from decomposition}
proves a conjecture on right descents of elements $w_1, \dots, w_k \in W$ that satisfy \eqref{eq:1.11},
and Proposition~\ref{prop:Ressayre} proves a statement from \cite{FR}.
In Section~\ref{sec: enumerate} we derive some enumerative results:
we evaluate the number of fine decompositions of $\Delta^+$ into
inversion sets for any root system $\Delta$, 
provide a recursive formula for the number of primitive inversion sets, and 
calculate the number of fine inversion sets. 
Note that, for a root system of type $\AA_n$, the number of fine decompositions of 
$\Delta^+$ equals the $n$-th Catalan numbers, so the numbers of fine decompositions of $\Delta^+$
for other root systems may be considered as analogs and extensions of Catalan numbers.

\section{Preliminaries} \label{sec: prelim}

\subsection{Quotient root systems} Quotient root systems (QRSs for short), originally introduced by Kostant, provide the
natural setting in which inversion sets of roots can be studied recursively. Below we provide the
necessary notation and background on QRSs; for more details, we refer the reader to \cite{K} and \cite{DF}.

\np 
Recall, \cite{B} and \cite{H}, that a {\it root system} is a finite set $\Delta$ of nonzero vectors in a Euclidean space 
$E$ with inner product $\langle \cdot, \cdot \rangle$ that satisfies the following properties:

\begin{enumerate}
    \item[(i)] $\Delta$ spans $E$;
    \item[(ii)] If $\alpha \in \Delta$ and $r \in \RR$ satisfy $r\alpha \in \Delta$, then $r = \pm 1$;
    \item[(iii)] If $\alpha, \beta \in \Delta$, then $\sigma_\alpha(\beta) \in \Delta$, where $\sigma_\alpha$
    denotes the reflection about $\alpha$;
    \item[(iv)] For $\alpha, \beta \in \Delta$, 
    $2\, \frac{\langle \alpha, \beta \rangle}{\langle \alpha, \alpha \rangle} \in \ZZ$. 
\end{enumerate}
The elements of $\Delta$ are called {\it roots}.

\np 
\begin{remark} \label{rem:zero_is_a_root}
It is convenient to modify slightly the definition above by requiring that the zero vector belongs to $\Delta$.
For the rest of the paper we will work with the assumption that all root systems (and quotient root systems, see below)
contain the zero vector. This convention requires obvious minor adjustments in the definitions and statements about root systems. 
For example, in conditions (ii) -- (iv) of the definition of a root system, we need to assume that $\alpha \neq 0$.
Note that, with this convention, $\Delta = \{0\}$ is a root system of rank zero.
\end{remark}

\np 
A subset $\Sigma \subseteq \Delta$ is a {\it base} of $\Delta$ if it is a basis of $E$ and every nonzero element of $\Delta$
is an integer combination of elements of $\Sigma$ with non-negative or non-positive coefficients only. A base $\Sigma$
yields the partition $\Delta = \Delta^+ \sqcup\{0\} \sqcup \Delta^-$ of $\Delta$ into {\it positive roots}, the
zero root, and {\it negative roots}. In other words, 
the positive and negative roots are the nonzero roots which are expressed respectively as 
non-negative and non-positive combinations of elements of $\Sigma$.

\np 
Given a subset $I \subseteq  \Sigma$, we introduce the following notation:
\begin{enumerate}
\item[--] $E_I$ and $E/I$ denote respectively the span of $I$ and its orthogonal complement in $E$;
\item[--] $\pi_i: E \to E/I$ is the projection defined by the decomposition $E = E_I \oplus E/I$;
\item[--] $\Delta_I: = \Delta \cap E_I$, $\Delta/I := \pi_I(\Delta) \subseteq E/I$.
\end{enumerate}
Clearly, $\Delta_I \subseteq E_I$ is a root system with a base $\Sigma_I:= I$ and positive roots 
$\Delta_I^+ = \Delta_I \cap \Delta^+$ with respect to $\Sigma_I$. 

\np
A {\it quotient root system} (or QRS for short) is a set of the form $R=\Delta/I$ for some root system $\Delta$ and
some subset $I$ of a base of $\Delta$. We call the elements of a QRS \emph{roots}.
Note that the quotient $\Delta/I$ depends on $I$ but not on $\Sigma$ as long as $I$ is a subset of a base of $\Delta$.
The \emph{rank} of $R$, denoted $\mathrm{rk} \,R$, is the dimension of the ambient vector space $E/I$.
Furthermore, a given QRS $R$ may be the quotient of different root systems, more precisely, 
quotients of two different root systems may be isomorphic, 
cf.~\cite{DF}. 
The subset $\Delta/I$ of $E/I$ is the \emph{quotient of $\Delta$ with respect to $I$}. 
Finally, $\Delta/\emptyset = \Delta$ and $\Delta/\Sigma = \{0\}$.
In particular, every root system is a QRS.

\np 
QRSs exhibit many properties similar to properties of root systems but also 
differ in some aspects, e.g. QRSs allow for multiples of roots which are also roots, for example, a QRS
may contain both $\alpha$ and $2\alpha$. For a detailed discussion on the properties of QRSs, see \cite{K} or \cite{DF}.

\np 
As for root systems, a \emph{base} $S$ of a QRS $R$
is a subset of $R$ that forms a basis of the underlying vector space $E/I$, 
and such that every root can be written 
as a non-negative or non-positive integral combination of elements of $S$. A base $S$ of $R$
determines a set of \emph{positive} roots $R^+$, namely the nonzero elements of $R$ which are
non-negative integral combinations of elements of $S$; the \emph{negative} roots $R^-$ are the nonzero elements of $R$ which are
non-positive integral combinations of elements of $S$. 
As in the case of root systems, we declare $0 \in R$ to be neither positive nor negative and 
we have the disjoint union decomposition $R = R^+ \sqcup \{0\} \sqcup R^-$.
A {\it positive system} $P$ in $R$ is a subset which is closed under addition, i.e., $(P+P) \cap R \subseteq P$ and
$R = P \sqcup \{0\} \sqcup (-P)$. Equivalently, $P \subseteq R$ is a positive system if and only 
if $P$ consists of all elements of $R$ in an open half-subspace of $E$ whose boundary hyperplane
does not contain any nonzero roots.
Positive systems in $R$ are in a bijection with bases of $R$,
see \cite{DF} for a detailed discussion of positive systems in $R$ and bases of $R$.

\np 
It is easy to see that, given a base $\Sigma$ of $\Delta$ and $I \subseteq \Sigma$, 
\[S:=\Sigma/I := \{ \pi_I(\beta) \, | \, \beta \in \Sigma \backslash I\}\]
is a base of $R$. Moreover,
$R^+ = \pi_I(\Delta^+) \backslash \{0\}$ and $R^- = \pi_I(\Delta^-) \backslash \{0\}$. 

\np 
Throughout the paper we will use $R$ and $S\subseteq R$ to denote respectively a QRS $R$ and a base $S$ of $R$,
while we will reserve
$\Delta$ and $\Sigma$ for root systems. 
If $R, S, \Delta$, and $\Sigma$ are used in the same discussion, $R$ is assumed to be the quotient of $\Delta$ with
respect to some $I \subseteq \Sigma$ and $S = \Sigma/I$.

\np 
If $R$ is a QRS with a base $S$ and $I\subseteq S$, we can define the 
quotient $R/I$ of $R$ with respect to $I$ in the same way we defined quotients of root systems, 
Taking quotients of QRSs is functorial:  
If $I \subseteq J \subseteq S$, then
\[(R/I)/(J/I) \cong R/J \ ,\]
where $J/I \subseteq S/I$ in the QRS $R/I$. 
Similarly, we can define a subsystem $R_I:=R \cap \Span I$ of $R$. Clearly, $R_I$ also is a QRS
and taking subsystems is functorial.

\np 
Let $R$ be a QRS in a Euclidean space $E$ and let $E$ be decomposed orthogonally
as $E = E^1 \oplus \dots \oplus E^s$. We say that $R = R^1 \times \dots \times R^s$
if $R = R^1 \cup \dots \cup R^s$, where $R^i := E^i \cap R$. 
If $R = R^1 \times \dots \times R^s$ and $S^i$ is a base of $R^i$, 
then $S = S^1 \cup \dots \cup S^s$ is a base of $R$. Moreover, every base of $R$ is of this form.
We say that $R$ is \emph{irreducible} if it cannot be written as $R = R' \times R''$ for two
nontrivial subsystems $R'$ and $R''$. Every nontrivial QRS decomposes uniquely
as $R = R^1 \times \dots \times R^s$, where each $R^i$ is irreducible. 
Since the decomposition of a QRS into irreducible components interplays well with most properties,
throughout the paper we will implicitly assume that the QRSs we consider are irreducible,  unless otherwise stated. Whenever necessary,
we will comment explicitly on the differences that arise for reducible QRSs.

\np 
The irreducible root systems
fall into the infinite series $\AA_l, \BB_l, \CC_l, \DD_l$ and the exceptional ones $\EE_6, \EE_7, \EE_8, \FF_4, \GG_2$.
The irreducible QRSs are described in \cite{DF}. 
If $R$ is an irreducible QRS, then as observed below, any quotient of $R$ is itself an irreducible QRS.

\np 
\begin{definition} \label{Order on roots}
    Let $R$ be a QRS with base $S = \{\theta_1, \dots, \theta_r\}$. We define a partial order on the roots by declaring
for two roots $\alpha = \sum_{i=1}^r a_i \theta_i$ and
$\beta = \sum_{i=1}^r b_i \theta_i$, that 
$\alpha \leq \beta$ if $a_i \leq b_i$ for $1 \leq i \leq r$.
\end{definition}

\np 
\begin{remark} \label{Remark:coefficient vector}
    If $\Delta$ is a root system, we can display a root in $\Delta$ by labeling the nodes of the 
    Coxeter-Dynkin diagram with its coefficients in the given base.
For instance, for the root system $\mathbb{D}_5$, every root can be represented in the form
\[\DFive {a \ }  {b \ }  {c \ } d e \ \]
and the highest root of $\mathbb{E}_6$ is represented as
\[\Esix 1 2 2 3  2 1 \ .\]
\end{remark}
\np 
\begin{definition}
    The \emph{support} of a root $\alpha = \sum_{i=1}^\ell a_i\theta_i$, with respect to a given base $S$, is 
\[\supp(\alpha) = \{\theta_i \in S \mid a_i \ne 0\}.\]
If $E$ is a set of roots, then we define its \emph{support} $\supp(E)$ to be the union of the supports of the roots in $E$.
\end{definition}

\np 
As in the case of root systems, if $R$ is a QRS and $\alpha \in R$, then 
$\supp(\alpha)$ is the base of an irreducible subsystem. Moreover, $R$ is 
irreducible if and only if there is a unique maximal root in $R$; in that case
its support is the corresponding base of $R$. In particular, the quotient of an
irreducible QRS is itself irreducible.

\np
We finish this subsection by recording two results that will be useful in the sequel.

\np 
\begin{prop}
\label{prop: if <a, b> < 0 then a + b is a root}
    Let $\alpha, \beta \in R$. If $\innerProd{\alpha}{\beta} < 0$, then $\alpha + \beta \in R$. 
    If $\innerProd{\alpha}{\beta} > 0$, then $\alpha - \beta \in R$.
\end{prop}

\begin{proof}
    This is part of Theorem~2.3 in \cite{K}.
\end{proof}

\begin{prop}
    \label{prop: lifts of sums in quotients}
    Let $I \subseteq S$ and $\bar \alpha, \bar \beta, \bar\gamma$ be non-zero elements  in $R / I $ such that
    $\bar\gamma = \bar\alpha + \bar\beta$. Then
    \begin{enumerate}
\item[(i)] If $\gamma \in \pi_I^{-1}(\bar\gamma)$, then
there are $\alpha \in \pi_I^{-1}(\bar\alpha)$ and
$\beta \in \pi_I^{-1}(\bar\beta)$ with $\gamma = \alpha + \beta$.
\item[(ii)] If $\alpha \in \pi_I^{-1}(\bar\alpha)$, then
there are
$\beta \in \pi_I^{-1}(\bar\beta)$ and
$\gamma \in \pi_I^{-1}(\bar\gamma)$ 
with $\gamma = \alpha + \beta$.
    \end{enumerate}
\end{prop}

\begin{proof} Statement (i) follows from \cite[Theorems 1.9 and 2.3]{K} and statement (ii) follows from (i) by considering the 
elements $\bar \alpha = \bar \gamma + (-\bar \beta)$.
\end{proof}

\subsection{Inversion sets}
In this subsection, we fix a QRS $R$  with base $S$. For a subset $\Phi$ of $R^+$, we denote by $\Phi^c$ its complement in $R^+$. 

\begin{definition} Let $\Phi \subseteq R^+$. We say that
\begin{enumerate}
    \item $\Phi$ is \emph{closed} if $\alpha + \beta \in \Phi$ whenever $\alpha, \beta \in \Phi$.
    \item $\Phi$ is \emph{co-closed} if $\Phi^c$ is closed, i.e., if 
    $\alpha + \beta \in \Phi^c$ whenever $\alpha, \beta \in \Phi^c$. 
    \item $\Phi$ is an \emph{inversion set} if it is both closed and co-closed. 
\end{enumerate}
\end{definition}  

\begin{remark}  It is easy to see that
    intersections of closed sets are closed and unions of co-closed sets are co-closed. 
\end{remark}

\np 
For roots systems, the notion of inversion set takes its origin from the fact that $\Phi$ is an inversion set precisely 
when there is an element $w$ in the Weyl group of the root system such that
\[\Phi = \{\alpha \in R^+ \mid w(\alpha) \in R^-\}.\]

\begin{remark} \label{rem: inversion Sets lie between two hyperplanes}
    A set $\Phi \subseteq  R^+$ is an inversion set if and only if 
    $\Phi = R^+ \cap P$
    for some positive system $P \subseteq R$. Moreover, the assignment $P \mapsto R^+ \cap P$ is a bijection
    between the positive systems in $R$ and inversion sets in $R^+$.
\end{remark}

\np 
As we will be interested in decomposing inversion sets, we need the following two definitions.

\begin{definition}
    Let $\Phi \subseteq R^+$ be a non-empty inversion set. We say that $\Phi$ is \emph{irreducible} if, whenever
    $\Phi = \Phi_1 \sqcup \Phi_2$ as a disjoint union of inversion sets, then $\Phi_1 = \Phi$ or $\Phi_2 = \Phi$.
\end{definition}

\np

\begin{definition}
    Let $\Phi$ be an inversion set in $R^+$. A \emph{decomposition} of $\Phi$ is an expression
    \[\Phi = \Phi_1 \sqcup \dots \sqcup \Phi_r\]
    where all $\Phi_i$ are pairwise disjoint inversion sets. Such a decomposition is \emph{fine} if each $\Phi_i$ 
    contains exactly one simple root. In particular, the number of inversion sets in a fine decomposition of $R^+$
    equals the rank of $R$.
\end{definition}

\begin{prop}
    Let $\Phi = \Phi_1 \sqcup \dots \sqcup \Phi_r$ be a decomposition of an inversion set $\Phi$. 
    Then for $\emptyset \ne I \subseteq \{1,\dots,r\}$, the union $\cup_{i \in I}\Phi_i$ is an inversion set.
\end{prop}

\begin{proof}
    It suffices to prove that $\Phi_1 \sqcup \Phi_2$ is closed and co-closed. It is clear that $\Phi_1 \sqcup \Phi_2$ 
    is co-closed as the union of co-closed sets.
    Further, $\Phi_1 \sqcup \Phi_2$ is closed as the complement of $\Phi_3 \sqcup \dots \sqcup \Phi_r \sqcup \Phi^c$, 
    which is co-closed.
\end{proof}

\section{Paths in QRSs} \label{sec: paths}

\np
One difficulty when working with QRSs, including root systems, is that sums of roots are not in general 
roots. In this section we study the following question: Given roots $\alpha$ and $\beta$ in a QRS $R$ 
and a set $K \subseteq R$ with 
$\beta - \alpha \in \ZSpan K$, is there a way to add roots in $K$ to get from $\alpha$ to $\beta$ without 
leaving $R$? 
Through Sections~\ref{subsec: 2 of 3 rule} and \ref{subsec: paths in R}, we 
build toward an affirmative answer to this question.
Section~\ref{subsec: 2 of 3 rule} introduces an identity for roots which will be used throughout the paper.
Section~\ref{subsec: paths in R} then introduces paths, and Proposition~\ref{proposition: existence of paths} 
guarantees their existence.
In Section~\ref{subsec: reduced paths} 
we consider how freely one may choose roots from $K$, and define an especially flexible class of paths.

\np In this section, although the focus is still on irreducible QRSs, the results we present 
apply to general QRSs. Remark~\ref{rem: paths cross components only at zero} below explains how paths 
in reducible QRSs are built from paths in irreducible components.

\subsection{Two-out-of-three rule} 
\label{subsec: 2 of 3 rule}
We start by establishing the ``two-out-of-three'' rule for sums of roots.

\begin{proposition}
\label{proposition: 2 of 3 rule}
    Let $\alpha, \beta, \gamma \in R$ be such that $\alpha + \beta + \gamma \in R$ 
    but $\beta + \gamma \notin R$. If $\alpha + \gamma \neq 0$, then $\alpha + \beta \in R$;
    if $\alpha + \beta \neq 0$, then $\alpha + \gamma \in R$.
    In particular, if $\alpha, \beta, \gamma \in R$, $\alpha + \beta + \gamma \in R$,
    and none of $\alpha + \beta, \alpha + \gamma, \beta + \gamma$ equals zero, then at least
    two of $\alpha + \beta, \alpha + \gamma, \beta + \gamma$ belong to $R$.
\end{proposition}

\begin{proof}
    We repeatedly use Proposition~\ref{prop: if <a, b> < 0 then a + b is a root}.
    We assume that $\alpha + \gamma \neq 0$ and show that $\alpha + \beta \in R$. It suffices to show that 
    either $\innerProd \alpha \beta < 0$ or $\innerProd{\alpha + \beta + \gamma}{\gamma} > 0$. For that reason, we will 
    assume $\innerProd \alpha \beta \geq 0$ and show that $\innerProd{\alpha + \beta + \gamma}{\gamma} > 0$. Note that 
    because $\beta + \gamma \notin R$, both $\innerProd \beta \gamma \geq 0$ 
    and $\innerProd{\alpha + \beta + \gamma}{\alpha} \leq 0$. 
    We extract from these inequalities that $\innerProd{\alpha + \gamma}{\alpha} \leq -\innerProd{\beta}{\alpha} \leq 0$. 
    Because $\innerProd{\alpha + \gamma}{\alpha + \gamma} > 0$ and $\innerProd{\alpha + \gamma}{\alpha} \leq 0$, 
    we now have that $\innerProd{\alpha + \gamma}{\gamma} > 0$. Finally, we note that 
    $\innerProd{\alpha + \beta + \gamma}{\gamma} = \innerProd{\alpha + \gamma}{\gamma} + \innerProd{\beta}{\gamma} > 0$, 
    and so we conclude that $\alpha + \beta \in R$. By the symmetry of the statement, $\alpha + \gamma$ 
    is a root provided $\alpha + \beta \neq 0$, proving the first part of the statement. The second part follows
    immediately from the first one.
\end{proof}

\begin{remark}
For root systems we can prove the result above by passing to the corresponding complex Lie algebra 
and using properties of its root decomposition and Jacobi's identity. Such an argument carries over
to the case of QRSs by using the eigenspace decomposition of a Lie algebra with respect to a
toral subalgebra which is not maximal, see \cite{K}.
\end{remark}

\subsection{Paths in \texorpdfstring{$R$}{\it R}}
\label{subsec: paths in R}

\begin{definition}
\label{def: paths}
    A {\em path} from a root $\alpha$ to a root $\beta$ consists of the pair $(\alpha,\beta)$ 
    and a possibly empty sequence of roots $\kappa_1, \kappa_2, \dots, \kappa_n$ 
    such that $\alpha + \kappa_1 + \dots + \kappa_n = \beta$
and for all $i$, we have $\alpha + \kappa_1 + \dots + \kappa_i \in R$. 
    We denote this path $[\alpha; \kappa_1, \dots, \kappa_n; \beta]$ and call each $\kappa_i$ a {\em step}. 
    The roots $\alpha + \kappa_1 + \dots + \kappa_i$ are called the \emph{partial sums} of the path.
    We say that $[\alpha; \kappa_1, \dots, \kappa_n; \beta]$ passes through $\gamma$ if $\gamma$ is such
    a partial sum.

\np 
    The path $[\alpha; -\alpha, \beta; \beta]$ is called the \emph{trivial path from $\alpha$ to $\beta$}.
    
\end{definition}

\begin{remark}\label{rem: paths cross components only at zero}
    If $\alpha$ and $\beta$ are nonzero roots in different components of $R$, then any path from $\alpha$
to $\beta$ necessarily passes through zero. Consequently, any path in $R$ that does not pass through zero is 
contained in a single component of $R$.
\end{remark}

\begin{proposition}
   \label{proposition: existence of paths}
    Let $\beta = \alpha + \kappa_1 + \dots + \kappa_n$, where $\alpha, \beta, \kappa_i \in R$ for $1 \leq i \leq n$. 
    There is a path 
    $[\alpha; \kappa_{i_1}, \kappa_{i_2}, \dots, \kappa_{i_m}; \beta]$ 
    from $\alpha$ to $\beta$ for some (possibly empty) subset 
    $\{i_1, i_2, \dots, i_m\}$ of $\{1, 2, \dots, n\}$.
\end{proposition}

\begin{proof} Set $\mu:= \kappa_1 + \dots + \kappa_n$. If $\mu = 0$, then there is nothing to prove. Thus we assume that $\mu \neq 0$.
First we prove that
there is $1 \leq i \leq n$ such that either $\alpha + \kappa_i \in R$ or $\alpha + \mu - \kappa_i \in R$.
For this we consider two cases for $\mu$:

\begin{enumerate}
\item[(i)] $\mu \not \in R$.  In this case $\alpha + \kappa_i \in R$ for some $1 \leq i \leq n$.
Indeed, since $\beta \in R$  but $\mu = \beta - \alpha \not \in R$, we conclude that $\alpha \neq 0$ and
    \[0 \geq \innerProd\alpha\beta = 
    \innerProd\alpha\alpha + \innerProd\alpha{\kappa_1} + \dots + \innerProd\alpha{\kappa_n} \ . \] 
    Since $\innerProd{\alpha}{\alpha} > 0$, there is $1 \leq i \leq n$ 
    such that $\innerProd{\alpha}{\kappa_i} < 0$ and thus  $\alpha + \kappa_i \in R$.
    \item[]
\item[(ii)] $\mu \in R$. 
    Since $\innerProd{\mu}{\kappa_1} + \innerProd{\mu}{\kappa_2} + \dots + \innerProd{\mu}{\kappa_n} = 
    \innerProd{\mu}{\mu} > 0$, 
    there is $1 \leq i \leq n$ such that $\innerProd{\mu}{\kappa_i} > 0$ and hence $\mu - \kappa_i \in R$. 
    Applying 
    Proposition~\ref{proposition: 2 of 3 rule} to the expression $\beta = (\mu - \kappa_i) + \alpha + \kappa_i$, we conclude 
    that if $\alpha + \kappa_i \notin R$, then $\mu - \kappa_i + \alpha \in R$, since
    $\mu = (\mu - \kappa_i) + \kappa_i \neq 0$.
\end{enumerate}

\np 
In proving the proposition, 
we may and will assume that no non-empty subcollection of $\kappa_1, \dots, \kappa_n$ sums to 0.
With this assumption in mind, the statement of the proposition follows by induction on $n$ from
the observation above.
\end{proof}

\subsection{Reduced paths}
\label{subsec: reduced paths}

Given a path between two roots, there are often multiple 
ways to rearrange its steps without breaking the condition that partial sums be roots. Having the freedom to permute 
the steps in a path may allow us to ensure that all partial sums satisfy additional properties, e.g. 
that they are positive or that they stay within an inversion set. 
It turns out that, for a certain class of paths,  every 
permutation of the steps of such a path defines another path, see 
Proposition~\ref{prop: any rearrangement of a reduced path is a reduced path} below.

\begin{definition}
    A path $[\alpha; \kappa_1, \dots, \kappa_n; \beta]$ is {\em reduced} if 
    no sum of two or more steps $\kappa_i$ is a root. 
\end{definition}

\begin{remark}
    \label{remark: No sum of steps is a root in a reduced path}
 {\bf (i)}   Proposition~\ref{proposition: existence of paths} implies that the path 
    $[\alpha; \kappa_1, \dots, \kappa_n; \beta]$ is reduced if and only if $\kappa_i + \kappa_j \not \in R$
    for $1 \leq i \neq j \leq n$. Indeed, assume that 
    $\kappa_i + \kappa_j \not \in R$ for $1 \leq i \neq j \leq n$ but 
    $\kappa_{i_1} + \dots + \kappa_{i_s} \in R$ for some $i_1, \dots, i_s$. Assume
    furthermore, that $s \geq 3$ is the smallest integer for which such $\kappa_{i_1}, \dots, \kappa_{i_s}$
    exist. Then $\kappa_{i_2} + \dots + \kappa_{i_s} \neq 0$ by the minimality of $s$ 
    and, by Proposition~\ref{proposition: existence of paths}, there is a path between 
    $\kappa_{i_1}$ and $\kappa_{i_1} + \dots + \kappa_{i_s}$. In particular, $\kappa_{i_1} + \kappa_{i_l} \in R$
    for some $2 \leq l \leq s$, contrary to the assumption that $\kappa_i + \kappa_j \not \in R$
    for $1 \leq i \neq j \leq n$.

\np
{\bf (ii)}
It is clear that, given a path $[\alpha; \kappa_1, \dots, \kappa_n; \beta]$, 
there exists a reduced path $[\alpha; \mu_1, \dots, \mu_m; \beta]$, where each 
$\mu_i$ is the sum of a subcollection of the steps $\kappa_1, \kappa_2, \dots, \kappa_n$ of the original path
(and the subcollections corresponding to different steps among $\mu_1, \dots, \mu_m$ are disjoint).
Of course, there may be several reduced paths $[\alpha; \mu_1, \dots, \mu_m; \beta]$ as above.

\np
{\bf (iii)} If $\alpha, \beta$ are nonzero roots in different components of $R$, then the only
reduced path from $\alpha$ to $\beta$ is the trivial one, see Remark~\ref{rem: paths cross components only at zero}.
\end{remark}

\begin{prop}
\label{prop: any rearrangement of a reduced path is a reduced path}
    Let $[\alpha; \kappa_1, \dots, \kappa_n; \beta]$ be a non-trivial reduced path.
    Then for any permutation $i_1, \dots, i_n$ of $1,
    \dots, n$, $[\alpha; \kappa_{i_1}, \dots, \kappa_{i_n}; \beta]$ is a reduced path.
\end{prop}

\begin{proof} If $\kappa_i = - \alpha$ for some $i$, then $\beta = \sum_{j \neq i} \kappa_i$ and the assumption that
$[\alpha; \kappa_1, \dots, \kappa_n; \beta]$ is reduced implies that $n \leq 2$. The paths 
$[\alpha; \kappa_1, \dots, \kappa_n; \beta]$ with $n \leq 2$ and $\kappa_i = -\alpha$ for some $i$ are  
$[\alpha; - \alpha; 0]$, 
$[\alpha; \beta, -\alpha; \beta]$ (if $\alpha + \beta \in R$), and
$[\alpha; -\alpha, \beta; \beta]$. The statement is true for the first two paths and the last path is the trivial path.

\np
Now assume that $n \geq 1$ and $\alpha \neq -\kappa_i$ 
    for all $1 \leq i \leq n$. 
    We prove that any transposition of adjacent steps in $[\alpha; \kappa_1, \dots, \kappa_n; \beta]$ 
    results in another reduced path, from which the desired result follows. To that end, we take some $1 \leq i < n$ and 
    consider swapping the positions of $\kappa_i$ and $\kappa_{i + 1}$. The resulting sequence of partial sums is the 
    same as that of the original path, except $\alpha + \kappa_1 + \dots + \kappa_{i - 1} + \kappa_i$ is replaced 
    by $\alpha + \kappa_1 + \dots + \kappa_{i - 1} + \kappa_{i + 1}$. We therefore need only to show that this is a root. 
    To that end, we apply Proposition~\ref{proposition: 2 of 3 rule} to the expression 
    $(\alpha + \kappa_1 + \dots + \kappa_{i - 1}) + \kappa_i + \kappa_{i + 1}$ to conclude 
    that $\alpha + \kappa_1 + \dots + \kappa_{i - 1} + \kappa_{i + 1} \in R$ provided 
    $\alpha + \kappa_1 + \dots + \kappa_{i - 1} + \kappa_i \neq 0$. If $i = 1$, this is guaranteed by the condition 
    that $\alpha \neq -\kappa_j$ for any $1 \leq j \leq n$. If $i > 1$, then we would have a non-trivial sum of steps in a 
    reduced path equal to $-\alpha$, which is impossible. 
    We therefore conclude that 
    $[\alpha; \kappa_1, \dots, \kappa_{i-1},\kappa_{i + 1}, \kappa_i,\kappa_{i+2}, \dots, \kappa_n; \beta]$ 
    is another reduced path.
\end{proof}

\begin{example}
    In the root system $\mathbb D_5$ we have the following equation
    $$ \DFive01000 - \DFive00001 + \DFive00101 + \DFive00111 + \DFive10000 = \DFive11211 \ .$$
    (Here roots have been identified with their coefficients as in 
    Remark~\ref{Remark:coefficient vector}.)  The equation above yields the path
    $$ \left[\DFive01000; \DFive10000, \DFive00111, -\DFive00001, \DFive00101; \DFive11211 \right] \ , $$
which is not reduced but gives rise to the reduced path
    $$ \left[\DFive01000; \DFive10000, \DFive00110, \DFive00101; \DFive11211 \right] \ .$$
    To verify Proposition~\ref{prop: any rearrangement of a reduced path is a reduced path}, we 
    check that each permutation of the steps in the reduced path
    above properly defines a path.
    $$\begin{array}{ccc}
        \left[\DFive01000; \DFive10000, \DFive00110, \DFive00101; \DFive11211 \right]& 
        \left[\DFive01000; \DFive10000, \DFive00101, \DFive00110; \DFive11211 \right]       & 
        \left[\DFive01000; \DFive00110, \DFive10000, \DFive00101; \DFive11211 \right]   \\[8pt]
        \left[\DFive01000; \DFive00110, \DFive00101, \DFive10000; \DFive11211 \right]       &
        \left[\DFive01000; \DFive00101, \DFive10000, \DFive00110; \DFive11211 \right]       & 
        \left[\DFive01000; \DFive00101, \DFive00110, \DFive10000; \DFive11211 \right]
    \end{array}$$
\end{example}

\np  It is not too difficult to see that every partial sum along each of these paths is a root in $\mathbb D_5$.

\subsection{An application}

The following proposition will be useful in Section~\ref{sec: components as inflations}. 

\begin{prop}
\label{prop: co-closed sets span their support}
    Let $\Phi$ be a co-closed subset (for instance, an inversion set) of $R^+$.  Then  
    \[\ZSpan \Phi=\ZSpan \supp \Phi \ . \]
\end{prop}

\begin{proof} It follows from the definition of $\supp \Phi$ that $\ZSpan \Phi \subseteq \ZSpan \supp \Phi$. 
To demonstrate the reverse inclusion, we take $\theta \in \supp \Phi$ and show that it is in $\ZSpan \Phi$. 
This is obvious if $\theta \in \Phi$, so we consider the case where $\theta \notin \Phi$. Since $\theta \in \supp \Phi$, 
there is some $\alpha \in \Phi$ such that $\theta < \alpha$. We may therefore, appealing to 
Proposition~\ref{proposition: existence of paths}, form a path from $\theta$ to $\alpha$ with positive simple roots as steps. 
By Remark~\ref{remark: No sum of steps is a root in a reduced path}(ii), we may replace 
this path by the reduced path $[\theta; \kappa_1, \dots, \kappa_n; \alpha]$ 
with positive steps. Using Proposition~\ref{prop: any rearrangement of a reduced path is a reduced path}, we choose an  
ordering of the $\kappa_i$'s with those steps in $\Phi$ preceding those in $\Phi^c$:
$$ [\theta; \underbrace{\kappa_1, \dots, \kappa_m}_{\kappa_i \in \Phi}, \underbrace{\kappa_{m + 1}, \dots, 
\kappa_n}_{\kappa_i \in \Phi^c}; \alpha]\ .$$
Co-closure implies that the partial sums $\alpha-\kappa_n$, $\alpha-\kappa_n-\kappa_{n-1}$
down to $\theta + \kappa_1 + \dots + \kappa_m$ are all in $\Phi$, 
from which it follows that $\theta \in \ZSpan \Phi$.
\end{proof}

 \section{Inflation on subsets of positive roots} \label{sec: inflations}

\np In this section, we introduce the notion of inflation in a QRS. The notion of 
inflation was first studied in the context of the symmetric group (see for example \cite{AAK}), 
by giving a natural way to decompose every inversion set in a unique way, and in particular 
giving rise to some elements that cannot be decomposed non-trivially (that we will call primitive, 
although other authors have used the term simple). On the one hand, the notion of inflation that we 
present generalizes the corresponding notion for the symmetric group. On the other hand, 
inflation will allow us to better understand inversion sets and their decompositions by providing some inductive methods 
to study QRSs.

\np 
Throughout this section $\Phi$ denotes a subset of $R^+$ and,
for a given $I \subseteq S$, $\Psi$ denotes a subset of $(R/I)^+$.  

\subsection{Definition and basic properties}

\begin{definition}
\label{def: inflation}
    Let $I \subseteq S$ and $\Phi \subseteq R^+$. We say that $\Phi$ is an \emph{inflation} from $I$ 
    of $\Psi \subseteq (R / I)^+$ by $X \subseteq R_I^+$ if 
    $\Phi$ consists of all preimages under the canonical projection $\pi_I: V \to V / \Span I$ 
    of the elements of $\Psi$ along with the elements of $X$.
    Formally,
    \[\Phi = (\pi_I^{-1}(\Psi)\cap R) \cup X \ .\]
 We write $\Phi = \inf_I^S(\Psi, X)$.  
\end{definition}

\np 
\begin{remark} \label{Remark:inflations for reducible}
Note that Definition~\ref{def: inflation} is meaningful regardless of whether $R$ is assumed to be reducible or 
irreducible. If $R$ is reducible, write $R = R^1 \times \cdots \times R^s$ where each $R^i$ is irreducible with base $S^i$.
    We can write $I = I^1 \cup \cdots \cup I^s$ with $I^i \subseteq S^i$. Observe that $R/I$ can be identified with 
    $R^1/I^1 \times \cdots \times R^s/I^s$. For any $\Phi \subseteq R^+$, we let $\Phi^i = \Phi \cap R^i$ and for 
    $\Psi \subseteq (R/I)^+$, we let $\Psi^i = \Psi \cap R^i/I^i$. Note that $\Phi^i \subseteq (R^i)^+$ and 
    $\Psi^i \subseteq (R^i/I^i)^+$.
    
   \np 
   Then for $\Phi \subseteq R^+$, we have that $\Phi = \inf_I^S(\Psi, X)$ if and only if 
   $\Phi^i = \inf_{I^i}^{S^i}(\Psi^i, X^i)$, where $X^i = X \cap R_{I^i}$. 
   This justifies restricting our attention to irreducible QRSs.
\end{remark}

\begin{remark} We will often use the following simple properties of inflation without explicit reference.
    \begin{enumerate}[(i)]
        \item $\inf_I^S(\Psi, X)$ consists of all positive roots in $R$ which project to 
        roots in $\Psi$ along with the elements of $X$ (which are some of the positive roots that project to 0).
        \item Every $\Phi\subseteq R^+$ is an inflation $\Phi = \inf_S^S(\emptyset, \Phi)$ as 
        well as $\Phi = \inf_\emptyset^S(\Phi, \emptyset)$.
        \item $\inf_I^S(\Psi_1, X_1) = \inf_I^S(\Psi_2, X_2)$ if and only if  $\Psi_1 = \Psi_2$ and $X_1 = X_2$.
        \item $\Phi = \inf_I^S(\Psi, X)$ if and only if $\Phi^c = \inf_I^S(\Psi^c, X^c)$.
        \item $\inf_I^S(\Psi_1, X_1) \cup \inf_I^S(\Psi_2, X_2) = \inf_I^S(\Psi_1 \cup \Psi_2, X_1 \cup X_2)$
    \end{enumerate}
\end{remark}

\np 
The following two propositions are straightforward and we omit the proofs here.

\np 
\begin{prop} Let $I \subseteq S$. The following are equivalent:
\label{prop: conditions equivalent to being an inflation}
    \begin{enumerate}[(i)]
        \item $\Phi$ is inflated from $I$.

        \item Suppose $\alpha, \beta \in R^+$ satisfy $\pi_I(\alpha) = \pi_I(\beta) \neq 0$. 
        Then $\alpha \in \Phi$ if and only if $\beta \in \Phi$. 

        \item Suppose $\alpha \in R^+$ and $\theta \in \pm I$ satisfy $\pi_I(\alpha) \neq 0$ 
        and $\alpha + \theta \in R^+$. Then $\alpha + \theta \in \Phi$ if and only if $\alpha \in \Phi$.

        \item For any $\bar \alpha \in (R / I)^+$, $\pi^{-1}(\bar \alpha)\cap R^+ \subseteq \Phi$ 
        or $\pi^{-1}(\bar \alpha)\cap R^+ \subseteq \Phi^c$.

        \item $\pi_I(\Phi) \cap \pi_I(\Phi^c) \subseteq \{ 0 \}$. \qed
    \end{enumerate}
\end{prop}

\np 
\begin{prop}
\label{prop: inflation of an inflation is an inflation 1}
    Let $I \subseteq J \subseteq S$ and let $X \subseteq (R_I)^+$, $T \subseteq (R_{J} / I)^+$, and $\Psi \subseteq (R/ J)^+$. Then

    \begin{equation}
    \label{eqn: inflation of an inflation is an inflation part 1}
        \inf_I^S(\inf_{J / I}^{S / I}(\Psi, T), X) = \inf_J^S(\Psi, \inf_I^J(T, X)) \ ,
    \end{equation}
    where the natural isomorphisms $(R_J)_I \cong R_I$, $(R/I)_{J/I} \cong R_J/I$, and $(R/I)/(J/I) \cong R/J$
    are used to make sense of $X$ in the right-hand side and of $T$ and $\Psi$ in the left-hand side of the identity. \qed 
\end{prop}

\np 
\begin{prop}
    \label{prop: inflation of an inflation is an inflation part 2}
    Assume $I \subseteq J$ and $\Phi = \inf_I^S(\Psi,X) = \inf_J^S(\Theta, Y)$. 
    Then $\Psi = \inf_{J/I}^{S/I} (\Theta, Z)$ for some $Z \subseteq R_{J/I}^+$.
\end{prop}

\begin{proof}
     Note that $Y \subseteq  \Span J \subseteq  \Span S$. Since
$Y = \Phi \cap \Span J$, using Proposition~\ref{prop: conditions equivalent to being an inflation}~(ii), 
one verifies easily that $Y = \inf_I^J(Z,X)$
for some $Z \subseteq  \Span J/I$. 
Then \eqref{eqn: inflation of an inflation is an inflation part 1} above implies
\[
\Phi = \inf_I^S(\Psi,X) = \inf_J^S(\Theta, Y) = \inf_J^S(\Theta, \inf_I^J(Z,X)) =
\inf_I^S(\inf_{J/I}^{S/I}(\Theta,Z), X) \ ,
\]
proving that $\Psi = \inf_{J/I}^{S/I} (\Theta, Z)$. \end{proof}

\np 
Next, we introduce the set $\Gen(\Phi)$ for each subset $\Phi$ of $R^+$, which captures the 
subsets of $S$ from which $\Phi$ can be inflated.

\np 
\begin{definition}
\label{def: Gen(Phi)}
Given $\Phi \subseteq R^+$, define 
$$\Gen(\Phi) := \{ K \subseteq S \, | \, \Phi = \inf_K^S(\Theta, Y) \text{ for some } \Theta \text{ and some } Y\} \ ,$$    
the sets from which $\Phi$ is inflated.
\end{definition}

\np 
\begin{remark} The following properties of $\Gen(\Phi)$ are obvious:
\begin{enumerate} 
\item[(i)] $\emptyset, S \in \Gen(\Phi)$
\item[(ii)] $\Gen(\Phi) = 2^S$ if and only if $\Phi = \emptyset$ or $\Phi = R^+$.
\item[(iii)] $\Gen(\Phi) = \Gen(\Phi^c)$.
\end{enumerate}
\end{remark}

\begin{prop}
    \label{prop: Gen(Phi) closed under union and intersection}
    Let $I, J \in \Gen(\Phi)$. Then $I \cup J \in \Gen(\Phi)$ and $I \cap J \in \Gen(\Phi)$.
\end{prop}

\begin{proof}
    The statement about $I \cap J$ is straightforward. The proof about $I \cup J$ is an immediate consequence 
    of the fact that, for any $K \subseteq  S$ and non-zero $\nu \in R/K$, any two elements of $\pi_K^{-1}(\nu) \cap R^+$ 
    may be joined by a path with steps in $K$,
     which is a consequence of Proposition~\ref{proposition: existence of paths}. 
\end{proof}
 
\subsection{Canonical form}

 The aim of this subsection is to provide a canonical way to describe a subset of $R^+$ as an inflation.  
 Recalling from Remark~\ref{Remark:inflations for reducible} that any inflation in a reducible system 
 $R$ can be viewed naturally as a disjoint union of inflations within the irreducible components of $R$, 
 we persist with the global assumption that $R$ is irreducible.

\begin{definition} \label{def_primitive}
Let $R$ be an irreducible QRS. 
   A subset $\Phi \subseteq  R^+$ is {\em primitive} if $\Phi \neq \emptyset, \Phi \neq R^+$, and $\Gen(\Phi) = \{\emptyset, S\}$.
\end{definition}

\begin{remark}\label{rem: phi primitive iff phi complement primitive}
\begin{enumerate}
 \item[(i)] If $\Phi$ is primitive, then $\Phi = \inf_\emptyset^S(\Phi, \emptyset)$ 
 and $\Phi = \inf_S^S(\emptyset, \Phi)$ are the only presentations of $\Phi$ as an inflation.
 \item[(ii)] $\Phi$ is primitive if and only if $\Phi^c$ is primitive.  
    \item[(iii)] If $R$ is reducible, then every subset $\Phi \subseteq R^+$ can be written as some non-trivial inflation 
    because every set is an inflation with respect to any quotient that sends an entire 
    irreducible component to $0$.  Hence the obvious 
    extension of Definition~\ref{def_primitive} to subsets of reducible 
    QRSs would yield no primitive sets.
 \end{enumerate}
\end{remark}

 \begin{theorem}
 \label{theorem: canonical inflation}
     Let $R$ be an irreducible QRS, and let $\Phi \subseteq R^+$. Then $\Phi = \inf_I^S(\Psi, X)$ where $I \subsetneq S$ 
     and one of the following mutually exclusive alternatives holds:
     \begin{enumerate}
        \item[(i)] $\Psi$ is primitive 
        \item[(ii)] $\Psi = \emptyset$ or $\Psi= (R/I)^+$.
     \end{enumerate}
     Moreover, the set $I$ is uniquely determined in case (i) and in case (ii) there is a minimum such $I$.  
 \end{theorem}

\begin{definition}
    Following  Theorem~\ref{theorem: canonical inflation}, we call the expression $\Phi = \inf_I^S(\Psi, X)$ where either
    \begin{enumerate}
    \item[(i)]$\Psi$ is primitive
    \item[] or 
    \item[(ii)] $\Psi = \emptyset$ or $(R/I)^+$ and $I$ is the minimal
    subset of $S$ with this property
    \end{enumerate}
 the {\it canonical form of $\Phi$}. In this case, we also say that $\Phi$ is \emph{canonically inflated from} $I$. 
\end{definition}

\np 
Note that in \cite{DDMRWW}, the notion of a ``simple form of a permutation'' was studied.
The canonical form of an inversion set in type $\mathbb A$ is the counterpart of the
simple form of the corresponding permutation.

\np  Before proving the theorem, we establish the following useful result.

 \begin{lemma}
     \label{lemma: inflation from I' cup J'}
     Assume that $I \cap J = \emptyset$ and $I \cup J = S$.
    Define 
    \[I' := I \cap J^\perp = \{\theta \in I \, | \, \theta \perp \lambda \ \forall  
    \lambda \in J\}  
    \quad
    {\text{and}} \quad
    J' := J \cap I^\perp = \{\theta \in J \, | \, \theta \perp \lambda \ \forall
    \lambda \in I\} \ .\]
    If $\Phi = \inf_I^S(\Psi, X) = \inf_J^S(\Theta,Y)$, then $\Phi = \inf_{I' \cup J'}^S(\Xi, Z)$,
    where $\Xi = \emptyset$ or $\Xi = (R/(I' \cup J'))^+$.
 \end{lemma}

 \begin{proof}
 For $\alpha, \beta \in R^+$ and $\Phi \subseteq R^+$ we write $\alpha \sim_\Phi \beta$ if either 
both $\alpha, \beta \in \Phi$ or both $\alpha, \beta \notin \Phi$.
      Let $\tau$ denote the highest root of $R^+$. 

\np 
Since $I' \subseteq I$ and $J' \subseteq I^\perp$, we have $R_{I'\cup J'} = R_{I'} \times R_{J'}$.
Noting that $\tau \not \in R_{I'\cup J'}$ (unless $I = \emptyset$ or $J = \emptyset$ when there is nothing to prove),
we conclude that the statement of the lemma is equivalent to proving that $\gamma \sim_\Phi \tau$ for
every $\gamma \in R^+ \backslash (R_{I'} \cup R_{J'})$. To prove this we use the fact that
\[ R^+ \backslash (R_{I'} \cup R_{J'}) = (R^+ \backslash (R_{I} \cup R_{J})) \cup 
 ((R_{I} \cup R_{J}) \backslash (R_{I'} \cup R_{J'})) \ ,\]
 and consider two cases:
 \begin{enumerate}
\item[(i)] $\gamma \in R^+ \backslash (R_{I} \cup R_{J})$.

\np
In this case there is a path $[\gamma; \kappa_1, \dots, \kappa_n; \tau]$ with $\kappa_i \in S = I\cup J$.
Since $\gamma \not \in R_I \cup R_J$, we see that $\gamma + \kappa_1 + \dots + \kappa_i \not \in R_I \cup R_J$
for every $1 \leq i \leq n$. Moreover, for $1 \leq i \leq n$,
\[
\gamma + \kappa_1 + \dots + \kappa_{i-1} \, \sim_\Phi \gamma + \kappa_1 + \dots + \kappa_{i-1} + \kappa_{i}
\]
because $\Phi$ is inflated both from $I$ and $J$ and $S = I \cup J$. Using that $\sim_\Phi$ is an equivalence
relation, we conclude that $\gamma \sim_\Phi \tau$.

\item[] 

\item[(ii)] $\gamma \in (R_{I} \cup R_{J}) \backslash (R_{I'} \cup R_{J'})$.

\np
Without loss of generality we may assume that $\gamma \in R_I \backslash R_{I'}$.
Then $\supp \gamma$ is contained in $I$ but not in $I'$. 
The definition of $I'$ implies that there are $\beta \in J$ and $\theta \in \supp \gamma$ such
that $\langle \beta, \theta \rangle \neq 0$. Hence $\langle \beta, \theta \rangle < 0$ and, 
moreover, $\langle \gamma, \beta \rangle < 0$. Thus $\gamma + \beta \in R^+$ and since 
$\pi_J(\gamma) = \pi_J(\gamma + \beta) \neq 0$, we conclude that 
$\gamma \sim_\Phi \gamma + \beta$. Noting that $\gamma + \beta \not \in R_I \cup R_J$
and using case (i), we conclude that $\gamma + \beta \sim_\Phi \tau$. Combining the two
equivalences, we arrive at $\gamma \sim_\Phi \tau$ which completes the proof. \qedhere
 \end{enumerate}
 \end{proof}

 \begin{proof}[Proof of Theorem~\ref{theorem: canonical inflation}]
     {\it Existence.} 
If $\Phi$ is primitive, $\Phi = \emptyset$, or $\Phi = R^+$, then
we observe that $\Phi = \inf_\emptyset^S(\Phi, \emptyset)$ and we are done.
If $\Phi$ is not primitive, then $\Phi = \inf_K^S(\Xi, Y)$ for some proper non-empty subset
$K$ of $S$. 

\np 
If $\Xi$ is primitive, then $\Phi = \inf_K^S(\Xi, Y)$ is the required decomposition.
If $\Xi = \emptyset$ or $\Xi = (R/K)^+$, then the required decomposition is 
$\Phi = \inf_I^S(\emptyset, Y)$ with $I := \supp \Phi$ or $\Phi = \inf_I^S((R/I)^+, Y)$ with $I := \supp \Phi^c$, respectively.

\np 
If $\Xi$ is not primitive, Proposition~\ref{prop: inflation of an inflation is an inflation 1} 
implies that  $\Phi$ is an inflation from a proper subset $K'$ of $S$ which contains $K$ as a proper subset.
This process is finite since, for every $L \subseteq S$ with $\#L = \rk R - 1$,
every subset of $(R/L)^+$ is either primitive, the empty set, or $(R/L)^+$ itself, proving the
existence part of the statement.

\np 
{\it Uniqueness.} If $\Phi = \inf_I^S(\emptyset, X)$ for some $X$ and $I \subsetneq S$,
then 
$\supp \Phi \subseteq I$. This in contrast to the case when $\Phi$ is an inflation
of $(R/J)^+$ for some proper $J \subsetneq S$, where $\supp \Phi = S$. This proves 
that $\Phi$ cannot be an inflation of both $\emptyset$ and $(R/K)^+$ for any $K \subsetneq S$.

\np 
Now assume that $\Phi = \inf_I^S(\Psi, X) = \inf_J^S(\Theta, Y)$, where $\Psi$ is primitive
and $\Theta$ is either primitive or equals $\emptyset$ or $(R/J)^+$. 
By Proposition~\ref{prop: Gen(Phi) closed under union and intersection}, we have that 
$\Phi = \inf_{I \cup J}^S(\Gamma, Z)$ for some $\Gamma$ and $Z$. Now, 
since $\Phi = \inf_I^S(\Psi, X)=\inf_{I \cup J}^S(\Gamma, Z)$, by 
Proposition~\ref{prop: inflation of an inflation is an inflation part 2}, it 
follows that $\Psi = \inf_{I\cup J/I}^{S/I}(\Gamma, T)$ for some $T$. Since $\Psi$ is primitive, 
we have $I \cup J = I$ or $I \cup J = S$. In the first case, uniqueness follows from 
Proposition~\ref{prop: inflation of an inflation is an inflation part 2}. Hence, we may assume that $I \cup J = S$.
Moreover, if $I\cap J \neq \emptyset$, then 
$\Phi = \inf_{I \cap J}^S(\Xi, Z)$, implying that $\Xi$ is an inflation of $\Psi$ and $\Theta$ which
proves the uniqueness by induction on $\rk R$. 

\np 
It remains to deal with the case $I\cap J = \emptyset$, $I \cup J = S$. Lemma~\ref{lemma: inflation from I' cup J'} implies
that $\Phi$ is an inflation from $I' \cup J'$, as defined in Lemma~\ref{lemma: inflation from I' cup J'}. 
If $J'$ is not empty, then $\Phi$ is
also an inflation from $I \cup J'$, contradicting the assumption that $\Psi$ is primitive.
If both $I'$ and $J'$ are empty, Lemma~\ref{lemma: inflation from I' cup J'} implies that $\Phi = \emptyset$
or $\Phi = R^+$, again contradicting the assumption that $\Psi$ is primitive.

\np 
Finally we consider the case when $I' \neq \emptyset$, $J' = \emptyset$, $\Theta = \emptyset$ 
(the case $\Theta = (R/J)^+$ follows because of Remark~\ref{rem: phi primitive iff phi complement primitive}). 
Lemma~\ref{lemma: inflation from I' cup J'}
implies that $\Phi$ is an inflation of the empty set from $I'$ and $J$. 
In particular, $\supp \Phi \subseteq I' \cap J = \emptyset$,
proving that $\Phi = \emptyset$.
 \end{proof}

\begin{remark}
    If $\Phi$ is canonically inflated from $I$, then so is $\Phi^c$.
\end{remark}
\begin{remark}
    If $\Phi$ is primitive then its canonical expression is
    $\Phi = \inf_\emptyset^S(\Phi,\emptyset)$,  c.f.\ Remark~\ref{rem: phi primitive iff phi complement primitive}.
\end{remark}

\subsection{The sets $\supp \Phi$ and $\Gen(\Phi)$ via the canonical form of $\Phi$}
We can now apply the canonical form of $\Phi$ to the problem of computing the sets 
$\supp \Phi$ and $\Gen(\Phi)$ introduced earlier.

\begin{prop} Let $\Phi = \inf_I^S(\Psi, X)$ be the canonical form of $\Phi$.
    \label{prop: support of inflations}
    \begin{enumerate}
    \item[(i)] $\Psi$ is primitive if and only if $\supp \Phi = \supp \Phi^c = S$;
    \item[(ii)] $\Psi = \emptyset$ if and only if $\supp \Phi$ is a proper subset of $S$ and $\supp \Phi^c = S$.
    Respectively, $\Psi = (R/I)^+$ if and only if $\supp \Phi = S$ and $\supp \Phi^c$ is a proper subset of $S$.
    \end{enumerate}
\end{prop}

\begin{proof}
     If $\Phi = \inf_I^S(\emptyset, X)$, then $\supp \Phi = I$. 
Conversely, if $\supp \Phi = J$ is a proper subset of $S$, then $\Phi \subseteq R_J^+$ and
$\Phi = \inf_J^S (\emptyset, \Phi)$. This proves that $\supp \Phi = S$ unless 
$\Phi = \inf_I^S(\emptyset, X)$.
\end{proof}

\np 
We next describe the set $\Gen(\Phi)$ when $\Psi$ in the canonical form 
$\Phi = \inf_I^S(\Psi, X)$
 is primitive. When $\Psi$ is not primitive, the description of $\Gen(\Phi)$ is more complicated and 
we omit it.  In the case when $\Phi$ is an inversion set, Proposition~\ref{prop: Gen(Phi) from hyperplane} provides a
characterization of all elements of $\Gen(\Phi)$ in terms of the positive system $P \subseteq R$ which determines $\Phi$.

\begin{prop}
    \label{prop: Gen(Phi) for primitive Psi}
    Let $\Phi = \inf_I^S(\Psi, X)$ be the canonical form of $\Phi \subseteq R^+$ where $\Psi$ is primitive. 
    Then $\Gen(\Phi) = \Gen(X) \cup \{ S \}$. 
\end{prop}

\begin{proof}
    First we prove that $\Gen(\Phi) \backslash \{S\}$ contains a unique maximal element.
Assume, by way of contradiction,
that $\Gen(\Phi) \backslash \{S\}$ contains two distinct maximal elements $K$ and $L$.
Then $K \cap L$ and $K \cup L$ are both elements
of $\Gen(\Phi)$ by Proposition~\ref{prop: Gen(Phi) closed under union and intersection} and 
$K \cup L = S$ by the maximality of $K$ and $L$.  Write
\[\Phi=\inf_K^S(\Psi',X') = \inf_L^S(\Psi'',X'') = \inf_{K\cap L}^S(\Gamma,Y) \ .\]  
We define $K'' := K/(K \cap L)$ and $L'' := L/(K \cap L)$ and $S'' := S/(K \cap L)$.  Then 
$K'' \cap L'' = \emptyset$ and $K'' \cup L'' = S''$.
By Proposition~\ref{prop: inflation of an inflation is an inflation part 2}, there exist $\theta',\theta'',Z',Z''$ such that 
\[\Gamma = \inf_{K''}^{S''}(\theta',Z') = \inf_{L''}^{S''}(\theta'',Z'') \ . \] 
Applying Lemma~\ref{lemma: inflation from I' cup J'},  we define
$K' := K'' \cap (L'')^\perp \subseteq S''$ and
$L' := L'' \cap (K'')^\perp \subseteq S''$ to get
$\Gamma = \inf_{K'\cup L'}^{S''}(\Xi,W)$ 
where $\Xi = \emptyset$ or $\Xi= R/(K'\cup L')$.

\np
Note that $K' \cup L' \neq S''$.  To see this, recall that $S''$ is irreducible and thus 
there exists $\alpha \in K''$ and $\beta \in L''$ with
$\alpha$ and $\beta$ not perpendicular.  Thus 
both $\alpha$ and $\beta$ lie in
$S'' \setminus (K'\cup L')$.

\np
Therefore 
\[ \Phi=\inf_{K\cap L}^S(\Gamma,Y) = 
\inf_{K\cap L}^S(\inf_{K'\cup L'}^{S''}(\Xi,W),Y)
= \inf_{K'\cup L'}^S(\Xi,U) \] 
for some $U$ where
$\Xi = \emptyset$ or $\Xi= R/(K'\cup L')$.  But this contradicts the hypothesis that
$\Phi = \inf_I^S(\Psi,X)$ is the canonical form of $\Phi$ where $\Psi$ is primitive.  
This contradiction shows that $\Gen(\Phi) \backslash \{S\}$ has a unique maximal element.  

\np
Next we note that this maximal element is $I$.  Suppose, to the contrary, that 
$\Phi = \inf_I^S(\Psi,X) = \inf_K^S(\Psi',X')$
where $I \subsetneq K \subsetneq S$.  
Then by Proposition~\ref{prop: inflation of an inflation is an inflation part 2}, we have
$\Psi = \inf_{K/I}^{S/I}(\Psi',Z)$ for some
$Z \subseteq \text{span}(R/I)$.  This contradicts the
assumption that $\Psi$ is primitive, proving that $I$ is the unique maximal element of $\Gen(\Phi) \backslash \{S\}$.

\np 
To complete the proof that $\Gen(\Phi) = \Gen(X) \cup \{S\}$, assume that $J \in \Gen(\Phi) \backslash \{S\}$.
Then $J \subseteq I$ and 
\[\Phi = \inf_I^S(\Psi, X) = \inf_J^S(K, Y) \ .\]
By Proposition~\ref{prop: inflation of an inflation is an inflation part 2}, $K = \inf_{I/J}^{S/J}(\Theta, Z)$ and
\[\Phi =  \inf_J^S(K, Y) = \inf_J^S(\inf_{I/J}^{S/J}(\Theta, Z),Y) 
= \inf_I(\Theta, \inf_I^J(Z, Y)) \ .\]
Comparing the above with $\Phi = \inf_I^S(\Psi,X)$, we conclude that $\Theta = \Psi$ and $X = \inf_J^I(Z,Y)$,
proving that $J \in \Gen (X)$. This shows that $\Gen(\Phi) \subseteq \Gen(X) \cup \{S\}$. The converse
inclusion is obvious. 
\end{proof}

\subsection{Inflations and inversion sets}

\begin{prop} \label{prop: inflations and inversions}
     Let $\Phi = \inf_I^S(\Psi, X)$. Then $\Phi$ is an inversion set if and only if 
    both $\Psi$ and $X$ are inversion sets.
\end{prop}

\begin{proof}
    Assume that $\Psi, X$ are inversion sets.
    To prove that $\Phi$ is closed, 
    take roots $\alpha, \beta \in \Phi$ with
   $\alpha + \beta \in R$ and show that
   $\alpha + \beta \in \Phi$. If $\alpha, \beta \in X$, 
    then $\alpha + \beta \in X \subseteq \Phi$ by closure of $X$. Assume that $\alpha \in X$ 
    and $\pi_I(\beta) \in \Psi$. Then $\pi_I(\alpha + \beta) = \pi_I(\beta) \in \Psi$. 
    Hence, $\alpha + \beta \in \Phi$. Finally, if 
    both $\pi_I(\alpha), \pi_I(\beta) \in \Psi$, then  closure of $\Psi$ 
    implies that $\pi_I(\alpha + \beta) = \pi_I(\alpha) + \pi_I(\beta) \in \Psi$. 
    Hence, $\alpha + \beta \in \Phi$. This proves that $\Phi$ is closed. The fact that $\Phi$ is co-closed follows from
    the identity
    $\Phi^c = \inf_I^S(\Psi^c, X^c)$ and the proof above.

\np 
    Conversely, assume that $\Phi$ is an inversion set. We first prove that $X$ is also. The 
    closure property of $X$ follows from the closure property of $\Phi$ and that 
    $X = \Phi \cap {\rm span}(I)$. For the co-closure, assume that $\alpha, \beta \in X^c$ 
    and $\alpha + \beta$ is a root. If $\alpha, \beta \in \Phi^c$, then the co-closure of $\Phi$ 
    implies that $\alpha + \beta \in X^c$. But if $\alpha \in \pi_I^{-1}(\Psi)$, 
    then $\alpha + \beta \not \in {\rm span}(I)$ so $\alpha + \beta \in X^c$. Hence, 
    $X$ is also co-closed and thus an inversion set. It remains to prove that $\Psi$ is an 
    inversion set. For the closure, let $\bar \alpha, \bar \beta \in \Psi$ with 
    $\bar \gamma := \bar \alpha + \bar \beta$ a root.  It follows from 
    Proposition~\ref{prop: lifts of sums in quotients} that for a given 
    $\gamma \in \pi_I^{-1}(\bar\gamma)$, there  are $\alpha \in \pi_I^{-1}(\bar\alpha)$ and
$\beta \in \pi_I^{-1}(\bar\beta)$ with $\gamma = \alpha + \beta$. If $\bar \gamma \not\in \Psi$, then 
$\gamma \in \Phi^c$. As $\bar \alpha, \bar \beta \in \Psi$, we have that $\alpha, \beta \in \Phi$. 
The closure property of $\Phi$ implies that $\gamma \in \Phi$, which is a contradiction.  Similarly, 
for the co-closure, if $\bar \alpha, \bar \beta \in \Psi^c$ with $\bar \alpha + \bar \beta$ a root 
in $\Psi$, then we get an expression $\gamma = \alpha + \beta$ of three roots with 
$\alpha, \beta \in \Phi^c$ and $\gamma \in \Phi$, which contradicts the co-closure of $\Phi$. 
Thus, $\Psi$ is an inversion set.
\end{proof}

\np 
We complete this section with a description of $\Gen(\Phi)$ different from the one in Proposition~\ref{prop: Gen(Phi) for primitive Psi} 
in the case when $\Phi$ is an inversion set. Before stating the result, we introduce a notion related to a positive system $P$ in $R$,
see Remark~\ref{rem: inversion Sets lie between two hyperplanes}. Let $S'$
be the base of $P$.
A subspace $W \subseteq E$ is {\it a $P$-Levi subspace} if it is spanned by a subset of $S'$.

\begin{proposition} \label{prop: Gen(Phi) from hyperplane}
Let $P$ be a positive system in $R$ and let $\Phi := R^+ \cap P$ be the corresponding inversion set.
Then $I \in \Gen(\Phi)$ if and only if $W:= \Span I$ is a $P$-Levi subspace of $E$.
\end{proposition}

\begin{proof}
Assume first that $W$ is a $P$-Levi subspace of $E$. To prove that $I \in \Gen(\Phi)$, we need to show that,
for any $\beta, \gamma \in R^+ \backslash W$ such that $\beta - \gamma \in W$, either both $\beta$ and $\gamma$
belong to $\Phi$ or both belong to $\Phi^c$. Since $\Phi = R^+ \cap P$, this is equivalent to saying that
$\beta$ and $\gamma$ are on the same side of a hyperplane defining $P$. 
Consider the hyperplane $\mathcal{H}_\vep := \ker \varphi_\vep$, 
where $\varphi_\vep:E \to \RR$ is defined by its values on the simple roots of $P$ as follows:
$\varphi_\vep(\theta) = \vep$ for $\theta \in W$
and $\varphi_\vep(\theta) = 1$ if $\theta \not \in W$. The hyperplane $\mathcal{H}_\vep$ determines $P$.
Since $W$ is $P$-Levi, there is $c >0$ such that $\varphi_\vep(\alpha) \leq c \vep$ if $\alpha \in R \cap W$ and
$\varphi_\vep(\alpha) \geq 1$ for $\alpha \in P \backslash W$. Thus,
for small enough positive values of $\vep$, 
any two roots $\beta, \gamma \in R \backslash W$ with $\beta - \gamma \in W$
lie on the same side of $\mathcal{H}_\vep$. This completes the argument that $I \in \Gen(\Phi)$.

\np 
Next we assume that $I \in \Gen(\Phi)$ but $W$ is not a $P$-Levi subspace of $E$. Label the 
roots in the base $S'$ of $P$ as $\theta_1', \dots, \theta_n'$ so that $\theta_1', \dots, \theta_k' \in W$
and $\theta_{k+1}', \dots, \theta_n' \not \in W$. The assumption that $W$ is not a $P$-Levi subspace of $E$
implies that $\theta_1', \dots, \theta_k'$ do not span $W$. Then there is $w' \in P \cap W$ such that 
$w' \not \in \Span \{\theta_1', \dots, \theta_k'\}$. (For example we can take $w' = \pm \theta$ where $\theta$
is a simple root of $R^+$ not contained in $\Span \{\theta_1', \dots, \theta_k'\}$.) Climbing down from $w'$ towards $0$
with steps among $\theta_1', \dots, \theta_k'$, we can replace $w'$ with $w$ which satisfies the following conditions:
$w \in W \cap P$ and $w = \beta + \gamma$, where $\beta, \gamma \in P \backslash W$.

\np
Assume further that $w \in R^+$. Then one of $\beta$ and $\gamma$ is in $R^+$ and the other one is in $R^-$. Indeed,
they cannot both be in $R^-$ because $w \in R^+$. If both $\beta$ and $\gamma$ are in $R^+$, since $\beta + \gamma \in W = \Span I$,
we would conclude that $\beta, \gamma \in W$ which contradict the way they were chosen. Say, $\beta \in R^+$ and $\gamma \in R^-$.
Then $\beta, -\gamma \in R^+\backslash W$, $\beta \in \Phi$, and $\beta - (-\gamma) = w \in W$ imply that $-\gamma \in \Phi$.
The last then implies that $- \gamma \in P$ which contradicts the assumption that $\gamma \in P$. This completes the proof 
in the case when $w \in R^+$. The case when $w \in R^-$ is dealt in the same way.
\end{proof}

\section{\texorpdfstring{The graph $\GPhi$}{The graph of G(Phi)}} \label{sec: graphs}

\np 
For the rest of the paper, the symbol $\Phi$ will denote an inversion set of $R$, and in particular, a 
subset of $R^+$. We recall that by $\Phi^c$ we mean the complement of $\Phi$ in $R^+$.

\np 

While Section~\ref{sec: paths} deals with the existence and properties of arbitrary paths between roots in $R$, 
when discussing an inversion set $\Phi$, it is often useful to focus on paths
whose partial sums belong to $\Phi$ and steps belong to $\pm \Phi^c$. To that end, 
we introduce the graph $\GPhi$ and study a partial order on its connected components,
which resembles the partial order on $R^+$. 
In fact, if $\Phi = R^+$, then each component of $\GPhi$ consists of a single 
element of $\Phi$ and thus $\Comp$ is identified with $R^+$ and the partial orders on these 
two sets are the same.

\subsection{Definition of $\GPhi$}
\begin{definition}
Let $\Phi \subseteq R^+$ be an inversion set. We define $\GPhi$ to be the graph whose vertices are the elements of $\Phi$, 
where two vertices $\alpha$ and $\alpha'$ are connected by an edge if $\alpha - \alpha' \in \pm \Phi^c$.

\np 
Note that in this definition we allow for $R$ to be reducible. It is not difficult to see that if 
$R = R^1 \times \dots \times R^s$
is a reducible QRS 
and $\Phi \subseteq R^+$, then $\Phi = \Phi^1 \sqcup \dots \sqcup \Phi^s$, where $\Phi^i := \Phi \cap (R^i)^+$
and $\GPhi$ is the union of $G(\Phi^1), \dots, G(\Phi^s)$ in the sense that its sets of vertices and edges are the 
(disjoint) unions of the respective sets of vertices and edges of $G(\Phi^1), \dots, G(\Phi^s)$.

\np 
We will be mainly interested in the connected components of the graph $\GPhi$ as opposed to 
the graph itself.  We denote by $\Comp$ the set of connected components
of $\GPhi$.  In the case when $\Phi = R^+$, the components are the singleton sets, each containing 
a single positive root.
\end{definition}

\begin{prop} \label{prop: decompositions don't split components} Let $\Phi$ be an inversion set and let $C \in \Comp$.
 If $\Phi = \Phi_1 \sqcup \Phi_2$, then $C \subseteq \Phi_1$ or $C \subseteq \Phi_2$.
    In particular, if $\Phi = \inf_I^S(\Psi, X)$, then $C \subseteq \inf_I^S(\Psi, \emptyset)$ or $C \subseteq X$.
\end{prop}

\begin{proof}
    Suppose to the contrary that neither $C \cap \Phi_1$ nor $C \cap \Phi_2$ is empty. By definition 
    we would have $\alpha + \kappa = \beta$ for some $\alpha \in \Phi_1, \beta \in \Phi_2,$ 
    and $\kappa \in \pm \Phi^c$. By symmetry, we may assume that $\kappa \in R^+$. 
    Then $\alpha, \kappa \in \Phi_2^c$, which
    contradicts the co-closure of $\Phi_2$.
\end{proof}

\np 
  When we have expressed an inversion set $\Phi=\inf_I^S(\Psi,X)$ as an inflation, 
  we usually will denote components contained in $X$ by $X_1, X_2, \dots$ and 
   components outside of $X$ by
  $Z_1, Z_2, \dots$.

\begin{example}\label{ex: component examples} 
We give three examples of inversion sets and their respective sets of components. 
As we introduce more structure on the set $\Comp$, we will revisit them. 
For each inversion set we describe its canonical inflation $\Phi=\inf_I^S(\Psi,X)$.

\begin{enumerate} 
\item[(i)] 
Here is an inversion set $\Phi$ in $\mathbb{E}_6$ which is $\Phi=\inf_I^S(\Psi,X)$ where\\
$I = \{\theta_2,\theta_4,\theta_5,\theta_6\}$,
$\Psi = \mathbb{C}_2^+ =\{10,01,11,12\}$ 
(which is not primitive) and\\
$X=\left\{\Esix000100,\Esix000001,
\Esix010100,\Esix000111\right\}$.\\
One can check that this choice of $I$ is minimal and thus, by
Theorem~\ref{theorem: canonical inflation}, the inflation given is canonical.
The graph $\GPhi$ has five components of which the
first four are outside $X$.  Here 
\begingroup
\setlength\arraycolsep{3pt}
$$    \Phi = \left\{ \begin{array}{cccccccccc}
\Esix100000,
&\Esix001000,
&\Esix000100,
&\Esix000001,
&\Esix101000,
&\Esix010100,
&\Esix001100,
&\Esix101100,
&\Esix011100,
&\Esix001110,\\
\Esix000111,
&\Esix111100,
&\Esix101110,
&\Esix011110,
&\Esix001111,
&\Esix111110,
&\Esix101111,
&\Esix011210,
&\Esix011111,
&\Esix111210,\\
\Esix111111,
&\Esix011211,
&\Esix112210,
&\Esix111211,
&\Esix011221,
&\Esix112211,
&\Esix111221,
&\Esix112221,
&\Esix112321,
&\Esix122321
 \end{array}
\right\}\ .$$
\endgroup

\np 
The first four components are: \\
$Z_1=\left\{
\Esix011211,
\Esix001000,
\Esix011110,
\Esix011221,
\Esix001111,
\Esix001100,
\Esix011210,
\Esix011111,
\Esix011100,
\Esix001110\right\}$,\\
$Z_2=\left\{\Esix100000\right\}$,\\
$Z_3 = \left\{
\Esix101111,
\Esix111210,
\Esix111111,
\Esix101000,
\Esix111211,
\Esix101100,
\Esix111221,
\Esix111100,
\Esix101110,
\Esix111110\right\}$, and\\
$Z_4 = \left\{\Esix112210,
\Esix112211,
\Esix112221,
\Esix112321,
\Esix122321\right\}$\ .\\
The fifth component is $X_1=X$. 

\item[] 
 
\item[(ii)] 
Here is an inversion set $\Phi$ in $\mathbb{E}_6$ which is   $\Phi=\inf_I^S(\Psi,X)$ where\\
$I = \{\theta_1,\theta_3,\theta_4,\theta_5\}$,
$\Psi=\mathbb{B}_2^+=\{10,01,11,21\}$
(which is not primitive) and\\
$X=\Big\{
\Esix000100,
\Esix000010,
\Esix000110\Big\}$.\\
As in the previous case, one can check that the chosen $I$ is minimal and thus this 
expression of $\Phi$ as an inflation is the canonical one.
Here $\Comp$ has seven elements of which the
first four are outside $X$.  Here 
\begingroup
\setlength\arraycolsep{3pt}
$$    \Phi = \left\{ \begin{array}{cccccccccc}
\Esix010000,
&\Esix000100,
&\Esix000010,
&\Esix000001,
&\Esix010100,
&\Esix000110,
&\Esix000011,
&\Esix011100,
&\Esix010110,
&\Esix000111,\\
\Esix111100,
&\Esix011110,
&\Esix010111,
&\Esix001111,
&\Esix111110,
&\Esix101111,
&\Esix011210,
&\Esix011111,
&\Esix111210,
&\Esix111111,\\
\Esix011211,
&\Esix112210,
&\Esix111211,
&\Esix011221,
&\Esix112211,
&\Esix111221,
&\Esix112221,
&\Esix112321,
&\Esix122321
 \end{array}
\right\}\ .$$
\endgroup

\np 
The first four components are:\\
$Z_1=\left\{\Esix011210,
\Esix010000,
\Esix111210,
\Esix010100,
\Esix112210,\\
\Esix011100,
\Esix010110,
\Esix111100,
\Esix011110,
\Esix111110\right\}$,\\
$Z_2=\left\{\Esix001111,
\Esix000001,
\Esix101111,
\Esix000011,
\Esix000111\right\}$,\\
$Z_3=\left\{\Esix111211,
\Esix010111,
\Esix011221,
\Esix112211,
\Esix111221,
\Esix112221,
\Esix112321,
\Esix011111,
\Esix111111,
\Esix011211\right\}$, and\\
$Z_4=\left\{\Esix122321\right\}$.

\np 
The last three components are\\
$X_1 =\left\{\Esix000100\right\}$,
$X_2=\left\{\Esix000010\right\}$,
$X_3=\left\{\Esix000110\right\}$
where
$X=X_1 \sqcup X_2 \sqcup X_3$.

\item[] 

\item[(iii)] 
Here is an inversion set $\Phi$ in $\mathbb{B}_5$ which is   $\Phi=\inf_I^S(\Psi,X)$ where 
$I = \{\theta_4,\theta_5\}$,
$\Psi=\{ 100, 001, 111, 012, 112 \}$,
which is primitive, and $X=\{ 00010, 00001, 00011, 00012 \}$.\\
The graph $\GPhi$ has six components of which the
first two are outside $X$.  Here 
\begingroup
\setlength\arraycolsep{3pt}
$$ \Phi = \left\{ \begin{array}{ccccccccc}
 10000,
& 00100,
& 00010,
& 00001,
& 00110,
& 00011,
& 11100,
& 00111,
& 00012,\\
  11110,
& 00112,
& 11111,
& 00122,
& 11112,
& 11122,
& 01222,
& 11222
 \end{array}
\right\}.$$
\endgroup

\np 
The first two components are:\\
$Z_1=\left\{
10000,
00100,
00110,
11100,
00111,
11110,\\
00112,
11111,
00122,
11112,
11122,
01222
\right\}$, and\\
$Z_2=\{11222\}$.

\np 
The last four components are\\
$X_1=\{00001 \}$,
$X_2=\{00010\}$
$X_3=\{00011 \}$,
$X_4=\{00012\}$
where
$X=X_1 \sqcup X_2 \sqcup X_3 \sqcup X_4$. 
 \end{enumerate}
\end{example}

\subsection{Partial order of components.}

Proposition~\ref{prop: decompositions don't split components} suggests that when studying 
decompositions of an inversion set $\Phi$, the relevant objects to examine are not arbitrary 
subsets of $\Phi$, but rather collections of components of $\GPhi$. For that reason, 
understanding the structure of the set $\Comp$ is the focus of the remainder of this and the following section.

\np 
We begin by defining a partial order on the set $\Comp$. First, we need the 
following lemma, which will also be useful for defining component addition in Section~\ref{sec: addition}.

\begin{lemma}
\label{lemma: If alpha + beta in C, alpha' + beta' in C}
    Let $\Phi \subseteq R^+$ and $A, B, C\in \Comp$ with $\alpha \in A, \beta \in B$, and $\alpha + \beta \in C$.

    \begin{enumerate}[(i)]
        \item If $A \neq C$, then for all $\alpha' \in A$ there exists $\beta' \in B$ such that $\alpha' + \beta' \in C$.
        \item If $A\neq C$ and $B \neq C$, then for all $\gamma' \in C$, 
        there exist $\alpha' \in A$ and  $\beta' \in B$ such that $\gamma' = \alpha' + \beta'$.
    \end{enumerate}
\end{lemma}

\begin{proof}
  First we  prove (i).  It suffices to verify (i) for every neighbour $\alpha' \in A$ of $\alpha$.  Write 
$\alpha' = \alpha - \kappa$ where $\kappa \in \pm \Phi^c$.  
Then $\alpha'+\kappa + \beta = \alpha+\beta$.  
Since the right hand side is a root, either $\alpha' + \beta \in R$ or 
$\beta + \kappa \in R$ by Proposition~\ref{proposition: 2 of 3 rule}. 
If $\alpha' + \beta$ is a root then it lies in $\Phi$ by the closure of 
$\Phi$.  Since $(\alpha' + \beta) - (\alpha+\beta) = -\kappa \in \pm\Phi^c$
we see that $\alpha' + \beta \in C$ verifying (i).
In the second case, suppose that $\beta + \kappa \in R$.  
If $\beta + \kappa \in \pm\Phi^c$ then $(\alpha+\beta) - \alpha' \in \pm\Phi^c$
and thus $C=A$.  
This contradiction shows that $\beta + \kappa \notin \pm\Phi^c$.
If $\beta + \kappa \in -\Phi$ then 
$-\kappa =\beta + (-\beta-\kappa) \in \Phi$.  
This contradiction shows that $\beta + \kappa \notin -\Phi$.
Therefore $\beta + \kappa \in \Phi$ and thus $\beta + \kappa \in B$.
Then $\alpha'+(\beta+\kappa) = \alpha+\beta \in C$ again verifying (i).

\np 
  To prove (ii), it suffices show that (ii) holds for every neighbour 
$\gamma' \in C$ of $\gamma:=\alpha+\beta$.
Write $\gamma' = \gamma + \kappa$ where
$\kappa \in \pm \Phi^c$.  Consider the equation
$\alpha+\beta+\kappa = \gamma'$.  Since the right hand side is a root either
$\alpha+\kappa \in R$ or $\beta+\kappa\in R$ again by Proposition~\ref{proposition: 2 of 3 rule}.  Without loss of generality we
may suppose that $\alpha+\kappa \in R$.
If $\alpha + \kappa \in \pm\Phi^c$ then 
$\gamma' - \beta = \alpha + \kappa \in \pm\Phi^c$
and thus $C=B$.  
This contradiction shows that $\alpha + \kappa \notin \pm\Phi^c$.
If $\alpha + \kappa \in -\Phi$ then 
$-\kappa =\alpha + (-\alpha - \kappa) \in \Phi$.  
This contradiction shows that $\alpha + \kappa \notin -\Phi$.
Therefore $\alpha + \kappa \in \Phi$ and thus $\alpha + \kappa \in A$.
Then $(\alpha + \kappa) + \beta = \gamma'$ as required.
\end{proof}

\begin{remark}
    Note that in Lemma~\ref{lemma: If alpha + beta in C, alpha' + beta' in C}, the assumptions that $A \ne C$ in part (i) and
    $A \ne C$, $B \ne C$ in part (ii) are necessary. For instance, in 
    Example~\ref{ex: component examples}~(i), if we take $A = Z_1, B = X_1, C= Z_1$, 
    then we can find roots $\alpha, \beta$ with $\alpha \in A, \beta \in B$ and $\alpha + \beta \in C$. 
    However, taking $\alpha'$ maximal in $A$, there is no $\beta' \in B$ with $\alpha' + \beta' \in C$.  
    Similarly, in the same example, if we take $\gamma'$ minimal in $C$ then it is not possible 
    to write $\gamma' = \alpha' + \beta'$ with $\alpha' \in A$ and $\beta' \in B$. 
\end{remark}

 \np
Next we introduce a partial order on $\Comp$ inherited from the partial order on $R^+$ from 
Definition~\ref{Order on roots}.

\begin{definition}
    Let $A, B\in \Comp$. We write $A \leq B$ if there exist $\alpha \in A$, $\beta \in B$ 
    with $\alpha \leq \beta$. As usual, $A< B$ means $A \leq B$ and $A \neq B$.
\end{definition}

\begin{prop}\label{prop: partial order equivalences}
        Let $A$ and $B$ be two distinct elements of $\Comp$.
        Then the following conditions are all equivalent:
        \begin{enumerate}[(i)]
            \item $A < B$.
            \item there exist components $A = C_0, C_1, \dots, C_m= B$ and a sequence $\gamma_0,\gamma_1,\dots,\gamma_m$ with 
        $\gamma_i \in C_i$ for all $i$  and $\gamma_{i+1} - \gamma_i \in R^+$.
        \item there exist components 
        $A = D_0 < D_1 <  \dots < D_n= B$ and for every $\alpha \in A$, a sequence 
        $ \alpha = \delta_0,\delta_1,\dots,\delta_n$ with 
        $\delta_i \in D_i$ for all $i$, and 
     $\delta_{i+1} - \delta_i \in \Phi$.
     \item for every $\alpha \in A$ there exists $\beta \in B$ with
     $\alpha < \beta$.
        \end{enumerate}
\end{prop}

\begin{proof} Assume $A < B$, i.e., $A\neq B$ and there exists
$\alpha \in A$ and $\beta \in B$ with $\alpha < \beta$. We will show (ii).
        Then, $\alpha + \sum_{i = 1}^n \theta_i = \beta$ for $\theta_i \in S$. 
        We then form a reduced path $[\alpha; \kappa_1, \kappa_2, \dots, \kappa_m; \beta]$ 
        from $\alpha$ to $\beta$ with $\kappa_i > 0$. By 
        Proposition~\ref{prop: any rearrangement of a reduced path is a reduced path}, 
        we may rearrange this path so that the $\kappa_i$ in $\Phi$ come first and those 
        in $\Phi^c$ after. Write $\gamma_i = \alpha + \kappa_1 + \dots \kappa_i$ 
        for $1 \le i \le m$, and define $\gamma_0 = \alpha$. We assume that 
        $\kappa_1, \dots, \kappa_s \in \Phi$ and $\kappa_{s+1}, \dots, \kappa_m \not \in \Phi$ 
        where $0 \leq s \leq m$. Our convention is that when $s=0$, no $\kappa_i$ lie in $\Phi$ 
        and when $s=m$, no $\kappa_i$ lie in $\Phi^c$.

    \np 
        We claim that all $\gamma_i$ are in $\Phi$. 
        If $\gamma_{i} \in \Phi$ and $\kappa_{i+1} \in \Phi$, then $\gamma_{i+1} \in \Phi$ by closure. 
        Using that $\alpha = \gamma_0 \in \Phi$, this shows that $\gamma_1, \dots, \gamma_s \in \Phi$. 
        Now, assume that for some $t > s$, we have that $\gamma_t \not\in \Phi$. By co-closure, 
        that implies that $\gamma_{j} \not \in \Phi$ for all $j \geq t$. In particular, 
        $\beta = \gamma_m \not \in \Phi$, which is a contradiction.  This proves our claim.

        \np 
         Thus, define $C_i$ to be the component containing $\gamma_i$. 
         Since $\gamma_{i + 1} - \gamma_i = \kappa_{i + 1} > 0$ for all $i$, this sequence 
         satisfies the desired conditions. This completes the proof of (ii).

\np 
    Now assume that (ii) is true. We will show that (iii) follows. 
    Let $D_0 = C_0 = A$, then set $D_1$ to be the next $C_i$ in the sequence which is different 
    from $D_0$, carrying on in this manner gives a sequence $A = D_0 < D_1 < \dots < D_n = B$ of components. 
    Take $\alpha \in A$, and set $\delta_0 = \alpha \in D_0$. 
    Now suppose we have a sequence of $\delta_i \in D_i$ which satisfy the result for all $i \leq k$. 
    By the way we defined $D_{k+1}$, there is a $C_i = D_k$, such that $C_{i+1} = D_{k+1}$, thus, $\gamma_{i+1}-\gamma_i \in R^+$. 
    Since the components $D_k$ and $D_{k+1}$ are distinct, $\gamma_{i+1}-\gamma_i$ cannot be in $\Phi^c$, 
    thus, we can use Lemma~\ref{lemma: If alpha + beta in C, alpha' + beta' in C}~(i) to 
    find a $\rho_{k+1} \in \Phi$ such that $\delta_{k+1} := \delta_k+\rho_{k+1} \in D_{k+1}$. 
    This proves (iii).

\np 
Now, assume that (iii) holds. Let $\alpha \in A$. Statement (iii) implies the existence of a 
sequence $\alpha=\delta_0, \delta_1, \dots, \delta_n$ such that $\delta_n \in B$ 
and $\delta_i < \delta_{i+1}$ for all $i$. In particular, $\alpha = \delta_0 < \dots <\delta_n$, 
which proves (iv). Clearly, (iv) implies (i).
\end{proof}

\np 
As a consequence of the equivalence of (i) and (iv) in Proposition~\ref{prop: partial order equivalences}, we have the following.

\begin{cor}\label{cor: A <= B => supp A subset of supp B}
Let $A, B\in \Comp$ with $A \leq B$. Then $\supp A \subseteq \supp B$. \qed
\end{cor}

\np 
We are now ready to prove that the relation $\leq$ is a partial order.

\begin{proposition} \label{prop:partial order on components is an order}
    The relation $\leq$ is a partial order on the set $\Comp$ .
\end{proposition}

\begin{proof}
    Reflexivity is clear. Transitivity follows from part~(iv) of 
    Proposition~\ref{prop: partial order equivalences}. To see antisymmetry, 
    assume by way of contradiction that $A < B$ and $B < A$. Let $\alpha \in A$ be a 
    maximal root in $A$. Then, by part~(iv) of Proposition~\ref{prop: partial order equivalences}, 
    there exists $\beta \in B$ satisfying $\alpha < \beta$ and some $\alpha' \in A$ 
    satisfying $\beta < \alpha'$. Together these imply that $\alpha < \alpha'$, contradicting maximality of $\alpha$.
\end{proof}

\begin{example}
    \label{ex: partial order on components}
    Below are Hasse diagrams that correspond, from left to right,
     to the posets of components for the inversion sets in 
    Example~\ref{ex: component examples} (i), (ii) and (iii).

\[\begin{tikzcd}[cramped]
	&&&&& {Z_4} &&&& {Z_2} \\
	& {Z_4} &&&& {Z_3} &&&& {Z_1} \\
	& {Z_3} &&& {Z_1} && {Z_2} &&& {X_4} \\
	{Z_1} && {Z_2} &&& {X_3} &&&& {X_3} \\
	& {X_1} &&& {X_1} && {X_2} && {X_1} && {X_2}
	\arrow[no head, from=5-2, to=4-3]
	\arrow[no head, from=5-2, to=4-1]
	\arrow[no head, from=4-1, to=3-2]
	\arrow[no head, from=4-3, to=3-2]
	\arrow[no head, from=3-2, to=2-2]
	\arrow[no head, from=5-5, to=4-6]
	\arrow[no head, from=5-7, to=4-6]
	\arrow[no head, from=4-6, to=3-7]
	\arrow[no head, from=4-6, to=3-5]
	\arrow[no head, from=3-5, to=2-6]
	\arrow[no head, from=3-7, to=2-6]
	\arrow[no head, from=2-6, to=1-6]
	\arrow[no head, from=5-9, to=4-10]
	\arrow[no head, from=5-11, to=4-10]
	\arrow[no head, from=4-10, to=3-10]
	\arrow[no head, from=3-10, to=2-10]
	\arrow[no head, from=2-10, to=1-10]
\end{tikzcd}\]
\end{example}

\section{Addition of components} \label{sec: addition}

\np In this section we introduce a partial addition of elements
of $\Comp$ which is compatible with the partial order discussed above
and exhibits many properties similar to the properties of addition in a QRS $R$.
When $\Comp=R^+$, i.e., when $\Phi = R^+$,  
the partial addition on $\Comp$ coincides with the (partial) addition on $R^+$.

\np
There are, however, some substantial differences between addition in $R$ and partial addition in $\Comp$.
If $R$ is a QRS, then addition in $R$ is the restriction of the addition in the ``root lattice of $R$'', i.e., 
the lattice generated by $R$. As a consequence, for any $\alpha, \beta \in R$, $\alpha + \beta$ is a well-defined 
element of the aforementioned lattice, which may or may not belong to $R$. Since addition in the root lattice is
commutative and associative and $R = - R$, addition in $R$ is commutative and associative and we only need to worry whether 
the elements involved in an identity like $(\alpha + \beta) + \gamma = \alpha + (\beta + \gamma)$ belong to $R$ or not.
In contrast, the partial 
addition in $\Comp$ does not extend naturally to addition in a lattice generated by $\Comp$. As a consequence, 
we need to study carefully properties such as commutativity, associativity, cancellation, etc. for addition in $\Comp$.

\np
By the very definition, the partial addition is commutative: $A+B$ is defined if and only if $B+A$ is defined and in 
that case they are equal. Associativity is more delicate:
the existence of $A+B$ and $(A+B)+C$ does not imply the existence of $(A+C) + B$. Nevertheless, we show that
if $A+B$ and $(A+B)+C$ exist, then at least one of $(A+C) + B$ or $(B+C) + A$ exists and equals $(A+B)+C$.
This property is the ``component'' analog of the two-out-of-three rule for addition in a QRS.
Similarly, cancellation rules do not apply to the partial addition in $\Comp$, i.e., 
the equality $A+ B = A + B'$ does not necessarily imply $B = B'$. However, the cancellation rule holds  
under the condition that $A+ B = A + B' \neq A$, see Proposition~\ref{prop: first cancellation rule for component addition}. 

\np 
Apart from discussing the questions stated above, we introduce the notion of a simple component. 
This a component that is not the sum of two other components or, equivalently, a component that contains 
a simple root of $R$. We then show that the properties of simple components resemble those of simple roots.

\np 
We complete the section by studying components of the form $kA$ for $A\in \Comp$ and $k \in \ZZ_{>0}$.

\subsection{Definition and basic properties}

\np 
We are now ready to define the addition on $\Comp$.

\begin{definition} 
Consider two elements $A, B \in \Comp$.
If there exist $\alpha \in A$ and $\beta \in B$ such that
$\alpha+\beta \in R$ then we define $A+B=C$ where $C$ is the element 
of $\Comp$ containing $\alpha+\beta$.  If no such pair
$(\alpha,\beta) \in A \times B$ exists then the sum $A+B$ is undefined.

\np 
Given $k \geq 3$ and $A_1, A_2, \dots, A_k \in \Comp$, we define inductively
$A_1 + A_2 + \dots + A_k$ as the component $(A_1 + A_2+\dots + A_{k-1}) + A_k$
if both $A_1 + A_2+\dots + A_{k-1}$ and $(A_1 + A_2+\dots + A_{k-1}) + A_k$ are defined
and  say that $A_1 + A_2 + \dots + A_k$ is undefined otherwise.
\end{definition}

\begin{remark} 
    \label{remark: additive characterization of partial order}
    \begin{enumerate}[(i)]
    \item It is not obvious that when $A+B$ is defined it is uniquely determined.  This is established by 
    Proposition~\ref{prop: addition works}. 
    \item It follows from 
    Proposition~\ref{prop: partial order equivalences}~(iii) that $A \leq B$ if 
    and only if there exists a possibly empty sequence of components
    $C_1,C_2,\dots,C_n$ such that $B = A + C_1 + C_2 + \dots + C_n$.
    \item The sum $A_1 + A_2 + \dots + A_k$ may be thought of as an extension of the notion of paths in QRSs. 
Indeed, $[\alpha; \kappa_{1}, \kappa_{2}, \dots, \kappa_{n}; \beta]$ is a path in $R$ if each of the partial sums 
$\alpha, \alpha + \kappa_1, \dots, \alpha + \kappa_1 + \dots + \kappa_n$ belongs to $R$ and 
$\beta = \alpha + \kappa_1 + \dots + \kappa_n$. In other words, $B = A + K_1 + K_2 + \dots + K_n$ for components 
is the counterpart of the path $[\alpha; \kappa_{1}, \kappa_{2}, \dots, \kappa_{n}; \beta]$. 
    \end{enumerate}
\end{remark}

\begin{prop}\label{prop: addition works}
Suppose that $A_1,A_2,\dots,A_n$ and $C$ 
are elements of $\Comp$.
    Let $\alpha_i, \alpha'_i \in A_i$ for $1 \leq i \leq n$ with $\sum_{i=1}^n \alpha_i \in C$ and 
    $\sum_{i=1}^n\alpha_i'  \in R$. Then $\sum_{i=1}^n\alpha_i' \in C$.
\end{prop}

\begin{proof}
It follows from Proposition~\ref{proposition: existence of paths} that 
the expression $\sum_{i=1}^n\alpha_i' \in R$ gives us a path 
$[\alpha_1'; \kappa_1, \ldots, \kappa_s, \sum_{i=1}^n\alpha_i']$ where each 
$\kappa_i$ is one of the $\alpha_j'$. By closure of $\Phi$, all partial 
sums of that path are in $\Phi$. In particular, $\sum_{i=1}^n\alpha_i' \in \Phi$.
 Let $C'$ denote the component of $\GPhi$ containing 
$\sum_{i=1}^n\alpha'_i$.  Then $C \geq A_i$ and $C' \geq A_i$ for all $i$.  If $C=A_j$ and $C'=A_k$
then $A_j \geq A_k$ and $A_k \geq A_j$ which implies $C=A_j=A_k=C'$.   Hence (interchanging $C$ and $C'$ 
if necessary) we may suppose, without loss of generality, that $C>A_i$ for all $i$.

\np 
For each $i$ with $1 \leq i \leq n$ there exists a path
$[\alpha_i; \kappa_{i,1},\kappa_{i,2}, \dots, \kappa_{i,k_i}; \alpha_i']$ in $A_i$ with
$\kappa_{i,j} \in \pm \Phi^c$ for all $i,j$. By Proposition~\ref{proposition: existence of paths}, there exists a path 
$[\sum_{i=1}^n \alpha_i; \nu_1, \nu_2, \dots, \nu_m; \sum_{i=1}^n \alpha_i']$ where each $\nu_k = \kappa_{i,j}$
for some $i,j$.   We claim that all partial sums along this path are in $\Phi$ (and therefore in $C$). It is clear that this 
path begins in $\Phi$, since $\sum_{i=1}^n \alpha_i \in C$.  Suppose that 
$\rho_k=\sum_{i=1}^n \alpha_i + \sum_{j=1}^k \nu_j \in \Phi$ is a partial sum along the path.
 We claim that the next partial sum $\rho_k + \nu_{k + 1}$ is also in $\Phi$.
 Note that $\rho_k + \nu_{k+ 1} = \rho_k + \kappa_{i,j}$ for some $i,j$.
  Since $\nu_{k+1}=\kappa_{i,j}$ appears as a step on a path in $A_i$, we may write 
  $\nu_{k+1} = \gamma - \gamma'$ where  $\gamma, \gamma' \in A_i$. 
  Thus, $\rho_k + \gamma - \gamma' \in R$. Observe that $\gamma \in R^+$ implies that
 $\rho_k + \gamma \ne 0$. Also, if $\rho_k - \gamma'=0$, then $\rho_k \in A_i$, which is 
 impossible since $\rho_k$ lies in a 
 component that is strictly greater than $A_i$. Hence, by Proposition~\ref{proposition: 2 of 3 rule} 
 either $\rho_k + \gamma \in R$ 
 or $\rho_k - \gamma' \in R$.

    \np 
If $\rho_k + \gamma \in R$, then $\rho_k + \gamma \in \Phi$ by the closure of $\Phi$. 
But then $\rho_k + \gamma - \gamma'$ cannot be in $R^-$, since if it were then $\gamma' > \rho_k + \gamma$, \
which is a contradiction because 
$\gamma' \in A_i$ and $\rho_k \in C$ and $A_i < C$.
The root $\rho_k + \gamma - \gamma'$ also cannot be in $\Phi^c$, since then $\rho_k + \gamma$ and $\gamma'$ 
would belong to the same component, which again yields a contradiction in the order
of components. Thus, in this case, $\rho_k + \gamma - \gamma' = \rho_k + \nu_{k + 1}$ must lie in $\Phi$. 

      \np 
If instead $\rho_k - \gamma'$ is a root, it cannot be in $R^-$ since that would imply $\rho_k < \gamma'$ 
contradicting the fact that $C > A_i$.  The root $\rho_k - \gamma'$ cannot be in $\Phi^c$, 
as otherwise $\rho_k$ and $\gamma'$ would lie in the same component. Therefore $\rho_k - \gamma' \in \Phi.$
Then the closure of $\Phi$ implies that 
$\rho_k + \nu_{k+1} = \rho + \gamma - \gamma' \in \Phi$, 
which in turn implies that the next partial sum $\rho_k+\nu_{k+1} \in \Phi$.

\np 
Hence we see that the entire path $[\sum_{i=1}^n \alpha_i; \nu_1, \nu_2, \dots, \nu_m; \sum_{i=1}^n \alpha_i']$ 
 remains in $C$. 
\end{proof}

\begin{proposition}\label{prop: component 2 of 3 rule}
    Let $A,B,C\in \Comp$ such that $(A+B)+C$ is defined. Then at least one of the two 
    sums $B+(A+C)$ or $A+(B+C)$ is defined and is equal to $(A+B)+C$. When all three sums are defined, they are all equal.
\end{proposition}

\begin{proof}
    We begin by proving the first assertion.  Let $D=A+B$.

\np 
    First we consider the situation when $(A+B)+C \in \{A,B\}$,
    say $(A+B)+C=A$.   Then $A=D=A+B$ since $A \leq A+B \leq A$.
    Furthermore $D+C = A + C$ is defined  and 
    $(A+B)+C = A+C = A = A+B =(A+C)+B$. 
    The case where $(A+B) + C = B$ is entirely similar.

    \np 
    Suppose then that $(A+B)+C > A$ and $(A+B)+C > B$. 
    Then either $D + C > D$ or both $D > A$ and $D > B$. We show that there exist 
    $\alpha \in A, \beta \in B$, and $\gamma \in C$ such that $\alpha + \beta + \gamma$ is a root. 
    If $D + C \neq D$, then we pick $\alpha \in A$ and $\beta \in B$ such that $\alpha + \beta \in D$. 
    Then, applying Lemma~\ref{lemma: If alpha + beta in C, alpha' + beta' in C}~(i), there exists 
    a $\gamma \in C$ such that $\alpha + \beta + \gamma \in (A + B) + C$. If $D\neq A, B$, 
    we pick $\delta \in D$ and $\gamma \in C$ such that $\delta + \gamma \in (A + B) + C$. 
    By Lemma~\ref{lemma: If alpha + beta in C, alpha' + beta' in C}~(ii), $\delta = \alpha + \beta$ 
    for some $\alpha \in A$, $\beta \in B$, so that $\alpha + \beta + \gamma \in (A + B) + C$. 

      \np 
    Then applying Proposition~\ref{proposition: 2 of 3 rule} we see that
    one of $\alpha+\gamma$ or $\beta + \gamma$ is a root.  Hence either
    $A+C$ or $B+C$ is defined. Moreover, because 
    $(\alpha + \gamma) + \beta = (\beta + \gamma) + \alpha = \alpha + \beta + \gamma$, 
    it follows that $(A + B) + C = (A + C) + B$ when $A + C$ is defined, and $(A + B) + C = (B + C) + A$ 
    when $B + C$ is defined.

\np 
    Finally we prove the last statement of the proposition.  Suppose then that say $(A+B)+C=(A+C)+B$ and that 
    $(B+C) + A$ is defined.  Applying the first assertion of 
    this proposition 
    to the sum $(B+C)+A$ we see that this sum is equal to
    either $(A+B)+C$ or to $(A+C)+B$.  Thus all three sums
    are equal in this case.
\end{proof}

\np 
Clearly, Proposition~\ref{prop: component 2 of 3 rule} is analogous to Proposition~\ref{proposition: 2 of 3 rule}.  
In the same way, the following proposition 
for components is the analog of Proposition~\ref{prop: any rearrangement of a reduced path is a reduced path} 
for roots.  The proof is also analogous to the proof of 
Proposition~\ref{prop: any rearrangement of a reduced path is a reduced path}, and we omit it.

\begin{proposition} \label{xxx}
Assume that the sum $A + K_1 + \dots + K_n$ is defined and $K_i + K_j$ is  undefined for all $1 \leq i \neq j \leq n$.
Then, for any permutation $i_1, i_2, \dots, i_n$ of $1,2, \dots, n$, the sum 
$A + K_{i_1} + \dots + K_{i_n}$ is defined and equal to $A + K_1 + \dots + K_n$. \hfill $\square$
\end{proposition}

\begin{example}
\label{ex: addition table examples}
We give the addition tables for the components corresponding to the inversion sets described in
Example~\ref{ex: component examples}.
\begin{enumerate}
\item[(i)] \label{ex: first addition table example}
The addition table for the inversion set of 
Example~\ref{ex: component examples}\,(i).
$$\begin{array}{l|ccccc}
&Z_1&Z_2&Z_3&Z_4&X_1\\
\hline
Z_1&-& Z_3& Z_4& -& Z_1\\
Z_2& Z_3& -& -& -& -\\
Z_3&Z_4& -& -& -& Z_3\\
Z_4&-& -& -& -& Z_4\\
X_1&Z_1& -& Z_3& Z_4& -
\end{array}$$

\np
Note that $(Z_1+Z_2)+X_1 = Z_3 + X_1  = Z_3$ while 
$Z_2+X_1$ is not defined. In particular, 
$(Z_1+Z_2)+X_1 = Z_1+(Z_2+X_1)$ does not hold.

\np 
\item[(ii)]\label{ex: second addition table example}
 The addition table for the inversion set of
Example~\ref{ex: component examples}\,(ii). 

$$\begin{array}{l|ccccccc}
&Z_1&Z_2&Z_3&Z_4&X_1&X_2&X_3\\
\hline
Z_1&-& Z_3& Z_4& -& Z_1& Z_1& Z_1\\  
Z_2&Z_3& -& -& -& Z_2& Z_2& Z_2\\
Z_3&Z_4& -& -& -& Z_3& Z_3& Z_3\\ 
Z_4&-& -& -& -& -& -& -\\
X_1&Z_1& Z_2& Z_3& -& -& X_3& -\\
X_2&Z_1& Z_2& Z_3& -& X_3& -& -\\
X_3&Z_1& Z_2& Z_3& -& -& -& -\\
\end{array}$$

\np
Note that $Z_1 + X_1 = Z_1 + X_2 = Z_1 + X_3 = Z_1$
but $X_1, X_2,$ and $X_3$ are pairwise distinct.

\np 
\item[(iii)] \label{ex: third addition table example}
 The addition table for the inversion set of
Example~\ref{ex: component examples}\,(iii).
$$\begin{array}{l|cccccc}
&Z_1&Z_2&X_1&X_2&X_3&X_4\\
\hline
Z_1& Z_2& -& Z_1& Z_1& Z_1& Z_1\\  
Z_2& -& -& -& -& -& -\\
X_1& Z_1& -& -& X_3& X_4& -\\ 
X_2& Z_1& -& X_3& -& -& -\\
X_3& Z_1& -& X_4& -& -& -\\
X_4& Z_1& -& -& -& -& -\\
\end{array}$$
\end{enumerate}
\end{example}

\subsection{Simple components}

We can define a concept analogous to simple roots for components as well.

\begin{proposition} \label{Simple Comp. equiv.}
    Let $C\in \Comp$. The following are equivalent.
    \begin{enumerate}
        \item[(i)] The component $C$ contains a simple root.
        \item[(ii)] Whenever $C = A+B$ for two components $A,B\in\Comp$, either $A=C$ or $B=C$.
    \end{enumerate}
\end{proposition}

\begin{definition}
\label{def: simple components}
    A component $C$ is {\em simple} if either (i) or (ii) of Proposition~\ref{Simple Comp. equiv.} holds. 
\end{definition}

\begin{proof}[Proof of Proposition~\ref{Simple Comp. equiv.}]
   \textit{(i) $\implies$ (ii)}: Suppose that $C$ is a component that contains a 
    simple root but does not satisfy (ii). Let $\theta \in C$ be a simple root and 
    $A, B \in\Comp$ 
    both not equal to $C$ such that $A+B=C$. By 
    Lemma~\ref{lemma: If alpha + beta in C, alpha' + beta' in C}~(ii), 
    there exists $\alpha \in A$ and $\beta \in B$ such that $\alpha+\beta=\theta$, a contradiction.

    \np 
    \textit{(ii) $\implies$ (i)}: Suppose $C$ is a component satisfying (ii) but not (i). 
    Let $\gamma \in C$ be minimal. By our assumption, $\gamma$ is not simple. Hence,
     there exist positive roots $\alpha, \beta$ such 
    that $\alpha + \beta = \gamma$. If one of $\alpha, \beta$ is in $\Phi^c$, 
    we would have that the other is in $\Phi$, and in particular, in $C$.  This contradicts the minimality of $\gamma$ in $C$.
    If $\alpha, \beta$ are both in $\Phi$, they belong to their own 
    components $A, B$ respectively, such that $A+B=C$. Since $C$ is simple, 
    either $A$ or $B$ is $C$. Either way, there is a root in $C$ strictly 
    smaller than $\gamma$, which contradicts minimality.
\end{proof}

\begin{remark}
    \label{rem: induction on comp}
    It follows from Proposition~\ref{Simple Comp. equiv.}~(ii) that components 
    minimal with respect to $\leq$ are simple. This is useful in inductive arguments on the 
    poset $(\Comp, \leq)$, since simple components are nicely characterized by the equivalence 
    outlined in Proposition~\ref{Simple Comp. equiv.} and are in general easy to work with.
\end{remark}

\begin{prop}
\label{prop: at most one simple component of full support}
    There is at most one simple component $C\in\Comp$ with $\supp C = \supp \Phi$.
\end{prop}

\begin{proof} Suppose $C_1,C_2 \in \Comp$ are both simple 
with $\supp C_1 = \supp C_2 = \supp \Phi$. 
By definition, 
there is a simple root $\theta \in C_1$, and since $C_2$ is of full 
support in $\Phi$, there exists a root $\alpha \in C_2$ supported on $\theta$. 
Because $\theta \leq \alpha$, we have $C_1 \leq C_2$. Symmetrically, $C_2 \leq C_1$, and we conclude that $C_1 = C_2$.
\end{proof}

\begin{proposition}
\label{prop: all components are standard sums of simple components}
    Any element $A\in\Comp$ can be written as $A = T_1 + T_2 + \dots + T_n$ for simple components $T_i$.
\end{proposition}
\begin{proof}
    It suffices to prove that if $A$ is not simple, then $A$ can be written as $A = B + C$ where $B < A$ and $C$ is simple. 
    The result then follows by induction in the poset $\Comp$.
    To that end, suppose $A$ is not simple and let $\alpha \in A$ be minimal. 
    The root $\alpha$ is not simple since $A$ is not simple.
    Thus, $\alpha = \beta + \theta$ with $\theta$ a simple root and $\beta \in R^+$.
    If $\theta \in \Phi^c$, then $\beta \in A$ which contradicts minimality of $\alpha$.
    If $\beta \in \Phi^c$, then $\alpha - \theta \in \Phi^c$ and so $A$ contains a simple root, which is a contradiction.
    Thus, writing $C$ for the component of $\theta$ and $B$ for the component of $\beta$, we have $A=B+C$ where $B < A$ and 
    $C$ is simple.
\end{proof}

\np
We can always derive the partial order on $\Comp$ from its addition table. 
Reasonably, we might ask what information about the addition table we can learn from the partial order.
For example, the statement $A+B=B$ implies $A \leq B$. 
The following proposition provides a partial converse. 

\begin{prop}
\label{prop: addition of simple with lesser component is absorptive}
    Let $A, T \in \Comp$ with $T$ simple. If $A < T$, then $A + T$ is defined and is equal to $T$. 

\end{prop}
\begin{proof}
 By Remark~\ref{remark: additive characterization of partial order}, 
    there is a sequence of components $B_1, B_2, \dots, B_n$ such that 
    $A + B_1 + \dots + B_n = T$.
    Since $T$ is simple, either $A + B_1 + \dots + B_{n-1} = T$ or $B_n = T$.
    Induction on $n$ allows us to find a $k < n$ such that $A + B_1 + \dots + B_k + T = T$. 
    If $A + B_1 + \dots + B_{k-1} + T$ is defined and
    \[
    (A + B_1 + \dots + B_{k-1} + T) + B_k = T,
    \]
    then $T \leq A + B_1 + \dots + B_{k-1} + T \leq T$. Otherwise, if $B_k + T$ is defined and
    \[
    (A + B_1 + \dots + B_{k-1}) + (B_k + T) = T,
    \]
    then $T \leq B_k + T \leq T$.
    In both cases, $A + B_1 + \dots + B_{k-1} + T = T$.
    By Proposition~\ref{prop: component 2 of 3 rule}, at least one of these two conditions must hold.
    Induction on $k$ gives $A + T = T$.
\end{proof}

\begin{prop}
    \label{prop: components below a simple are additively closed}
    Let $T$ be a simple component. If $A, B < T$ and $A + B$ is defined, then $A + B < T$.
\end{prop}

\begin{proof}
    It suffices to prove that $A + B \leq T$, since the possibility that $A + B = T$ 
    is excluded by the fact that $T$ is simple. By 
    Proposition~\ref{prop: addition of simple with lesser component is absorptive}, 
    $A + T = B + T = T$. Thus, we may choose $\alpha \in A$ and $\beta \in B$ 
    such that $\alpha + \beta \in A + B$ and $\tau, \tau' \in T$ 
    such that $\alpha + \tau = \tau'$. We note that $(\alpha + \beta) + \tau - \tau' = \beta$ 
    so that by Proposition~\ref{proposition: 2 of 3 rule} either $\alpha + \beta + \tau$ 
    or $\alpha + \beta - \tau' = \beta - \tau$ is a root. If $\alpha + \beta + \tau$ is a root, 
    then $A + B + T$ is defined and by Proposition~\ref{prop: component 2 of 3 rule} it is equal to $T$, 
    yielding the desired result. If instead $\beta - \tau$ is a root, it must not be in $\pm \Phi^c$,
    as $B \neq T$. It further must not be in $\Phi$, since $B < T$. Thus, $\beta - \tau \in -\Phi$ 
    and $\tau - \beta \in \Phi$. Noting then that $(\alpha + \beta) + (\tau - \beta) = \tau' \in T$, 
    we conclude that $\alpha + \beta \leq \tau'$ and that $A + B \leq T$. 
\end{proof}

\np 
\begin{corollary}
    Let $T$ be a simple component with $\supp T = \supp \Phi$. If $A \ngeq T$, then $A + T = T$.
\end{corollary}

\begin{proof}
    By Proposition~\ref{prop: all components are standard sums of simple components}, 
    $A = T_1 + T_2 + \dots + T_n$ where $T_i$ is simple. Since $A \ngeq T$, $T_i \neq T$ 
    for all $1 \leq i \leq n$, and because $T$ is of full support in $\Phi$, $T_i < T$ for all $i$. 
    The result then follows from Propositions~\ref{prop: components below a simple are additively closed} 
    and \ref{prop: addition of simple with lesser component is absorptive}.  
\end{proof}

\subsection{Cancellation rules}
The partial addition of components does not necessarily satisfy the property that
$A+B = A+C$ implies $B = C$, i.e., cancellation does not hold in general. 
Below we study the question of when does the cancellation rule apply.

\begin{lemma}
    \label{lemma: alpha + beta = alpha' + gamma}
    Suppose $A, B, C, D \in \Comp$ satisfy $A + B = A + C = D$ and $A \neq D$. 
    Then there exist roots $\alpha, \alpha' \in A$, $\beta \in B$, and $\gamma \in C$ such that $\alpha + \beta = \alpha' + \gamma \in D$.
\end{lemma}

\begin{proof}
    If $B = C$, this is immediate. Otherwise, since $B \neq C$, either 
    $B \neq D$ or $C \neq D$, so we will assume without loss of generality that $C \neq D$ 
    and choose some $\alpha \in A$ and $\beta \in B$ such that $\alpha + \beta \in D$. Then, 
    by Lemma~\ref{lemma: If alpha + beta in C, alpha' + beta' in C}, $\alpha + \beta = \alpha' + \gamma$ 
    for some $\alpha' \in A$ and $\gamma \in C$, as desired.
\end{proof}

\np 
\begin{prop} \label{prop: first cancellation rule for component addition}
Let $A, B, C\in \Comp$.
If $A + B = A + C \neq A$, then $B = C$. \label{first}
\end{prop}

 \begin{proof} 
Let $D := A + B = A + C$. If $A=B$ and $A=C$ then $B=C$ and we are done.
Thus we may assume, without loss of generality, that $A \neq B$.
Applying Lemma~\ref{lemma: alpha + beta = alpha' + gamma}, we choose roots 
$\alpha, \alpha' \in A$, $\beta \in B$ and $\gamma \in C$ such that $\alpha + \beta = \alpha' + \gamma \in D$. 
We next fix a path $[\alpha'; \kappa_1, \kappa_2, \dots, \kappa_m; \alpha]$ with $\kappa_i \in \pm \Phi^c$. 
Noting that $\gamma - \beta = \alpha - \alpha'$ and applying 
Proposition~\ref{proposition: existence of paths}, we may construct the 
path $[\beta; \nu_1, \nu_2, \dots, \nu_n; \gamma]$ where $\nu_i$ is chosen from $\{ \kappa_1, \kappa_2, \dots, \kappa_m \}$.
We show that this path remains in $\Phi$ and thus in $B$.

Clearly the path begins in $\Phi$. We claim that if the $k$-th partial 
sum $\rho_k$ along the path is in $\Phi$, then $\rho_{k + 1}$ is as well. 
Indeed, if we suppose $\rho_{k + 1} = \rho_k + \nu_{k + 1} \in -\Phi$, 
then $-\rho_k - \nu_{k + 1} \in \Phi$, and so by closure $-\nu_{k + 1} \in \Phi$, 
which is a contradiction since $\nu_{k + 1} \in \pm \Phi^c$. Next we show that $\rho_{k + 1} \notin \pm\Phi^c$.
Writing $\nu_{k + 1} = \eta - \eta'$ for $\eta, \eta' \in A$, we apply 
Proposition~\ref{proposition: 2 of 3 rule} to conclude that either $\rho_k + \eta \in R$ or $\rho_k - \eta' \in R$. 
If $\rho_k + \eta \in R$, then $\rho_k + \eta \in D = A + B$, so $\rho_{k + 1} = \rho_k + \eta - \eta' \notin \pm \Phi^c$ 
since $\rho_{k + 1}$ would otherwise link $A$ and $D$. 
If $\rho_k - \eta' \in R$, then either (i) $\rho_k - \eta' \in \Phi$
or (ii) $\rho_k - \eta' \in -\Phi$ or (iii) $\rho_k - \eta' \in \pm \Phi^c$.

  \begin{enumerate}[(i)]
      \item In the first case, $\rho_k - \eta' \in \Phi$ implies $\rho_{k + 1} \in \Phi$ by closure of $\Phi$.

      \item In the second case, 
$\rho_k - \eta' \in -\Phi$ implies that $\eta' - \rho_k$ belongs to a component $E$ 
satisfying $E + B = A$. If $\rho_{k + 1} \in \pm \Phi^c$, 
then $(\eta' - \rho_k) - \eta = -\rho_{k + 1} \in \pm \Phi^c$, so $E = A$. 
We obtain from here the contradiction that $D = A + B = E + B = A$. 

    \item Finally, if $\rho_k - \eta' \in \pm \Phi^c$, then $A=B$, which contradicts the initial assumption that $A \neq B$.

\end{enumerate}

\np
Since $\rho_{k + 1} \notin -\Phi$ and $\rho_{k + 1} \notin \pm \Phi^c$, we 
conclude that $\rho_{k + 1} \in \Phi$, completing the proof. \end{proof}

\np
Next we provide a necessary condition for the failure of the cancellation property for $A+B = B$.

\begin{lemma}
    
\label{lemma: no cancellation implies A < simple < B}
    Suppose $A, B \in \Comp$ such that $A+B = B$. Then there exists a simple component $T \in \Comp$
     such that $A \leq T \leq B$.  As a consequence, if $A+A = A$, then $A$ is simple.
\end{lemma}
\begin{proof}
    If $B$ is simple, $A \leq B$ and there is nothing to prove, so assume $B$ is not simple.
    Proposition~\ref{prop: all components are standard sums of simple components}
    implies the existence of a component $C < B$ and a simple component $T$ such that $B = C+T$.
    Plugging this into $A+B = B$ gives $A+(C+T) = B = C+T$.
    Proposition~\ref{prop: component 2 of 3 rule} implies either $A+T$ is defined and $(A+T) + C = T + C$,
    or $A+C$ is defined and $(A+C) + T = C + T$.
    Since $B \neq T$ and $B \neq C$, we can cancel a term in both of the previous equations by 
    Proposition~\ref{prop: first cancellation rule for component addition} to see that either $A+T = T$ or $A+C = C$.
    In the first of the two cases, $A \leq T$ and so we are done. 
    In the second case, $C < B$ and so we can use induction on the poset $\Comp$ to obtain the desired result.
\end{proof}

\np
Next we state a condition which ensures that cancellation holds in $\Comp$.

\begin{prop}
If $U + T \neq T$ for all pairs of simple components $T, U \in \Comp$, then cancellation holds for
the partial addition in $\Comp$.
\end{prop}

\begin{proof}
    Proposition~\ref{prop: first cancellation rule for component addition} implies that 
    if  $A + B \neq B$ for all pairs of components $A, B \in \Comp$, then cancellation holds for
    the partial addition in $\Comp$. Hence it remains to show that 
    $U + T \neq T$ for all pairs of simple components $T, U \in \Comp$ implies that
    $A + B \neq B$ for all pairs of components $A, B \in \Comp$.

    \np
    Assume that $U + T \neq T$ for all pairs of simple components $T, U \in \Comp$ but 
    there are $A, B \in \Comp$  with $A + B = B$. The proof of 
    Lemma~\ref{lemma: no cancellation implies A < simple < B} implies that there exists a 
    simple component $T$ such that $A \leq T \leq B$ and $A + T = T$.  Then $A < T$ since 
    otherwise $T+T =T$ contradicting the assumption.
    Hence there is a simple $U$ such that $U \leq A$ and $U + T = T$ in contradiction to the assumption. 
\end{proof}

\subsection{Multiples of a component}
\np 
Proposition~\ref{prop: first cancellation rule for component addition} and 
Lemma~\ref{lemma: no cancellation implies A < simple < B} allow us to characterize the 
multiples of a component in $\Comp$.

\np 
If $r$ is a positive integer and $A \in \Comp$, we define $rA$ inductively as follows:
$1A := A$ and, for $r\geq 2$, 
\[rA:= \left\{ \begin{array}{ll}
 (r-1)A + A, & \text{\it{if this sum is defined}};  \\
 {\text {{\it undefined}}},    & {\text{\it{otherwise}}.} 
\end{array} \right.\]
There are three mutually exclusive alternatives for $A+A$: 
$A+A$ is undefined, $A+A=A$, or $A < A+A=2A$. 
The following lemma shows that these alternatives for $A+A$ determine how all multiples of $A$ behave.

\begin{prop}
    \label{prop: Repeated addition satisfies standard properties}
   Let $\langle A \rangle = \{ nA \mid n \geq 1, nA \text{ is defined} \}$. 
    There is $k$ such that $A, 2A, \dots, kA$ are defined and distinct, and
    $\langle A \rangle = \{A, 2A, \dots, kA \}$.
    Moreover, one of the following three mutually exclusive alternatives holds:
    \begin{enumerate}[(i)]
        \item $k=1$ and $nA$ is undefined for all $n \geq 2$.
        \item $k=1$ and $nA$ is defined and equal to $A$ for all $n \geq 2$. In this case, by 
        Lemma~\ref{lemma: no cancellation implies A < simple < B}, $A$ is necessarily simple.
        \item  $k \geq 2$, $nA$ is defined if and only if $1 \leq n \leq k$. Furthermore,
     for $1 \leq m, n \leq k$,
        $$mA + nA = \left\{ \begin{array}{ccc} 
                        (m+n)A, &\text{ if } &  m + n \leq k;\\
                        \text{undefined}, &\text{ if } & m + n > k.
                   \end{array} \right.$$
    \end{enumerate}
\end{prop}

\begin{proof} We let $k$ denote the cardinality of $\langle A \rangle$. 
If $nA$ and $(n+1)A$ are both defined and distinct, then $nA < (n+1)A$.
Hence $\langle A \rangle$ consists of the $k$ distinct elements $A, 2A, \dots, kA$. 
Then $(k+1) A$ is either undefined or $(k+1)A = kA$.
In the latter case $k=1$. Indeed, otherwise we would have $kA + A = (k-1)A  + A$ and  
Proposition~\ref{prop: first cancellation rule for component addition} would imply $kA = (k-1)A$, which contradicts
the assumption that $A, 2A, \dots, kA$ are distinct. 
To complete the proof, it remains to check the addition formula in (iii).

\np 
  Every $\alpha \in rA = (r-1)A + A$ may be written as 
$\alpha = \alpha' + \alpha''$ with $\alpha' \in (r-1)A$ and 
$\alpha'' \in A$ for all $r=1,2,\dots,k$.
It follows that every root $\alpha \in rA$
may be expressed as $\alpha = \alpha_1 + \alpha_2 + \dots + \alpha_r$
with $\alpha_1,\alpha_2, \dots, \alpha_r \in A$.  
From this, it is easy to see using Proposition~\ref{prop: addition works} that $mA + nA = (m+n)A$ whenever
$m+n \leq k$.  

\np 
    Consider now the case where $m+n > k$. By contradiction, assume that $mA + nA$ is defined and let 
    $t \geq k+1$ be minimal such that there exist
$1 \leq m,n \leq k$ with $m+n=t$ and 
$mA + nA$ is defined.
Further suppose $n$ is minimal such that 
$(t-n)A + nA$ is defined.

\np
If $n=1$, we already have an expression for $mA+A$, and there is nothing to show. 
Thus $n \geq 2$ and $mA + ((n-1)A + A)$ is defined.
  By Proposition~\ref{prop: component 2 of 3 rule} 
either $(mA + (n-1)A) + A$ or 
$(mA + A) + (n-1)A$ is defined.
In the former case, we again must have $m+n-1=k$ 
which cannot happen.
The latter case violates the minimality of 
$n$.   Either way we arrive at a contradiction.
\end{proof}

\np 
Here we give examples of each of the three possibilities from 
Proposition~\ref{prop: Repeated addition satisfies standard properties}.
\begin{example} \label{Example Z+Z}
\begin{enumerate}
    \item[(i)] 
    The set $\Phi = \{010, 110, 011\}$
is a primitive inversion set in $\mathbb B_3$.
There is a single component $Z$ in $\Comp$ and the sum $Z+Z$ is undefined.

\item[(ii)]
The set
$\Phi = \{010, 110, 011, 111, 122\}$
is a primitive inversion set in $\mathbb B_3$.
There is a single component $Z$ in $\Comp$.
Since $011 + 111 = 122$ we see that $Z+Z=Z$.
Hence $Z + Z + \dots + Z = Z$ for all such sums.

\item[(iii)]
The set 
$$\Phi=\left\{ 
\DFour1000,\DFour0010,\DFour0001,\DFour1110,\DFour1101,\DFour0111,\DFour1111\right\}$$
is a primitive inversion set in ${\mathbb D}_4$.
The graph $\GPhi$ has two components:\\
$Z = \{\DFour1000,\DFour0010,\DFour0001,\DFour1110,\DFour1101,\DFour0111\}$
and 
$2Z = \{ \DFour1111 \}$.
\end{enumerate}
\end{example}

\begin{remark}
    We are only aware of two other examples of primitive inversion sets with $k\geq 2$ as in 
    Proposition~\ref{prop: Repeated addition satisfies standard properties}.  They both have $k=2$.
They are $\Phi_1 = \{10, 21, 31 \} \subseteq \mathbb{G}_2$ with $Z=\{10,21\}$
and $\Phi_2 = \{100,001,111,012,112\} \subseteq \mathbb{B}_3$ with 
$Z=\{100,001,111,012\}$.
These examples are related to each other by diagram foldings. By taking $\Phi$ from 
Example~\ref{Example Z+Z}(iii), which is an inversion set in $\mathbb{D}_4$, and 
folding it to either $\mathbb{G}_2$ or $\mathbb{B}_3$, we get the inversion sets $\Phi_1$ and $\Phi_2$ above.
 \end{remark}

\section{Primitive and irreducible inversion sets} \label{sec: components as inflations}

\np In this section we continue to assume that $R$ is an irreducible QRS, $\Phi \subseteq R^+$ is an inversion set,
and we retain the notation from Sections~\ref{sec: graphs} and 
\ref{sec: addition}. 
We investigate the relationship between the set $\Gen(\Phi)$ and the supports of elements of $\Comp$.
In particular we prove that when $\Phi$ is written canonically as $\Phi = \inf_I^S(\Psi, X)$, 
then all elements of $\Comp$ are inflated from $I$. 
We then use this result to prove that
an inversion set $\Phi$ is primitive if and only if both $\Phi$ and $\Phi^c$ are irreducible.
The results of this section are used in Section~\ref{sec: GIT} below.

\subsection{The set $\Gen(A)$ for $A \in \Comp$}

\begin{prop}
\label{prop: sums of inflations is an inflation}
    Let $A, B\in \Comp$ such that $A+B$ is defined. If $I \in \Gen(A)\cap \Gen(B)$, then $I \in \Gen(A + B)$.
\end{prop}

\begin{proof}
    If $A + B$ is equal to $A$ or $B$ the result is immediate. 
    Otherwise, for any $\gamma \in A + B$ such that $\pi_I(\gamma) \neq 0$, $\pm \theta \in I$, and $\gamma+\theta \in R$,
    we would like to show that $\gamma+\theta \in A+B$.
    To that end, we apply Lemma~\ref{lemma: If alpha + beta in C, alpha' + beta' in C}~(ii) 
    and write $\gamma = \alpha + \beta$ where $\alpha \in A, \beta \in B$. 
    Then $\gamma + \theta = \alpha + \beta + \theta$. 
    
    \np
    We note that $\pi_I(\alpha) \neq 0$. Otherwise $\beta$ and $\alpha + \beta$ would be in 
    the same fibre, and since $I \in \Gen(B)$, we would have $A + B \subseteq B$ 
    which would imply that $A+B = B$. Similarly, $\pi_I(\beta) \neq 0$. 
    In particular, $\alpha + \theta$ and $\beta + \theta$ are all non-zero. 
    Note also that $\alpha + \beta \neq 0$, because both $\alpha, \beta$ are positive. 
    Therefore, using Proposition~\ref{proposition: 2 of 3 rule}, we conclude that either $\alpha + \theta$ is 
    a root or $\beta  + \theta$ is a root. 
    Say $\alpha + \theta \in R$.
Since $I \in \Gen(A)$, this means that $\alpha + \theta \in A$ and therefore $(\alpha + \theta) + \beta \in A + B$. 
    We thus conclude that $I \in \Gen(A + B)$.
\end{proof}

\begin{prop}
\label{prop: supports of components are inflaters of Phi}
    For all $C\in \Comp$, $\supp C \in \Gen(\Phi)$.
\end{prop}

\begin{proof}
Let $I = \supp C$ and $D = \bigcup_{C' \leq C} C'$. It follows from 
Corollary~\ref{cor: A <= B => supp A subset of supp B} that $\supp D = I$. 
Next, suppose $\alpha + \beta \in D$ for some $\alpha, \beta \in R^+$. 
We will show that either $\alpha$ or $\beta$ is in $D$ to prove that $D$ is co-closed. 
Since $C \subseteq \Phi$ and $\Phi$ is co-closed, either $\alpha$ or $\beta$ is in $\Phi$, 
so without loss of generality we let $\alpha \in \Phi$. We then give the labels 
$C_{\alpha + \beta}$ and $C_\alpha$ to the components respectively containing $\alpha + \beta$ 
and $\alpha$. Since $\alpha \leq \alpha + \beta$, we have 
that $C_\alpha \leq C_{\alpha + \beta} \leq C$, so $C_\alpha \subseteq D$, $\alpha \in D$, and $D$ is co-closed.

\np 
Having constructed the co-closed set $D$, we use it to prove our result. 
We suppose toward a contradiction that there exist $\gamma, \delta \in R^+$ 
with $\pi_I(\gamma) = \pi_I(\delta) \neq 0$, $\gamma \in \Phi$, and $\delta \in \Phi^c$. 
By Proposition~\ref{prop: co-closed sets span their support}, $\delta - \gamma \in \ZSpan(D)$, and so 
Proposition~\ref{proposition: existence of paths} guarantees a path $[\gamma; \kappa_1, \dots, \kappa_n; \delta]$ 
with $\kappa_i \in \pm D$. Both $\gamma$ and $\delta$ belong to the same fibre 
of $R / I$, and since $\kappa_i \in \pm D \subseteq \Span I$, every partial sum along 
this path belongs to that same fibre and hence must be a positive root. Since 
$\gamma \in \Phi$, $\delta \in \Phi^c$, and all partial sums are positive, there will be 
some $i$ such that the $i$-th partial sum $\gamma_i$ is in $\Phi$ while the $(i + 1)$-th 
partial sum $\gamma_{i + 1}$ is in $\Phi^c$. We would then have $\gamma_i + \kappa_{i + 1} = \gamma_{i + 1}$. 
We note that $\kappa_{i + 1}$ must be negative since since $\Phi$ is closed and $\gamma_{i + 1} \in \Phi^c$. 
But then $\gamma_i - (-\kappa_{i + 1}) = \gamma_{i + 1}$, so that $\gamma_i$ and $-\kappa_{i + 1}$ are
linked in $\GPhi$. Since $-\kappa_{i + 1} \in D$ and $D$ is a union of components of $\GPhi$, 
this implies $\gamma' \in D$, which is not possible because $\pi_I(\gamma') \notin \pi_I(D) = \{ 0 \}$.
\end{proof}

\begin{prop}\label{prop: special component in inflation of a prime}
   
    Let $\Phi = \inf_I^S(\Psi, X)$ expressed canonically with $\Psi$ primitive. There is a 
    unique simple component $Z$ in $\Comp$ not contained in
    $X$ with $\supp Z = S$.
    All other simple components are contained in $X$.
\end{prop}

\begin{proof}
    Pick a simple root $\theta \in \inf_I^S(\Psi, \emptyset)$ and let $Z\in \Comp$ be the component 
    containing $\theta$. By Proposition~\ref{prop: supports of components are inflaters of Phi}, $\supp Z \in \Gen(\Phi)$, 
    and by Proposition~\ref{prop: Gen(Phi) for primitive Psi}, either $\supp Z \subseteq I$ or $\supp Z = S$. 
    Since $\theta \in \supp Z$ but $\theta \notin I$, we conclude that $\supp Z = S$.
    Because we chose $\theta$ arbitrarily, it follows that any component containing a simple root 
    of $\inf_I^S(\Psi, \emptyset)$ has support $S$, and by 
    Proposition~\ref{prop: at most one simple component of full support} there is only one such component. 
    All other simple components therefore must be contained in $X$.
\end{proof}

\np  
The following corollary is useful in the proof of Proposition~\ref{prop: components are inflated by the canonical}.

\begin{corollary} \label{cor:sum of simples}
    Let $\Phi = \inf_I^S(\Psi, X)$ written in canonical form with $\Psi$ primitive. 
    Let $X', Z \in \Comp$ be simple components where $X' \subseteq X$ and $\supp Z = S$. Then $X' + Z = Z$.
\end{corollary}

\begin{proof}
    Let $\theta \in X'$ be simple. Since $\supp Z = S$, there exists some $z \in Z$ such 
    that $\theta \leq z$, so that $X' \leq Z$. The result then follows from 
    Proposition~\ref{prop: addition of simple with lesser component is absorptive}. 
\end{proof}

\np 
While this is all that is required for the proof of 
Proposition~\ref{prop: components are inflated by the canonical}, one might 
reasonably wonder whether $Y + Z = Z$ for any (not necessarily simple component) 
$Y \subseteq X$. This is implied by the following statement.

\begin{remark} 
Let $T \in \Comp$ be a simple component with $\supp T = S$. Then,
    \begin{enumerate}[(i)]
        \item All components are comparable to $T$.
        \item All components greater than $T$ are equal to $kT$ for some $k > 1$.
        \item Letting $\Phi = \inf_I^S(\Psi, X)$ be the canonical form of $\Phi$, $\Psi$ 
        is either primitive or $(R / I)^+$ where $R / I$ is a rank-one system.
    \end{enumerate}
\end{remark}

\begin{prop}\label{prop: components are inflated by the canonical}
    Let $\Phi$ be an inversion set canonically inflated from $I$. Then all components of $\GPhi$ are inflated from $I$.
\end{prop}

\begin{proof} We write $\Phi = \inf_I^S(\Psi, X)$ and split the argument into three mutually exclusive cases.

\begin{enumerate}
    \item[(i)] $\Psi = \emptyset$:

\np 
Here the result is trivial, since $\supp C \subseteq \supp \Phi = I$ for every component $C$ of $\GPhi$, so that $I \in \Gen(C)$.

\item[] 
    \item[(ii)] $\Psi = (R / I)^+$:

\np 
Since $I$ is canonical, $\supp \Phi^c = I$. This implies that each component of $\GPhi$ is 
contained in a single fibre of $\pi_I$. To show that each is inflated from $I$ we need only 
to show that those components that are contained in a fibre over a non-zero root of $R / I$ 
are exactly the fibre that contains them. This is a consequence of 
Propositions~\ref{prop: co-closed sets span their support} and \ref{proposition: existence of paths}, 
that state respectively that roots in $\Phi^c$ span the difference between any two roots $\alpha, \beta$ 
in the same fibre, and that roots of $\Phi^c$ can therefore be used to construct a path between $\alpha$ 
and $\beta$. This path stays in the same fibre and therefore in $\Phi$, so that components of $\GPhi$ 
are either fibres of $\pi_I$ over non-zero elements or contained in the zero fibre. In both cases, they are inflated from $I$.

\item[] 
    \item[(iii)] $\Psi$ is primitive:

\np 
Let $Z$ be the unique simple component with $\supp(Z)=S$ defined in 
Proposition~\ref{prop: special component in inflation of a prime} and let $X_1, X_2, \dots, X_n$ denote the 
other simple components, which lie in $X$.
Every $X_i$ is clearly inflated from $I$. We will now show that $Z$ is also inflated from $I$. 
It follows from Corollary~\ref{cor:sum of simples} that
$X_i + Z = Z$ for all $i$.

\np 
Let $\gamma \in Z$ 
and $\theta \in \pm I$ such that $\gamma + \theta \in R$. Since $\Phi$ is inflated from $I$, $\gamma + \theta \in \Phi$. 
If $\theta \in \pm \Phi^c$, then $\gamma + \theta$ is adjacent to $\gamma$ in $\GPhi$, 
so $\gamma + \theta \in Z$. If $\theta \in \Phi$, then the component $X_\theta$ of $\GPhi$ 
containing $\theta$ is simple and contained in $X$, so that $\gamma + \theta \in X_\theta + Z = Z$. 
Otherwise, if $\theta \in -\Phi$, then $\gamma + \theta \leq \gamma$ and so the component containing 
it is less than or equal to $Z$. Since $Z$ is minimal among components not contained in $X$ 
and $\pi_I(\gamma + \theta) = \pi_I(\gamma) \neq 0$, it follows that $\gamma + \theta \in Z$. 
This shows that $I \in \Gen(Z)$. With this, we have that all simple components are inflated from $I$.

\np  Since all components are sums of simple components by 
Proposition~\ref{prop: all components are standard sums of simple components} and sums of 
inflations from $I$ are inflations from $I$ by Proposition~\ref{prop: sums of inflations is an inflation}, 
every component is inflated from $I$. \end{enumerate} 

\np 
Since the three preceding cases cover all possibilities for $\Psi$ in the canonical expression of 
$\Phi$, the proof is complete. \end{proof}

\begin{corollary}
    \label{corollary: decompositions of canonical I inflations are I inflations}
    Let $\Phi, \Phi_1, \dots, \Phi_k$ be inversion sets such that $\Phi = \Phi_1 \sqcup \dots \sqcup \Phi_k$.
    Assume that $\Phi = \inf_I^S(\Psi, X)$ is the canonical form of $\Phi$.
    Then $I \in \Gen(\Phi_1)\cap \dots \cap \Gen(\Phi_k)$. \hfill $\square$
\end{corollary}

\subsection{Irreducibility of $\Phi$ and $\Comp$}
We start with establishing necessary and sufficient conditions for  an inversion set $\Phi$ 
to be irreducible in terms of its associated graph $\GPhi$.
\begin{theorem}
\label{theorem: irreducibility equivalent conditions}
    Let $\Phi$ be an inversion set. The following are equivalent:

    \begin{enumerate}[(i)]
        \item $\Phi$ is irreducible.

        \item Every component $C\in\Comp$ satisfies $\supp C = \supp \Phi$.

        \item The set $\Comp$ contains a unique simple component.

        \item There exists a component $Z \in \Comp$ such that $\Comp = \{ Z, 2Z, \dots, kZ \}$ for some $k \geq 1$.
    \end{enumerate}

\end{theorem}

\begin{proof}
We show that \textit{(i) $\implies$ (ii)} by contrapositive. If there were a component $C$ with $\supp C \subsetneq \supp \Phi$, 
then $\Phi = \inf_{\supp C}^S(\Psi, X)$ by Proposition~\ref{prop: supports of components are inflaters of Phi}. 
Because $\supp C \neq \supp \Phi$, $\Psi$ is non-empty. Since $C \subseteq  X$, $X$ is also non-empty, 
and thus $\Phi$ is reducible. 
That \textit{(ii) $\implies$ (iii)} follows immediately from 
Proposition~\ref{prop: at most one simple component of full support}. \textit{(iii) $\implies$ (iv)} because by 
Proposition~\ref{prop: all components are standard sums of simple components} every component can be written as 
the sum of simple components, and by \textit{(iii)} there is a unique simple component $Z$. 
To see that \textit{(iv) $\implies$ (i)}, 
suppose that $\Phi = \Phi_1 \sqcup \Phi_2$. Assume $Z \subseteq \Phi_i$. 
By the closure of $\Phi_i$, every other component is also contained in $\Phi_i$, 
so $\Phi_i = \Phi$. Since we began with an arbitrary decomposition, $\Phi$ must be irreducible.
\end{proof}

\begin{corollary} \label{cor: primitive implies irreducible}
    Every primitive inversion set is irreducible.
    Conversely, if $\mathrm{rk} \, R > 2$,
    then every irreducible inversion set 
    whose complement is also irreducible is necessarily primitive. \qed
\end{corollary}
\begin{proof}
The first assertion is clear by Proposition~\ref{prop: special component in inflation of a prime} 
combined with Theorem~\ref{theorem: irreducibility equivalent conditions}.  For the second assertion,
    suppose that $\Phi$ and $\Phi^c$ are irreducible but not primitive.
    Then there is an $I \subsetneq S$ with $\Phi = \inf_I^S(\Psi, X)$, and $\Phi^c = \inf_I^S(\Psi^c, X^c)$.
    Without loss of generality, assume $\Psi$ is not empty. 
    Then since $\Phi$ is irreducible, $X$ is empty.
    If $I$ is empty, then the fact that $\Phi$ is not primitive implies that $\Psi = R^+/I = R^+$.
    In the case that $I$ is not empty, irreducibility of $\Phi^c$ implies that one of 
    $\Psi^c, X^c = R_I^+$ is empty, therefore $\Psi^c = \emptyset$.
    We see that in every case, $\Phi = \inf_I^S(R^+/I, \emptyset)$ and $\Phi^c = \inf_I^S(\emptyset, R_I^+)$.
    The inversion sets $R^+_I$, and $R^+/I$ are irreducible only if they are both quotient 
    root systems of rank no more than $1$. Since $\mathrm{rk} \, R = \mathrm{rk} \, R_I + \mathrm{rk} \, R/I$, 
    we conclude that $\mathrm{rk} \, R \leq 2$.
\end{proof}

\begin{remark}
    The condition appearing in the above corollary that $\mathrm{rk} \, R > 2$ is necessary 
    as the following example from $\mathbb A_2$ shows.  
Take $\Phi=\{\theta_1,\theta_1+\theta_2\}$ and observe that both $\Phi$ and $\Phi^c = \{ \theta_2\}$ are irreducible. But since
$$ \Phi = \inf_{\{\theta_2\}}^{\{\theta_1,\theta_2\}}(\{\bar{\theta}_1\},\emptyset) \quad \quad \text{and} \quad \quad \Phi^c = \inf_{\{\theta_2\}}^{\{\theta_1,\theta_2\}}(\emptyset,\{\theta_2\}),$$
neither is primitive.

\end{remark}

\np 
Theorem~\ref{theorem: irreducibility equivalent conditions} implies that, using the 
structure of the graph $\GPhi$, one can determine whether an inversion set $\Phi \subseteq R^+$ 
is irreducible in time which is polynomial in $\mathrm{rk} \, R$. This is a significant improvement 
over a brute force search over possible decompositions of $\Phi$, since, for a root 
system $\Delta$, the number of inversion sets in $\Delta^+$ is at least $(\mathrm{rk} \, \Delta + 1)!$.

\begin{corollary} If $R$ is a QRS and 
$\Phi \subseteq R^+$ is an inversion set, then the time required
to determine the irreducibility of $\Phi$ is polynomial in $\mathrm{rk}\, R$.\qed
\end{corollary}

\section{\texorpdfstring{Addition in $\Comp$ via the canonical form of $\Phi$}{Addition in Comp via the canonical form of Phi}}
\label{sec:additive structure through}

\np 
Recall that $\pi_I$ is the canonical projection from $E$ onto $E / I$, the orthogonal complement
of $E_I = \Span I$ in $E$.

\begin{prop}
\label{prop: pi_I defines map from components of phi to those of psi}
    Let $\Phi = \inf_I^S(\Psi, X)$. If $C$ is a component of $\GPhi$, then 
$\pi_I(C) = \{ 0 \}$ or $\pi_I(C) = \bar C$ for some component $\bar C$ of $\GPsi$.
\end{prop}

\begin{proof}
    If $C \subseteq X$, then of course $\pi_I(C) = \{ 0 \}$.
    Otherwise, pick some $\alpha \in C$ and let $\bar C$ be the component 
    of $\GPsi$ containing $\bar \alpha := \pi_I(\alpha)$.
    To show that $\bar C \subseteq \pi_I(C)$, we take $\bar \beta$ to be a 
neighbour of $\bar \alpha$ in $\GPsi$ and show that $\bar \beta \in \pi_I(C)$.
    Since $\bar \alpha$ is adjacent to $\bar \beta$, 
$\bar \alpha + \bar \kappa = \bar \beta$ for $\bar \kappa \in \pm \Psi^c$.
    By Proposition~\ref{prop: lifts of sums in quotients} there exists some lifts 
$\kappa$ of $\bar \kappa$ and $\beta$ of $\bar \beta$ such that 
$\alpha + \kappa = \beta$, $\pi_I(\beta) = \bar \beta$, and $\pi_I(\kappa) = \bar \kappa$.
    Since $I \in \Gen(\Phi)$, $\kappa \in \pm \Phi^c$ and $\beta \in \Phi$.
    We therefore find that $\beta \in C$ and $\bar \beta \in \pi_I(C)$.
    It follows that $\pi_I(C) \subseteq \bar C$.
    To demonstrate the reverse inclusion, we suppose $\beta$ is a neighbour of 
$\alpha$ in $\GPhi$ and show that $\pi_I(\beta) \in \bar C$.
    Since $\beta$ is adjacent to $\alpha$, $\alpha + \kappa = \beta$ for some $\kappa \in \pm \Phi^c$.
    We either have that $\bar \kappa := \pi_I(\kappa) = 0$ or $\bar \kappa \in \pm \Psi^c$.
    In the first case, $\pi_I(\beta) = \bar \alpha \in \bar C$.
    In the latter, $\bar \alpha + \bar \kappa =  \pi_I(\beta)$, 
so $\bar \alpha$ and $\pi_I(\beta)$ are adjacent in $\GPsi$ and $\pi_I(\beta) \in \bar C$.
    We conclude that $\pi_I(C) = \bar C$.
\end{proof}

\begin{remark} \label{rem: pi_i is additive}
    \begin{enumerate}[(i)]
        \item Proposition~\ref{prop: pi_I defines map from components of phi to those of psi} 
allows us to extend $\pi_I$ to a map between $\Comp$ and ${\rm Comp}(\Psi) \cup \{\{0\}\}$.
\item The map $\pi_I$ respects addition, in the sense that if $A, B, C \in \Comp$ 
with $A + B = C$, then $\pi_I(A) + \pi_I(B) = \pi_I(C)$ where the addition in $\mathrm{Comp}(\Psi)$
is extended to an addition in $\mathrm{Comp}(\Psi) \cup \{\{0\}\}$ in the obvious way.
    \end{enumerate}
\end{remark}

\begin{prop}
    \label{prop: subgraph when intersecting with X is GX}
     Let $\Phi = \inf_I^S(\Psi, X)$. The full subgraph of $\GPhi$ whose vertices are the roots in $X$ is $\GX$.
\end{prop}

\begin{proof}
    If there is an edge in $\GX$ between two vertices $\alpha, \beta$ in $X$, then of 
course the edge is also present in $\GPhi$, since $X^c := R_I^+ \setminus X \subseteq \Phi^c$. 
Conversely, if there is an edge in $\GPhi$ between $\alpha, \beta \in X$, then $\alpha + \kappa = \beta$ 
for $\kappa \in \pm \Phi^c$. Since $\pi_I(\kappa) = \pi_I(\beta) - \pi_I(\alpha) = 0$, 
$\kappa \in \pm X^c$, so there is also an edge between $\alpha$ and $\beta$ in $\GX$.
\end{proof}

\begin{prop}
    \label{prop: pi_I is surjective on components and bijective when I is canonical}
    Consider the inflation $\Phi = \inf_I^S(\Psi, X)$.  Then $\pi_I$ induces a surjective map
from $\Comp \setminus \mathrm{Comp}(X)$ onto $\mathrm{Comp}(\Psi)$. When $\Phi$ is canonically inflated from $I$, this map is bijective. 
\end{prop}

\begin{proof}
    Surjectivity of the map is clear: pick $\bar \alpha \in \Psi$ and some $\alpha \in R$ 
such that $\pi_I(\alpha) = \bar \alpha$. Then the component containing $\alpha$ is projected onto 
the component containing $\bar \alpha$, and surjectivity follows. Next, suppose $\pi_I(C_1) = \pi_I(C_2)$ 
for components $C_1, C_2$ of $\GPhi$. Letting $\bar \beta \in \pi_I(C_1)$, we see that there must 
exist $\beta_1 \in C_1, \beta_2 \in C_2$ such that $\pi_I(\beta_1) = \pi_I(\beta_2) = \bar \beta$. 
Because of this, $\beta_1$ and $\beta_2$ belong to the same fibre of $\pi_I$. Since components of $\GPhi$ 
are inflated from $I$ when $I$ is canonical by Proposition~\ref{prop: components are inflated by the 
canonical}, we conclude that in this case $C_1 = C_2$ and $\pi_I$ is injective.
\end{proof}

\begin{prop} \label{addition within or without X}
    Let $\Phi = \inf_I^S(\Psi, X)$ when represented canonically. 
    Let $X_1, \dots, X_m \in\Comp$ be 
the components contained in $X$, and $Z_1, \dots, Z_n$ be the 
remaining elements of $\Comp$ (which are contained in $\Phi \setminus X$). Then:
    \begin{enumerate}[(i)]
        \item \label{X only}The addition of the components $X_1, \dots, X_m$ among themselves 
        is given by the addition table of $\mathrm{Comp}(X)$  in the QRS $R_I$.

        \item\label{Z only} The addition of the components $Z_1, \dots, Z_n$ among themselves 
        is given by the addition table of $\mathrm{Comp}(\Psi)$  in the QRS $R/I$.

        \item\label{Xs and Zs mixed} If $Z_i + X_j$ is defined, then $Z_i + X_j= Z_i$.
    \end{enumerate}
\end{prop}
\begin{proof}
    The first assertion is just Proposition~\ref{prop: subgraph when intersecting with X is GX}.  
    The second and third assertions follow immediately from 
    Proposition~\ref{prop: pi_I is surjective on components and bijective when I is canonical} and 
    Remark~\ref{rem: pi_i is additive}. 
\end{proof}

\begin{prop}\label{prop: addition is independent of bracketing}
Consider a collection of (not necessarily distinct) components $C_1,C_2,\dots,C_n\in \Comp$.
Suppose $T, T' \in \Comp$ are two sums of $C_1,C_2,\dots,C_n$ which may differ in the order of the summands and 
their parenthesization. 
Then $T=T'$.
\end{prop}

\begin{proof} 

The proof is by (strong) induction on $\rk R$. The base case is when the rank is $1$, which is clear. 
So assume that $\rk R \geq 2$. If $\Phi = R^+$, then the elements in $\Comp$ are the positive roots 
and the statement states that if two roots can be expressed using the same summands, then they are equal. 
This is obvious.  Consider now the case where $\Phi$ is primitive.  There are exactly $k$ components, all multiples of
 the single simple component $Z$ with
 $Z < 2Z < \dots < kZ$ where $k \geq 1$.  In this case the result follows easily from 
 Proposition~\ref{prop: Repeated addition satisfies standard properties}.

\np 
   Now suppose that $\Phi$ is not primitive and let $\Phi = \inf_I^S(\Psi,X)$
be the canonical expression for $\Phi$ as an inflation.  Recall that $I \subsetneq S$. 
Also, if $I = \emptyset$, since $\Phi$ is not primitive, this yields $\Phi = R^+$, which has been treated above. 
Hence, we may assume that $\emptyset \subsetneq I \subsetneq S$.
 Write $X_1,X_2,\dots,X_r$ to denote the elements of $\Comp$ contained in $X$
and $Z_1,Z_2,\dots,Z_s$ to denote the elements of $\Comp$ outside $X$.  
Since, by Proposition~\ref{addition within or without X}(iii), $X_j + Z_i = Z_i$ 
whenever this sum is defined, we may assume that a sum of components is a sum composed of only 
the $Z_i$ or is a sum composed of only the $X_j$. 

\np 
 In the latter case, the addition may be viewed as a sum of components of
$\GX$ and we consider the QRS $R_I$ with $1 \le \rk R_I < \rk R$. Although $R_I$ may not be irreducible, 
all the components $X_i$ must be components of an irreducible factor $C$ of $R_ I$. 
We proceed by induction since $1 \le \rk C \le \rk R_I < \rk R$. 
Alternatively, if the sum involves only the components 
$Z_1,Z_2,\dots,Z_s$ then the addition may be viewed as an addition 
of components of $\GPsi$.  In this case, we use induction as $\rk R/I < \rk R$.
\end{proof}

\begin{prop} \label{prop: more cancellation rules for component addition}
Let $A, B, C\in\Comp$.
\begin{enumerate}[(i)]
\item If $A + B + C = A + B$, then $A + C = A$ or $B + C = B$.
\label{second}
\item If $A + B + C = B + C$, then $A + B = B$ or $B + C = C$.
\label{third}
\end{enumerate}
\end{prop}

\begin{proof} 

Induction on $\rk R$.
As in the proof of Proposition~\ref{prop: addition is independent of bracketing}, we see that the statement holds in the 
following cases: (i) if $\rk R = 1$, (ii) if $\Phi = R^+$, or (iii) if $\Phi$ is primitive.

\np 
Assume that none of the cases (i), (ii), or (iii) holds and write $\Phi=\inf_I^S(\Psi,X)$ for the canonical
expression of $\Phi$. Since $\Phi$ is not primitive and $\Phi \ne R^+$, we conclude that $I \neq \emptyset$.
If $A,B,C$ all lie in $X$, then the addition involves components of
$\GX$ where $X \subseteq  \text{span } I$ and the result follows by induction, as the sum involves components in an 
irreducible factor of whose rank is strictly smaller than $\rk R$.  If none of $A,B,C$ is contained in $X$ 
then additions of $A$, $B$, and $C$ are governed by additions of their respective counterparts in 
$\GPsi$. Since $\Psi \subseteq  R^+/I$ and  $\rk R/I < \rk R$, the result follows by induction.

\np 
It remains to consider the cases where some of $A,B,C$ are contained in $X$ and some are not.
Suppose that $A + B + C = A + B$. If $C \subseteq X$, then $A+C = A$ or $B+C = B$ since at least one of 
$A$ and $B$ is not in $X$. If $C \not\subseteq X$, then one of $A, B$ is in $X$ and the other is not. Say $A\subseteq X$,
$B \not\subseteq X$. Then $A+B = B$ and hence $A+B+C = B+C$. Combining these two with the assumption that
$A+B+C = A+B$, we conclude that  $B+C = B$.

\np 
Suppose $A+B+C = B+C$. Applying Proposition~\ref{prop: first cancellation rule for component addition}
to $A':=C$, $B':= A+B$, and $C':= B$ in place of $A,B, C$, we conclude that either $B'=C'$ or $A'+C' = A'$, i.e.,
that $A + B = B$ or $B + C = C$.
\end{proof}

\section{Applications to GIT} \label{sec: GIT}

\np 
If $\Delta$ is a root system and $w$ is an element of the Weyl group of $\Delta$, 
the inversion set $\Phi(w)$ of $w$ is defined as
\[\Phi(w) := \{\alpha \in \Delta^+ \, | \, w(\alpha) \in \Delta^-\} = \Delta^+ \cap w^{-1}(\Delta^-) \ .\]
Taking inversion sets of elements of $W$ establishes a bijection between the Weyl group $W$ and
the inversion sets in $\Delta^+$: $\Phi(w)$ is an 
inversion set in $\Delta^+$ and, conversely, every inversion set in $\Delta^+$ equals $\Phi(w)$ 
for a unique $w \in W$. In this section we show how several problems which have recently arisen 
in geometric invariant theory can be studied using the methods developed in this paper.

\subsection{Decomposing inversion sets} \label{sec:8.1}
The problem of decomposing a given inversion set $\Phi(w)$ into the disjoint union of inversion 
sets $\Phi(w_1), \dots, \Phi(w_k)$ arises naturally in connection with different problems
in geometric invariant theory, e.g. the description of the faces of the Littlewood-Richardson cone or 
the Belkale-Kumar product on $H^*(G/B, \ZZ)$, see \cite{BK}, \cite{DR}, \cite{FR} for details. 
In \cite{DDMRWW} decompositions 

\np
\begin{equation} \label{equ_decomposition_0}
\Phi(w) = \Phi(w_1) \sqcup \Phi(w_2) \sqcup \dots \sqcup \Phi(w_k) 
\end{equation}
were studied for roots systems of types $\mathbb{A}, \mathbb{B}$, and $\mathbb{C}$. The approach in
\cite{DDMRWW} is based on the fact that the Weyl group of the root system $\mathbb{A}_n$ is the symmetric
group on $n+1$ elements and the Weyl groups of the root systems $\mathbb{B}_n$ and $\mathbb{C}_n$ can
be realized as groups of ``symmetric'' permutations. Consequently, the methods of \cite{DDMRWW} do not
extend beyond these three types of root systems. In this section we develop a type-independent
approach to studying the analogs of the decomposition \eqref{equ_decomposition_0} and other related problems. 
More precisely, we provide a description of decompositions like \eqref{equ_decomposition_0} which is recursive 
in terms of subsystems and quotients of $\Delta$. Consequently, the natural setting in which we 
study \eqref{equ_decomposition_0} and related topics is provided by quotient root systems.

\np
Let $R$ be a QRS. First we characterize all decompositions 

\begin{equation} \label{equ_decomposition_simeq}
R^+ = \Phi_1 \sqcup \Phi_2 \sqcup \dots \sqcup \Phi_k \ ,
\end{equation}
where $\Phi_i \subseteq \R^+$ is an inversion set for every $1 \leq i \leq k$.
Both inflations and the decomposition \eqref{equ_decomposition_simeq}
are compatible with the decomposition of a QRS into its irreducible components. More precisely,
if $R$ is a QRS with irreducible components $R^1, R^2, \dots, R^s$, then every decomposition of $R^+$ into
inversion sets yields a decomposition of each $R^i$ into inversion sets and vice-versa. As a consequence, 
it is sufficient to describe the decompositions \eqref{equ_decomposition_simeq} for irreducible QRSs. 

 \begin{theorem}
\label{theorem: main theorem}
    Let $R$ be an irreducible QRS and let
    \[R^+ = \Phi_1 \sqcup \Phi_2 \sqcup \dots \sqcup \Phi_k \ ,\] 
    where $\Phi_1, \Phi_2, \dots, \Phi_k$ are inversion sets. Assume that $\Phi_1$ contains the highest root of $R$. 
    If $\Phi_1 = \inf_I^S(\Psi, X_1)$ when expressed canonically, then up to a relabelling, 
    $\Phi_2 = \inf_I^S(\Psi^c, X_2)$ and $\Phi_i = \inf_I^S(\emptyset, X_i)$ for $i > 2$.
\end{theorem} 

\begin{proof}
    We note that $\Psi \neq \emptyset$, since otherwise $\supp \Phi_1 = I \subsetneq S$, 
    and $\Phi_1$ would not contain the highest root of $R$. We therefore split our argument into the two remaining cases:

    \begin{enumerate}[(i)]
        \item $\Psi = (R / I)^+$:
        In this case the result is clear, since $\Phi_i \subseteq \Phi_1^c = \inf_I^S(\emptyset, X_1^c)$ for each $i \geq 2$.

        \item $\Psi$ is primitive:
        Note that $\inf_I^S(\Psi^c, X_1^c) = \Phi_1^c = \Phi_2 \sqcup \Phi_3 \sqcup \dots \sqcup \Phi_k$, 
        so by Corollary~\ref{corollary: decompositions of canonical I inflations are I inflations}, 
        $I \in \Gen(\Phi_i)$ for each $i \geq 2$. Writing $\Phi_i = \inf_I^S(\Xi_i, X_i)$, 
        we find that $\Psi^c = \Xi_2 \sqcup \Xi_3 \sqcup \dots \sqcup \Xi_k$. Since $\Psi^c$ is 
        primitive and therefore irreducible by Corollary~\ref{cor: primitive implies irreducible}, 
        all but one $\Xi_i$ is empty. Labelling the non-empty set $\Xi_2$, we find that 
        $\Phi_2 = \inf_I^S(\Psi^c, X_2)$ and $\Phi_i = \inf_I^S(\emptyset, X_i)$ for $i > 2$. \qedhere
    \end{enumerate}
\end{proof}

\begin{corollary} \label{cor: support in decompositions}
Let, as in Theorem~\ref{theorem: main theorem}, $R$ be an irreducible QRS and 
    \[R^+ = \Phi_1 \sqcup \Phi_2 \sqcup \dots \sqcup \Phi_k \ .\] 
Then $\supp \Phi_i = S$ for either one or two values of $i$. More precisely, 
$\supp \Phi_i = S$ for a unique $i$ if one of the sets $\Phi_1, \dots, \Phi_k$ is inflated from $(R/I)^+$ and 
$\supp \Phi_i = S$ for two values of $i$ if two of the sets $\Phi_1, \dots, \Phi_k$ are inflated from
(complementary) primitive inversion sets in $(R/I)^+$. \hfill $\square$
\end{corollary}

\begin{remark} \label{rem: symmetric and non-symmetric}
Theorem~\ref{theorem: main theorem} allows us to describe the decompositions of a given inversion set 
$\Phi \subseteq R^+$ into a disjoint union of inversion sets. 
Let $\Phi \subseteq R^+$, $\Phi = \inf_I^S(\Psi, X)$ when expressed canonically. 
We want to describe all decompositions 

\begin{equation} \label{eq: decomposition in QRS}
\Phi = \Phi_1 \sqcup \Phi_2 \sqcup \dots \sqcup \Phi_k \ . 
\end{equation}
First we note that, by Corollary~\ref{corollary: decompositions of canonical I inflations are I inflations},
$\Phi_i = \inf_I^S(\Psi_i, X_i)$ for every $1 \leq i \leq k$. We consider three cases depending on $\Psi$:

\begin{enumerate}
\item[(i)] If $\Psi = \emptyset$, then $\Psi_i = \emptyset$ and \eqref{eq: decomposition in QRS} is equivalent
to decomposing $X$ as  $X = X_1 \sqcup X_2 \sqcup \dots \sqcup X_k$.

\item[(ii)] If $\Psi$ is primitive, then \eqref{eq: decomposition in QRS} is equivalent to
\[R^+ = \Phi^c \sqcup \Phi_1 \sqcup \dots \sqcup \Phi_k \]
and Corollary~\ref{cor: support in decompositions} implies that, after relabeling, we have $\Psi_1 = \Psi$,
$\Psi_i = \emptyset$ for $2 \leq i \leq k$ and $X = X_1 \sqcup X_2\sqcup\dots \sqcup X_k$. 

\item[(iii)] If $\Psi = (R/I)^+$, then \eqref{eq: decomposition in QRS} is equivalent to
$(R/I)^+ = \Psi_1 \sqcup \Psi_2 \sqcup \dots \sqcup \Psi_k$ and $X = X_1 \sqcup X_2 \sqcup \dots \sqcup X_k$.
\end{enumerate}
Note that the decomposition $X = X_1 \sqcup X_2 \sqcup \dots \sqcup X_k$ above is in the proper subsystem $R_I$ of $R$.
The decomposition $(R/I)^+ = \Psi_1 \sqcup \Psi_2 \sqcup \dots \sqcup \Psi_k$ in case (iii) is in the root system $R/I$
which is of lower rank than $R$ except for the case when $I = \emptyset$. The latter is exactly the case treated
by Theorem~\ref{theorem: main theorem}. Thus, in view of Theorem~\ref{theorem: main theorem}, 
all decompositions \eqref{eq: decomposition in QRS}
can be described recursively in terms of such decompositions in quotients and subsystems of $R$.
\end{remark}

\subsection{A basis of $E$ related to a decomposition \eqref{eq: decomposition in QRS}}
Let $\Delta$ be a root system in the Euclidean space $E$  and $w \in W$. Set 
\[S(w) := \{\alpha \in \Delta^+ \, | \, -w(\alpha) {\text { is a simple root}}\} \ .\]
In the case when $\Delta$ is of type $\mathbb A$, it was proved in \cite[Proposition 8.1]{DDMRWW} that, if 
\[\Delta^+ = \Phi(w_1) \sqcup \Phi(w_2) \sqcup \Phi(w_3) \ ,\]
then $S(w_1), S(w_2), S(w_3)$ are disjoint and their union is a basis of $E$. This result was then used to
show that the corresponding face of the Littlewood-Richardson cone is simplicial and to provide
an efficient algorithm for determining its generating rays. Next we will prove that the analog of
\cite[Proposition 8.1]{DDMRWW} holds true for every root system. It is not difficult then to show that
this analog implies the analogs of the results derived as consequences of \cite[Proposition 8.1]{DDMRWW}.
We will not discuss these application here as the interested reader will be able to fill in the details 
following the exposition of \cite{DDMRWW}. 

\np Let $R$ be a QRS and let $\Phi \subseteq R^+$ be an inversion set. To define the set 
$S(\Phi)$ we note that $\Phi = R^+ \cap P$ for 
a (unique) positive system $P \subseteq R$. Set
\[S(\Phi) := \{ \alpha \in \Phi \, | \, \alpha {\text { is a simple root of }} P \} \ .\]
If $R$ is a root system and $w \in W$, then $S(w) = S(\Phi(w))$ because $\Phi(w) = R^+ \cap w^{-1}(R^-)$.

\begin{proposition} \label{prop: basis from decomposition}
Let $R^+ = \Phi_1 \sqcup \Phi_2 \sqcup \dots \sqcup \Phi_k$. Then $S(\Phi_1), S(\Phi_2), \dots,
S(\Phi_k)$ are disjoint sets whose union is a basis of $E$.
\end{proposition}

\begin{proof}
Since $\Phi_1, \dots, \Phi_k$ are disjoint and $S(\Phi_i) \subseteq \Phi_i$, the sets $S(\Phi_1), \dots, S(\Phi_k)$
are also disjoint. We prove that their union is a basis of $E$ by induction on the rank of $R$. If $\rk R = 1$
the statement is obvious. For the rest of the proof we assume that the statement is true 
for every QRS of rank less than $r$ and consider a QRS $R$ of rank $r$.

\np 
Let $R^+ = \Phi_1 \sqcup \Phi_2 \sqcup \dots \sqcup \Phi_k$ where
$\Phi_i = \inf_I^S(\Psi_i, X_i)$ as in Theorem~\ref{theorem: main theorem} and let $W:= \Span I$.
For every $1 \leq i \leq k$ set
\[
S'(\Phi_i):= S(\Phi_i) \cap W \quad {\text{and}} 
\quad S''(\Phi_i) = S(\Phi_i) \backslash S'(\Phi_i) \ .
\]
Since $R_I^+ = X_1 \sqcup X_2 \sqcup \dots \sqcup X_k$, we have 
$X_i = (R^+ \cap W)\cap P = R_I^+ \cap (P \cap W)$ and hence $S(X_i) = S'(\Phi_i)$. 
Thus, by induction,
$S'(\Phi_1) \cup \dots \cup S'(\Phi_k)$ is a basis of $W$.

\np 
Next we consider the two possible cases for $\Phi_1$:
\begin{enumerate}
\item[(i)] $\Psi_1 = (R/I)^+$:

\np 
Then $\Psi_i = \emptyset$ for $i > 1$ and hence $S''(\Phi_i) = \emptyset$ for $i>1$.
Let $P$ be the positive system of $R$ for which $\Phi_1 = R^+ \cap P$. Label the simple roots of $P$
as $\theta_1', \dots, \theta_n'$ so that 
$\theta_1', \dots, \theta_m'$ belong to $W$ and $\theta_{m+1}', \dots, \theta_n'$ do not belong to $W$. 
Proposition~\ref{prop: Gen(Phi) from hyperplane} implies that $W$ is $P$-Levi subspace of $E$, which is 
equivalent to saying that $\theta_1', \dots, \theta_m'$ is a basis of $W$. Moreover, 
$P \backslash W \subseteq R^+$. Indeed, if $\alpha \in P\backslash W$ belongs to $R^-$, then
$-\alpha \in R^+ \backslash W$ belongs to $\Phi_1$ which would imply that $-\alpha \in P$,
contradicting the assumption that $\alpha \in P$.
In particular, $\theta_j' \in P$ for $m+1 \leq j \leq r$, proving that
$S''(\Phi_1) = \{\theta_{m+1}', \dots, \theta_r'\}$.  Hence
\[
S(\Phi_1) \cup \dots \cup S(\Phi_k) = S''(\Phi_1) \cup (S'(\Phi_1) \cup \dots \cup S'(\Phi_k))
=\{\theta_{m+1}', \dots, \theta_r'\} \cup (S'(\Phi_1) \cup \dots \cup S'(\Phi_k)) \ .
\]
Noting that both $S'(\Phi_1) \cup \dots \cup S'(\Phi_k)$ and $\theta_1', \dots, \theta_m'$ are bases of $W$
and $\theta_1', \dots, \theta_r'$ is a basis of $E$, we conclude that $S(\Phi_1) \cup \dots \cup S(\Phi_k)$
is a basis of $E$.

\item[] 

\item[(ii)] $\Psi_1$ is primitive:

\np
Then $\Psi_2 = \Psi_1^c$ and $\Psi_i = \emptyset$ for $i>2$.
Let $P_1$ and $P_2$ be the positive systems in $R$ such that $\Phi_1 = R^+ \cap P_1$ and $\Phi_2 = R^+ \cap P_2$.
By Proposition~\ref{prop: Gen(Phi) from hyperplane}, $W$ is both $P_1$-Levi and $P_2$-Levi subspace of $E$.
Set $\bar{P}_i := \pi_I(P_i)$ for $i = 1,2$. Then $\Psi_i = (R/I)^+ \cap \bar{P}_i$. 
Since $\Psi_1 \sqcup \{0\} \sqcup \Psi_2 = (R/I)^+$, we have
\[
\Psi_2 = \Psi_1^c = ((R/I)^+ \cap \bar{P}_1)^c = (R/I)^+ \cap (-\bar{P}_1)
\]
and hence $\bar{P}_2 = - \bar{P}_1$. Label the bases of
$\bar{P}_1$ and $\bar{P}_2$  as $\bar{\theta}_{m+1}', \dots, \bar{\theta}_{r}'$ and 
$\bar{\theta}_{m+1}'', \dots, \bar{\theta}_{r}''$ respectively so that $\bar{\theta}_{j}'' = - \bar{\theta}_{j}'$
for $m+1 \leq j < r$. Finally, label the bases of $P_1$ and $P_2$ as $\{\theta_1', \dots, \theta_r'\}$ and
$\{\theta_1'', \dots, \theta_r''\}$ so that $\theta_j', \theta_j'' \in W$ for $1 \leq j \leq m$ and
$\pi_I(\theta_j') = \bar{\theta}_j'$ and $\pi_I(\theta_j'') = \bar{\theta}_j''$ for $m+1 \leq j \leq r$.

\np
Since $S''(\Phi_1) = \{\theta_{m+1}', \dots, \theta_r'\} \cap R^+$ and 
$S''(\Phi_2) = \{\theta_{m+1}'', \dots, \theta_r''\} \cap R^+$, we have 
$\pi_I(S''(\Phi_1)) = \{\bar{\theta}_{m+1}', \dots, \bar{\theta}_r'\} \cap (R/I)^+$ and
$\pi_I(S''(\Phi_2)) = \{\bar{\theta}_{m+1}'', \dots, \bar{\theta}_r''\} \cap (R/I)^+$ because $W$ is 
is both a $P_1$-Levi and a $P_2$-Levi subspace of $E$. Using the fact that $\bar{\theta}_{j}'' = - \bar{\theta}_{j}'$,
we conclude that $\pi_I(S''(\Phi_1))$ and $\pi_I(S''(\Phi_2))$ are disjoint and 

\begin{equation} \label{eq:8.77}
\pi_I(S''(\Phi_1) \cup S''(\Phi_2)) = \pi_I(S''(\Phi_1)) \cup \pi_I(S''(\Phi_2)) = 
\{\vep_{m+1} \bar{\theta}_{m+1}', \dots, \vep_n \bar{\theta}_{r}'\} \ ,
\end{equation}
where $\vep_j = \pm 1$ depending on whether $\bar{\theta}_{j}'$ belongs to $(R/I)^+$ or $(R/I)^-$.
The argument in (i) above shows that $\{\bar{\theta}_{m+1}', \dots, \bar{\theta}_{r}'\}$
is a basis of $E/I$ and thus by \eqref{eq:8.77} so is $\pi_I(S''(\Phi_1) \cup S''(\Phi_2))$.

\np
Finally, since
\[
S(\Phi_1) \cup \dots \cup S(\Phi_k) = (S''(\Phi_1) \cup S''(\Phi_2)) \cup (S'(\Phi_1) \cup \dots \cup S'(\Phi_k)) \ ,
\]
$S'(\Phi_1) \cup \dots \cup S'(\Phi_k)$ is a basis of $W$, and $\pi_I(S''(\Phi_1) \cup S''(\Phi_2))$ is a
basis of $E/I$, we conclude that $S(\Phi_1) \cup \dots \cup S(\Phi_k)$ is a basis of $E$. \qedhere
\end{enumerate}
\end{proof}

\np 
\begin{remark} \label{rem:paper on dimensions}
Proposition~\ref{prop: basis from decomposition} implies \cite[Conjecture 1]{HP} for
any finite Weyl group. 
\end{remark}

\subsection{A theorem of Francone and Ressayre} \label{sec:Ressayre}
In \cite{FR} it is proved that all nonzero coefficients in the Belkale-Kumar product on $G/B$ equal one. 
An important step in the proof of the main theorem is a statement about inversion sets, \cite[Theorem 2]{FR}.
In a previous version the authors provided a long case-by-case proof of \cite[Theorem 2]{FR}. In the latest
version, they provide a shorter proof based on the results of the previously posted version of this paper. 
Below we provide two short proofs of a statement which implies \cite[Theorem 2]{FR}. 

\begin{proposition} \label{prop:Ressayre}
Let $\Phi_1, \Phi_2, \Phi_3$ be inversion sets in a QRS $R$ such that
\[
R^+ = \Phi_1 \sqcup \Phi_2 \sqcup \Phi_3 \ . 
\]
Assume that $\beta, \gamma \in R^+$ be such that  $\beta, \beta + \gamma \in \Phi_1$ and $\gamma \in \Phi_2$.
Then $[\beta, \gamma] \cap \Phi_3 = \emptyset$.
\end{proposition}

\begin{proof}[Proof using $\GPhi$.]
Without loss of generality we assume that $R$ is irreducible.
Let $\Phi = \Phi_1 \sqcup \Phi_3$. Then $\beta$ and $\beta + \gamma$ belong to the same component $C$ of $\GPhi$.
Assume, by way of contradiction, that there is $\alpha \in [\beta, \gamma] \cap \Phi_3$ and denote by $C'$ the component of
$\GPhi$ containing $\alpha$. Then
\[\beta \leq \alpha \leq \gamma \leq \beta + \gamma\]
implies that $C \leq C' \leq C$. Hence $C = C'$ which is impossible because $\Phi = \Phi_1 \sqcup \Phi_3$
and $C \subseteq \Phi_1$ while $C' \subseteq \Phi_3$.
\end{proof}

\begin{proof}[Proof using Theorem~\ref{theorem: main theorem}.]
Induction on the rank of $R$, the case $\rk R = 1$ being obvious. Assume $\rk R > 1$.
By Corollary~\ref{cor: support in decompositions} there is a proper subsystem $R_I$ such that one of the sets 
$\Phi_1, \Phi_2, \Phi_3$ is contained in $R_I$. If both $\beta$ and $\gamma$ are contained in $R_I$, then 
the statement follows by induction because the statement reduces to a statement about elements of $R_I$. 
For the rest of the proof we assume that not both $\beta$ and $\gamma$ are contained in $R_I$. 

\np 
If $\Phi_1 \subseteq R_I$ or $\Phi_2 \subseteq R_I$, then $\gamma \in R_I$ and either $\beta \in R_I$ or
$[\beta, \gamma] = \emptyset$ and we are done. If $\Phi_3 \subseteq R_I$ and 
$[\beta, \gamma] \cap \Phi_3 \neq \emptyset$, then $\beta \in R_I$ and , according to the assumption, 
$\gamma \not \in \Delta_I$. Then 
\[\pi_I (\gamma) = \pi_I(\beta + \gamma) \neq 0 \ ,\]
which is a contradiction with the assumption that $\beta + \gamma \in \Phi_1$ but $\gamma \in \Phi_2$.
\end{proof}

\section{Enumerative results} \label{sec: enumerate}

\np 
In this section we show how the methods and results developed so far can be applied to obtain various statistics 
related to inversion sets in root systems. 
For example, one can ask how many primitive (inversions) sets does a given root system (or a given QRS) have. 
In the cases of root systems of types $\mathbb{A}$, $\mathbb{B}$, and $\mathbb{C}$, functional equations for 
the generating functions of the number of primitive inversion sets are known, \cite{AAK} and \cite{DDMRWW}. 
However, these equations do not yield explicit formulas for the their numbers. On the other hand, an estimate 
in the case of root systems of type $\mathbb{A}$ is known. 

\np
In Sections~\ref{sec:9.1} -- \ref{sec:9.3} we discuss three enumerative problems: 
The first one concerns the number of {\it fine decompositions} of a root system, i.e., decompositions in which 
every inversion set contains a single simple root, see below.
The second one discusses the number of primitive inversion sets of a root system.
The third one discusses the number of inversion sets which occur in fine decompositions. 
These are just a couple of examples of integral sequences arising from inversion sets. 
We have chosen to include them because they illustrate the applicability of our methods and results and 
because some of the sequences already appear in the OEIS in different contexts. In particular, we propose
an analog of Catalan numbers related to root systems of type $\mathbb{D}$ (the analog related to root 
systems of types $\mathbb{B}$ and $\mathbb{C}$ already appeared in \cite{DDMRWW}).

\subsection{Fine decompositions} \label{sec:9.1}
Recall that a decomposition $R^+ = \Phi_1 \sqcup \Phi_2 \sqcup \dots \sqcup \Phi_n$ of the 
positive roots of a QRS $R$ is said to be fine if $n = \rk R$ and each inversion set is non-empty.

\np 
The aim of this section is to use Theorem~\ref{theorem: main theorem} to count the number of fine 
decompositions of the root systems. 
The paper \cite{DDMRWW} derived formulae for the number of fine decompositions of $\mathbb A_n$ 
and $\mathbb B_n$ but not for $\mathbb D_n$. 
With the hindsight of the more natural definition of inflation we 
see that part of the difficulty arises from the fact that there are multiple types of rank 2 
quotients of $\mathbb D_n$, which have different numbers of primitive inversion sets.  
As we will see, Theorem~\ref{theorem: main theorem} allows us to treat each root system with a 
uniform approach. We begin by introducing a few relevant propositions.

\begin{prop}
    \label{prop: basic properties of fine decompositions}
    Let $R^+ = \Phi_1 \sqcup \Phi_2 \sqcup \dots \sqcup \Phi_n$ be a fine decomposition of a QRS $R$ with base $S$. Then
  $|\Phi_i \cap S| = 1$ for $1\leq i \leq n$.
Moreover, there exists $I \subseteq S$ such that either
    \begin{enumerate}[(i)]
    \item $\rk (R/I)=1$ with
    $\Phi_1 = \inf_I^S((R/I)^+, \emptyset)$ and 
           $\Phi_i = X_i = \inf_I^S(\emptyset,X_i)$ for $2 \leq i \leq n$; 
           \item[]or
    \item $\rk(R / I) = 2$ with  
    $\Phi_1 = \inf_I^S(\Psi, \emptyset)$ and 
    $\Phi_2 = \inf_I^S(\Psi^c, \emptyset)$ for some primitive $\Psi \subseteq (R/I)^+$ and 
    $\Phi_i = X_i = \inf_I^S(\emptyset,X_i)$ for $3 \leq i \leq n$.
     \end{enumerate}
     Furthermore, the set $I$ and the sets $\Psi$ and $X_i$ are unique up to reordering.
\end{prop}

\begin{proof}
    By co-closure, $|\Phi_i \cap S| \geq 1$ for each $i$. The $n$ simple roots of $S$ are divided into $n$ 
    pairwise disjoint inversion sets in such a way that each inversion set contains at least one simple root. 
    This forces $|\Phi_i \cap S| = 1$ for all $i$. 

    \np To prove that $\rk(R / I) \leq 2$, we consider the decompositions 
    $(R/I)^+ = \Psi_1 \sqcup \Psi_2 \sqcup \dots \sqcup \Psi_n$ and 
    $(R_I)^+ = X_1 \sqcup X_2 \sqcup \dots \sqcup X_n$ defined in Theorem~\ref{theorem: main theorem} 
    from $R^+ = \Phi_1 \sqcup \Phi_2 \sqcup \dots \sqcup \Phi_n$.
    Since $|\Phi_i \cap S| = |\Psi_i \cap S/I| + |X_i \cap I|$, we conclude that $|\Psi_i \cap S/I| = 1$ 
    unless $\Psi_i = \emptyset$. By Corollary~\ref{cor: support in decompositions}, exactly
    one or two of the sets $\Psi_1, \dots, \Psi_k$ are non-empty, so we conclude that $\rk(R / I) \leq 2$.
\end{proof}

\np 
Proposition~\ref{prop: basic properties of fine decompositions} allows us to count fine 
decompositions of a QRS $R$ with rank $n$ as follows. We first take $I \subseteq S$ 
where $|I| = n - 1$ or $|I| = n - 2$. If $|I| = n - 1$, we let $\Phi_1 = \inf_I^S((R / I)^+, \emptyset)$ 
and $\Phi_2, \Phi_3, \dots, \Phi_n$ be a fine decomposition of $R_I$. If $|I| = n - 2$, we choose a 
primitive inversion set $\Psi \subseteq  (R / I)^+$ and let $\Phi_1 = \inf_I^S(\Psi, \emptyset)$ 
and $\Phi_2 = \inf_I^S(\Psi^c, \emptyset)$. Once again, we let $\Phi_3, \Phi_4, \dots, \Phi_n$ 
form a fine decomposition of $R_I$. It is clear that in both cases $\Phi_1, \Phi_2, \dots, \Phi_n$ forms a 
fine decomposition of $R$, and by Proposition~\ref{prop: basic properties of fine decompositions}, 
this characterizes all fine decompositions. Introducing some notation, 
we let $S = \{ \theta_1, \theta_2, \dots, \theta_n \}$ be a base of the QRS $R$
and let $F(R)$ denote the number of fine decompositions of $R$.
Then

\begin{equation}
\label{eqn: fine decomposition}
    F(R) = \sum_{i = 1}^n F(R_{\{ \theta_i \}^c}) + \sum_{i = 1}^{n - 1} \sum_{j = i + 1}^n 
    \Pi(R / \{\theta_i, \theta_j \}^c)F(R_{\{\theta_i, \theta_j \}^c})
\end{equation}
where $\Pi(R) := \dfrac12 \left|\{ \Psi \subseteq  R^+ \mid \Psi \text{ is a primitive inversion set}\}\right|$. 
We introduce one more proposition before counting the fine decompositions of all root systems.

\np 
Recall that if $R_1, R_2$ are QRSs with ambient Euclidean spaces $E_1, E_2$, then we have a 
QRS $R_1 \times R_2$ with ambient Euclidean space 
${\mathbb E}_1 \oplus E_2$. 
Its roots are identified with $R_1 \cup R_2$. It is realized as a quotient of the root 
system $\Delta_1 \times \Delta_2$ where $R_i$ is a quotient of $\Delta_i$ for $i=1,2$.

\np 
In the following, we write $\mathbb{A}_0 = \mathbb{B}_0 = \mathbb{D}_0$ to denote the rank zero 
root system with trivial ambient Euclidean space. 
We also use the convention that 
$\mathbb{D}_1 = \mathbb{B}_1 = \mathbb{A}_1$, 
$\mathbb{D}_2 = \mathbb{A}_1 \times \mathbb{A}_1$, 
and $\mathbb{D}_3=\mathbb{A}_3$. The proof of the following follows from the remark in the 
paragraph preceding Theorem~\ref{theorem: main theorem}. 

\begin{prop}
    Suppose $R_1$ and $R_2$ are QRSs. Then $F(R_1 \times R_2) = F(R_1) F(R_2)$.
    \hfill\qedsymbol
\end{prop}

\subsubsection{Type \texorpdfstring{$\mathbb A_n$}{A root systems}.} 

We take some $n \geq 1$ and consider the fine decompositions of the 
root system $\mathbb A_n$. To do this, we first consider the quotients and subsystems of $\mathbb A_n$. 
It is not hard to see that all quotients of $\mathbb A_n$ are also of type $\mathbb A_k$ for some $k \leq n$. 
Moreover, if $I = \{ \theta_{i_1}, \theta_{i_2}, \dots, \theta_{i_m} \}$ are simple roots for some 
$1 \leq i_1 < i_2 < \dots < i_m \leq n$, then 
$(\mathbb A_n)_{I^c} = \mathbb A_{i_1 - 1} \times \mathbb A_{i_2 - i_1 - 1} \times \dots \times \mathbb A_{n - i_m}$. 
Defining $a_n := F(\mathbb A_n)$ to be the number of fine decompositions of $\mathbb A_n$, we apply 
Equation \eqref{eqn: fine decomposition} to find that
        $$a_n = \sum_{i = 1}^n a_{i - 1}a_{n - i} + \sum_{i = 1}^{n - 1}\sum_{j = i + 1}^n 
        \Pi(\mathbb A_2) a_{i - 1}a_{j - i - 1}a_{n - j}$$
        It is quickly verified that there are no primitive inversion sets in $\mathbb A_2$, so $\Pi(\mathbb A_2) = 0$. Thus,
        $$a_n = \sum_{i = 1}^n a_{i - 1}a_{n - i}$$
        along with the fact that $a_0 = 1$, this recurrence relation implies that $a_n = \frac{1}{n+1}\binom{2n}{n}$, 
        the $n$-th Catalan number. As is well-known, the corresponding generating function 
        $A(z):= \sum_{n=0}^\infty a_n z^n$ of the sequence of Catalan numbers is
\begin{equation*}
    A(z) = \frac{1 - \sqrt{1-4z}} {2z} \ .
\end{equation*}

\subsubsection{Types \texorpdfstring{$\mathbb B_n$ and $\mathbb C_n$}{B and C root systems}.}
Since these two types produce the same results, we consider only type $\mathbb B_n$ and 
proceed similarly to above. However, unlike type $\mathbb A$, the 
quotients of root systems of type $\mathbb B$ are not necessarily of type $\mathbb B$. However, the indivisible 
elements (i.e., the roots $\alpha$ for which no fractional multiple of $\alpha$ is also a root) in a quotient of a root system of type 
$\mathbb B$ form a root system of type $\mathbb B$ (see \cite{DF}). 
Thus Remark~\ref{rem: inversion Sets lie between two hyperplanes} implies that counting the inversion sets
    (and the respective decompositions)
    in a quotient of a root system of type $\mathbb B$ is equivalent to counting the inversion sets (and the respective decompositions)
    in the corresponding root system of type $\mathbb B$.

    \np 
    Taking $n \geq 1$, if $I = \{ \theta_{i_1}, \theta_{i_2}, \dots, \theta_{i_m} \}$ are simple 
    roots for some $1 \leq i_1 < i_2 < \dots < i_m \leq n$, 
    then $(\mathbb B_n)_{I^c} \cong \mathbb A_{i_1 - 1} \times \mathbb A_{i_2 - i_1 - 1} \times \dots \times \mathbb B_{n - i_m}$.
    Unlike in $\mathbb A_2$, there are two primitive inversion sets in $\mathbb B_2$, so $\Pi(\mathbb B_2) = 1$. 
    Once again applying Equation \eqref{eqn: fine decomposition}, we define $b_n := F(\mathbb B_n)$ to be the 
    number of fine decompositions of $\mathbb B_n$ and find
     \begin{align*}
        b_n &= \sum_{i = 1}^n a_{i - 1}b_{n - i} + \sum_{i = 1}^{n - 1}\sum_{j = i + 1}^n a_{i - 1}a_{j - i - 1}b_{n - j} 
        = \sum_{i = 1}^n a_{i - 1}b_{n - i} + \sum_{\substack{k + \ell + m = n - 2 \\ k, \ell, m \geq 0}} a_k a_\ell b_m \\
        &= \sum_{i = 1}^n a_{i - 1}b_{n - i} + \sum_{i = 2}^n\sum_{j = 1}^{i - 1} a_{i - j - 1}a_{j - 1}b_{n - i} 
        = \sum_{i = 1}^n a_{i - 1}b_{n - i} + \sum_{i = 2}^n b_{n - i} \sum_{j = 1}^{i - 1} a_{i - j - 1}a_{j - 1} \\
        &= \sum_{i = 1}^n a_{i - 1}b_{n - i} + \sum_{i = 2}^n b_{n - i}a_{i - 1} 
        = b_{n - 1} + 2\sum_{i = 2}^n a_{i - 1}b_{n - i}
    \end{align*}
where the second-to-last equality follows from the recurrence relation
$a_{i-1} = \sum_{j = 1}^{i - 1} a_{i - j - 1}a_{j - 1}$ from the previous subsection. The first few values of 
the sequence $(b_n)_{n \geq 1}$ are as follows: $1, 3, 9, 29, 97, 333, 1165$.  
The generating function  $B(z) := \sum_{n=0}^\infty b_n z^n$ of this sequence is
\begin{equation*}
    B(z) = \frac{1} { z + \sqrt{1-4z}}
\end{equation*}
as shown in \cite{DDMRWW}.

 \subsubsection{Type \texorpdfstring{$\mathbb D_n$}{ D root systems}.} 
    Let $I := \{\theta_i,\theta_j\}^c$, where $i < j$. Then for $i,j \in \{1, n-1, n\}$, 
    the QRS is of type $\mathbb A_2$, which has no primitive inversion sets. 
    Whereas in the other cases, the quotient has the same number of primitive inversion sets as $\mathbb B_2$.
    The simple roots in $I$ span a root system isomorphic to
    \begin{equation*}
        \mathbb D_n^{(i,j)} := \begin{cases}
            \mathbb A_{i-1} \times \mathbb A_{j-i-1} \times \mathbb D_{n-j} &\textnormal{ if $j < n-1$} \\
            \mathbb A_{i-1} \times \mathbb A_{n-i-1}                        &\textnormal{ if $j = n-1$ or $j=n$}
        \end{cases}
    \end{equation*}
    When $I = \{\theta_i\}^c$, the roots in $I$ span a root system isomorphic to
    \begin{equation*}
        \mathbb D_n^{(i)} := \begin{cases}
            \mathbb A_{i-1} \times \mathbb D_{n-i} &\textnormal{ if $i < n-1$} \\
            \mathbb A_{n-1}                        &\textnormal{ if $i = n-1$ or $i=n$}
        \end{cases}
    \end{equation*}
    
    \np {\allowdisplaybreaks  
    We now let $d_n := F(\mathbb D_n)$ for $n \geq 4$
    and $d_0=d_1=d_2=1$ and $d_3=a_3=5$.  Then for $n\geq 3$ we have
    \begin{align*}
        d_n =& \sum_{i=1}^{n} F(\mathbb D_n^{(i)}) 
        + \sum_{i=1}^{n-1}\sum_{j=i+1}^n \Pi(\mathbb D_n/\{\theta_i, \theta_j \}^c) F(\mathbb D_n^{(i,j)})  \\ 
        =& \sum_{i=1}^{n-2} a_{i-1} d_{n-i} + 2 a_{n-1} 
        + \sum_{i=1}^{n-3}\sum_{j=i+1}^{n-2} \Pi(\mathbb D_n/\{\theta_i, \theta_j \}^c) F(\mathbb D_n^{(i,j)}) \\
        &+ \sum_{i=1}^{n-2}\Pi(\mathbb D_n/\{\theta_i, \theta_{n-1} \}^c) F(\mathbb D_n^{(i,n-1)}) 
        + \sum_{i=1}^{n-1}\Pi(\mathbb D_n/\{\theta_i, \theta_n \}^c) F(\mathbb D_n^{(i,n)}) \\
        =& \sum_{i=1}^{n-2} a_{i-1} d_{n-i} + 2 a_{n-1} 
        + \sum_{i=1}^{n-3}\sum_{j=i+1}^{n-2} a_{i-1}a_{j-i-1}d_{n-j} \\
        &+ \sum_{i=2}^{n-2}\Pi(\mathbb D_n/\{\theta_i, \theta_{n-1} \}^c) F(\mathbb D_n^{(i,n-1)}) 
        + \sum_{i=2}^{n-2}\Pi(\mathbb D_n/\{\theta_i, \theta_n \}^c) F(\mathbb D_n^{(i,n)}) \\
        =& \sum_{i=1}^{n-2} a_{i-1} d_{n-i} + 2 a_{n-1} 
        + \sum_{i=1}^{n-3}\sum_{j=i+1}^{n-2} a_{i-1}a_{j-i-1}d_{n-j}
        + \sum_{i=2}^{n-2} a_{i-1}a_{n-i-1} 
        + \sum_{i=2}^{n-2} a_{i-1}a_{n-i-1} \ .
\end{align*}
} 

Now

\begin{align*} 
 \sum_{i=1}^{n-3} &\sum_{j=i+1}^{n-2} a_{i-1} a_{j-i-1}d_{n-j}
= \sum_{j=2}^{n-2} \sum_{i=1}^{j-1} a_{i-1} a_{j-i-1}d_{n-j}
= \sum_{j=2}^{n-2} \sum_{i=0}^{j-2} a_{i} a_{j-i-2}d_{n-j}\\
&= \sum_{j=0}^{n-4} \sum_{i=0}^{j} a_{i} a_{j-i}d_{n-j-2}
= \sum_{j=0}^{n-4} \left(\sum_{i=0}^{j} a_{i} a_{j-i}\right) d_{n-j-2}
= \sum_{j=0}^{n-4} a_{j+1} d_{n-j-2}\ .
\end{align*}

\np 
Therefore 
   \begin{align*}
        d_n =& \sum_{i=1}^{n-2} a_{i-1} d_{n-i} + 2 a_{n-1} 
        + \sum_{i=1}^{n-3}\sum_{j=i+1}^{n-2} a_{i-1}a_{j-i-1}d_{n-j} 
        + \sum_{i=2}^{n-2} a_{i-1}a_{n-i-1} 
        + \sum_{i=2}^{n-2} a_{i-1}a_{n-i-1} \\
=& a_0 d_{n-1} + \sum_{i=2}^{n-2} a_{i-1} d_{n-i} + 2 a_{n-1} 
+  \sum_{j=0}^{n-4} a_{j+1} d_{n-j-2}
        + 2\sum_{i=2}^{n-2} a_{i-1}a_{n-i-1} \\
=& d_{n-1} + \sum_{i=0}^{n-4} a_{i+1} d_{n-i-2}  + 2 a_{n-1} 
        +  \sum_{j=0}^{n-4} a_{j+1} d_{n-j-2}
        + 2\sum_{i=2}^{n-2} a_{i-1}a_{n-i-1} \\
=&  d_{n-1} + 2\sum_{i=0}^{n-4} a_{i+1} d_{n-i-2} + 2 a_{n-1} 
        + 2 \sum_{i=1}^{n-3} a_{i} a_{n-i-2}\\
=& d_{n-1} -2a_{n-2} d_1 -2 a_{n-1} d_0 + 2\sum_{i=1}^{n-1} a_{i} d_{n-i-1} + 2 a_{n-1}
         -2a_0 a_{n-2} -2a_{n-2}a_0 + 2 \sum_{i=0}^{n-2} a_{i} a_{n-i-2}\\
=& d_{n-1} + 2\sum_{i=0}^{n-2} a_{i+1} d_{n-i-2} -6a_{n-2} + 2 \sum_{i=0}^{n-2} a_{i} a_{n-i-2}\\
=& d_{n-1} + 2\sum_{i=0}^{n-2} a_{i+1} d_{n-i-2} + 2a_{n-1} -6a_{n-2}\ .
\end{align*}

\np 
The first few values of the sequence $(d_n)_{n \geq 1}$ are: 1, 1, 5, 19, 69, 249.

\np 
The generating function $D(z) := \sum_{n=0}^\infty d_n z^n$ of this sequence is
\begin{equation*}
    D(z) = \frac{ 2z^2 - 5z + 2 + (3z-1)\sqrt{ 1-4z}} { z + \sqrt{ 1-4z}}\ .
\end{equation*}

\subsubsection{Exceptional types.}
For $\mathbb{G}_2$ there are 5 fine decompositions.
For $\mathbb{F}_4$ there are 46 fine decompositions.
Computer constructions show that the root systems
$\mathbb{E}_6$, $\mathbb{E}_7$, and $\mathbb{E}_8$ have 320, 1534, and 8932 
fine decompositions respectively.

\subsection{Primitive inversion sets} \label{sec:9.2}

\np 
In Table 1 are some data on the number of primitive inversion sets.  Albert, Atkinson and Klazar 
gave a generating function, defined recursively for the $\mathbb{A}_n$ case (see \cite{AAK}).

\np
We wish to count the number of primitive inversion sets in 
some QRS $R$.  Using the canonical form of an inflation we may
relate the number of inversion sets in $R$ to the number
of primitive inversion sets in a quotient of $R$.

\np
Given a QRS $R$ we write $\Inv(R)$ to denote the number of inversion sets in $R$, $\Pr(R)$ to denote the number of
primitive inversion sets in $R$ and $\D(R)$ to denote the number of 
inversions sets in $R$ of full support. (Note that in Section~\ref{sec:9.1} above we 
introduced the notation $\Pi(R)$ to denote one-half of
the number of primitive inversion sets in $R$ i.e., $\Pr(R) = 2\Pi(R)$.) 
Here 
$\Inv(\emptyset)=\D(\emptyset)=1$. 

\np
It is clear that
if $R=R^1 \times R^2$ is reducible then
$\Inv(R)=\Inv(R^1)\Inv(R^2)$ and $\D(R)=\D(R^1)\D(R^2)$.
Thus it will be sufficient to consider irreducible QRSs.

\np
Since every inversion set has full support in exactly one subsystem of R, i.e., 
for exactly one choice of $I$, we have
$$
 \Inv(R) = \sum_{ \emptyset \subseteq I \subseteq S} \D(R_I) \ .
$$

\begin{proposition}
$$
\Inv(R) = \sum_{ \emptyset \subseteq I \subsetneq S}
              \left( \Inv(R_I)\Pr(R/I) + 2\D(R_I) \right) \ .
$$
\end{proposition}
\begin{proof}
   Let $\phi \subseteq R$ be an inversion set and write
$\phi = \inf_I^S(\Psi,X)$ for the canonical form of $\phi$.
If $\psi$ is primitive then $X$ is any inversion set contained 
in $I$.  There are $\Pr(R/I)\Inv(R_I)$ such inversion sets.
Otherwise either $\psi=(R/I)^+$ or $\psi=\emptyset$.  
One of these 
two possibilities occurs for $\phi$ and the other for $\phi^c$.
If $\psi=\emptyset$ then $\supp X = I$ and thus there are
$D(I)$ choices for $X$.  Accounting for 
both $\phi$ and $\phi^c$ we see that there are $2D(I)$ 
inversion sets inflated from $I$ with $\psi$ not primitive.
\end{proof}

\np 
Isolating one term in each of the above equations we have
$$
 \D(R) = \Inv(R)-\sum_{ \emptyset \subseteq I \subsetneq S} \D(R_I)
$$
and
$$ \Pr(R) = \Inv(R)-2-\sum_{ \emptyset \subsetneq I \subsetneq S}
    \left( \Inv(R_I)\Pr(R/I) + 2\D(R_I) \right)\ .
$$

\np 
Using the above two equations we may recursively solve for
$\D(R_I)$ and $\Pr(R/I)$ for every $I$ if we know 
each $\Inv(R_I)$.  The difficulty lies in working with
subsystems of quotients.  Fortunately we have the 
following isomorphism of QRSs which allows us to replace
a subsystem of a quotient by a quotient of a subsystem.

\begin{lemma} Assume that $I$ and $J$ are disjoint subsets of $S$. Then
  $$ (R/I)_{\overline{J}} \cong \left( R_{I \cup J}\right) / I\ ,$$
  where $\overline{J} \subseteq S/I$ is the image of $J$ under the projection $S \twoheadrightarrow S/I$.
\end{lemma}

\begin{proof}
Set $K := (I \cup J)^c$ and let
\[
I = \{\theta_1', \dots, \theta_l'\} \ , \quad J = \{\theta_1'', \dots, \theta_m''\} \ , \quad
K = \{\theta_1''', \dots, \theta_s'''\} \ .
\]
For $\gamma \in R$ we denote the image of $\gamma$ under the surjection $R \twoheadrightarrow R/I$ by $\bar{\gamma}$.
Then
\[
R/I = \{ \bar{\gamma} = \sum_j a_j'' \bar{\theta}_j'' + \sum_k a_k''' \bar{\theta}_k''' \, | \,
{\text { there is }} \gamma = \sum_i a_i' \theta_i' + \sum_j a_j'' \theta_j'' + \sum_k a_k''' \theta_k''' \in R \}
\]
and hence
\[(R/I)_{\overline{J}} = \{ \bar{\gamma} = \sum_j a_j'' \bar{\theta}_j''  \, | \,
{\text { there is }} \gamma = \sum_i a_i' \theta_i' +  \sum_j a_j'' \theta_j'' \in R \} \ .
\]
Similarly,
\[
R_{I\cup J} = \{ \gamma \in R \, | \,  \gamma = \sum_i a_i' \theta_i' + \sum_j a_j'' \theta_j'' \}
\]
and hence 
\[(R_{I \cup J})/I = \{ \bar{\gamma} = \sum_j a_j'' \bar{\theta}_j''  \, | \,
{\text { there is }} \gamma \in R_{I \cup J}, \gamma = \sum_i a_i' \theta_i' +  \sum_j a_j'' \theta_j'' \} \ .
\]
Comparing the expressions above, we conclude that $(R/I)_{\overline{J}} = \left( R_{I \cup J}\right)/I$.
\end{proof}

\np
The references \cite{DF}, \cite{CH}, and \cite{CM} 
together provide $\Inv(R)$ for all root systems. 
Using the values of $\Inv(R)$ and the above equations, we can find all the data in 
Table~\ref{table:number of primitives}.  In fact, with the
exception of $\Pr({\mathbb E}_8)$ and $\D({\mathbb E}_8)$
all the required data may be determined by generating all the
inversion sets for each QRS and performing exhaustive counts.
Results for ${\mathbb E}_7$ took several hours on a Mac Air using the computer algebra system 
Magma (\cite{magma}).  Directly computing this data for ${\mathbb E}_8$ would be several orders of magnitude more difficult.

\np 
In principle one can also write explicitly the recursive formulas computing 
$\Pr(\Delta)$ for every root system $\Delta$. The first remark is that if $R$ is 
a quotient of rank $n$ of a root system other then $\mathbb{A}_n$, then the Weyl groupoid of $R$, cf.~\cite{CH},
is one of $n+1$ Weyl groupoids. Moreover, the number of inversion sets of $R$ depends on its Weyl
groupoid and is known. The difficulty then is that different positive systems in $R$ may have 
different number of primitive inversion sets which makes the notation very cumbersome for
root systems of type $\mathbb{D}$, so we decided to not pursue this direction.

\begin{table}[h]
    \centering
    \begin{tabular}{crrrrrr}
Rank $r$ & $\mathbb{A}_r$ & $\mathbb{B}_r/\mathbb{C}_r$  & $\mathbb{D}_r$ & $\mathbb{E}_r$ & $\mathbb{F}_4$& $\mathbb{G}_2$\\
        1 & 2   &             & \\
        2 & 0  & 2           & & & & 6 \\
        3 & 2 & 10          & \\
        4 & 6    & 90          & 30 & & 514&\\
        5 & 46    & 966         & 366\\
        6 & 338   & 12\,338     & 5018 &16\,058\\  
        7 & 2926  & 181\,470    & 76\,958 & 1\,247\,086 \\ 
        8 & 28146 & 3\,018\,082 & 1\,314\,946 &400\,658\,018\\  
        9 & 298\,526 &55\,995\,486 &24\,856\,542  & \\ 
        \multicolumn{1}{c}\vdots &  \multicolumn{1}{c}\vdots &\multicolumn{1}{c}\vdots &\multicolumn{1}{c}\vdots\\
         \\
    \end{tabular}
    \caption{Number of Primitive Inversion Sets}
    \label{table:number of primitives}
\end{table}

\subsection{Fine inversion sets} \label{sec:9.3}
\np 
Next we study some statistics concerning fine inversion sets, i.e., inversion sets which contain
only one simple root.

\np
Let $\Delta$ be an irreducible root system and 
let $S$ be the fixed base of $\Delta$. 
Let $\Phi = \inf_I^S(\Psi,X)$ be an inversion set. Since
\[|\Phi \cap S| = |\Psi \cap (S/I)| + |X \cap I| \ ,\]
$\Phi$ is fine if and only if one of the sets $\Psi$ and $X$ is fine and the other one is empty. 
Let $S':= \supp \Phi$ and let $\Delta'$ be the subsystem of $\Delta$ generated by $S'$. Since $X$
is empty or fine, we conclude that $\Delta'$ is an irreducible subsystem of $\Delta$.
Moreover, $\Phi \subseteq \Delta'$ and 

\begin{equation}  \label{eq:fine inversion}
\Phi= \inf_J^{S'}(\Psi', \emptyset) \ ,
\end{equation}
where $\Psi'$ is fine and primitive or $\rk (\Delta'/J) = 1$ and $\Psi' = (\Delta'/J)^+$. In short, we have the following.

\begin{proposition} \label{prop: fine inversion sets}
Equation \eqref{eq:fine inversion} establishes a bijection between the fine inversion sets in $\Delta$ and
the pairs $(\Delta'/J, \Psi')$, where $\Delta'/J$ is a subquotient of an irreducible subsystem $\Delta'$ of $\Delta$
and $\Psi' \subseteq (\Delta'/J)^+$ is fine and primitive or $\rk(\Delta'/J) = 1$ and $\Psi' = (\Delta'/J)^+$.
\hfill $\square$
\end{proposition}

\np
Proposition~\ref{prop: fine inversion sets} reduces any problem of counting fine inversion sets in $\Delta$
to counting fine and primitive inversion sets in (irreducible) subquotients of $\Delta$. With the exception of
root systems of type $\mathbb{A}$, see Remark~\ref{rem: fine for A}, the latter problem is difficult.
Restricting our attention to fine inversion sets which occur in a fine decomposition simplifies the problem
because it imposes the additional condition 
\[\rk (\Delta'/J) \leq 2 \]
to the pairs $(\Delta'/J, \Psi')$ in Proposition~\ref{prop: fine inversion sets}, 
cf.~Proposition~\ref{prop: basic properties of fine decompositions}.
Hence, to count the fine inversion sets which occur in a fine decomposition, we need to
understand all rank 2 subquotients of $\Delta$. If $\Delta$ is a root system
and $R$ is an irreducible rank 2 subquotient of $\Delta$, then $R^+$ is one of the following three sets:
\[\{\alpha, \beta, \alpha + \beta\}, \quad 
\{\alpha, \beta, \alpha + \beta, \alpha + 2 \beta\}, \quad 
\{\alpha, \beta, \alpha + \beta, \alpha + 2 \beta, 2 \alpha + 2 \beta\}\ .\]

\np If $\Delta$ is a root system, Propositions~\ref{prop: basic properties of fine decompositions} 
and \ref{prop: fine inversion sets} combined 
with the discussion above, allow us to count the inversions sets that occur in at least one fine decomposition 
of $\Delta^+$ sorted by the respective simple roots they contain as well as the number of such inversion sets which, 
in addition, contain the highest root of $\Delta$.  We leave the details to the interested reader and provide the results below.
The simple roots $\theta_1, \dots, \theta_n$ of $\Delta$ are labeled as in \cite{B}.

\np 
If $\Delta$ is of type $\mathbb{A}_n$, let $a_{n,k}$ denote the number of fine
inversion sets that occur in a fine decomposition and contain the root $\theta_k$ and let 
 $\tilde{a}_{n,k}$ denote the fine inversion set that, in addition, contain the highest root of $\Delta$.
 Then
 \[a_{n,k} = k(n-k+1) \quad \quad {\text{and}} \quad \quad \tilde{a}_{n,k} =1 \ .\]

\np
If $\Delta$ is of type $\mathbb{B}_n$, the analogous numbers $b_{n,k}$ and $\tilde{b}_{n,k}$ are given by
\[b_{n,k} = \frac{1}{2}k(4n-3k+1) \quad \quad {\text{and}} \quad \quad 
\tilde{b}_{n,k} = \left\{\begin{array}{ccl} 1 & {\text{if}} & k = 1 \\  n-k+2 &{\text{if}} & 1 < k \leq n
\ . \end{array} \right.\]

\np
If $\Delta$ is of type $\mathbb{C}_n$, the analogous numbers $c_{n,k}$ and $\tilde{c}_{n,k}$ are given by
\[c_{n,k} = \frac{1}{2}k(4n-3k+1) \quad \quad {\text{and}} \quad \quad 
\tilde{c}_{n,k} = \left\{\begin{array}{ccl} k+1 & {\text{if}} & 1 \leq k <n \\ 1 &{\text{if}} &  k = n
\ . \end{array} \right.\]

\np
If $\Delta$ is of type $\mathbb{D}_n$, the analogous numbers $d_{n,k}$ and $\tilde{d}_{n,k}$ are given by
\[d_{n,k} = \left\{\begin{array}{ccl}
\frac{k(4n-3k+3)-4}{2} & {\text{if}} & 1 \leq k < n-1 \\ &&\\
\frac{n (n-1)}{2} & {\text{if}} & n-1 \leq k \leq n \end{array} \right. 
 \quad {\text{and}} \quad 
\tilde{d}_{n,k} = \left\{\begin{array}{ccl} 1 & {\text{if}} & k = 1, n-1, n \\ &&\\ 
n-k+2 &{\text{if}} & 1 < k < n-1
\ . \end{array} \right.\]

\np 
For the exceptional root systems the data are:\\ \\
$\mathbb{G}_2$: [5, 5], [2, 3]\\ \\
$\mathbb{F}_4$: [13, 20, 20, 13], [6,6,5,3]\\ \\
$\mathbb{E}_6$: $\Big[\Esix{15\ }{22\ }{28\ }{40\ }{28\ }{15}\Big]$, 
$\Big[\Esix{1\ }{7\ }{5\ }{8\ }{5\ }{1}\Big]$\\ \\
$\mathbb{E}_7$: $\Big[\Eseven{29\ }{40\ }{51\ }{69\ }{56\ }{40\ }{22}\Big],
\Big[\Eseven{10\ } {7\ } {11\ } {11\ } {9\ } {6\ } 1 \Big]$\\ \\
$\mathbb{E}_8$: 
$\Big[\Eeight{54\ }{78\ }{88\ }{116\ }{104\ }{83\ }{64\ }{37}\Big],
\Big[ \Eeight{10\ }{13\ }{14\ }{17\ }{18\ }{17\ }{17\ }{16}\Big]$ .\\

\begin{remark} \label{rem:on dicrepancy}
Note that $b_{n,k} = c_{n,k}$ as one can expect since the Weyl groups of the root systems $\mathbb{B}_n$
and $\mathbb{C}_n$ are the same. However, $\tilde{b}_{n,k} \neq \tilde{c}_{n,k}$ because the expressions
of the respective highest roots are different.
\end{remark}

\begin{remark} \label{rem: fine for A}
If $\Delta$ is of type $\mathbb{A}_n$, counting fine inversion sets is simpler because all irreducible 
subquotients of $\Delta$ are of type $\mathbb{A}$ as well. Here are a few results:
\begin{enumerate}
\item[(i)] The number of fine inversion sets in $\mathbb{A}_n$ containing $\theta_k$ is $\binom{n+1}{k} -1$.
\item[(ii)] The number of all fine inversion sets in $\mathbb{A}_n$ is $2^{n+2}-(n+2)$.
\item[(iii)] If $n > 1$ is odd, there is one fine and primitive inversion set in $\mathbb{A}_n$;
if $n$ is even, there are no fine and primitive inversion sets in $\mathbb{A}_n$.
\end{enumerate}
These results can be derived either using Propositions~\ref{prop: basic properties of fine decompositions} 
and \ref{prop: fine inversion sets} or directly by interpreting the elements of the Weyl group of $\mathbb{A}_n$
as permutations of $1, 2, \dots, n+1$, cf.~\cite{DDMRWW}.
\end{remark}

\np 
{\bf {Acknowledgement.}}
 All authors were partially supported by the
Natural Sciences and Engineering Research Council of Canada. In addition, C.P. and D.W.  
were partially supported by the Canadian Defence
Academy Research Programme.

\end{document}